\documentclass[onefignum,onetabnum]{siamonline171218}

\usepackage{lipsum}
\usepackage{amsfonts}
\usepackage{graphicx}
\usepackage{epstopdf}
\usepackage{algorithmic}
\ifpdf
  \DeclareGraphicsExtensions{.eps,.pdf,.png,.jpg}
\else
  \DeclareGraphicsExtensions{.eps}
\fi

\usepackage{enumitem}
\setlist[enumerate]{leftmargin=.5in}
\setlist[itemize]{leftmargin=.5in}

\newsiamremark{remark}{Remark}
\newsiamremark{hypothesis}{Hypothesis}
\crefname{hypothesis}{Hypothesis}{Hypotheses}
\newsiamthm{claim}{Claim}

\headers{Spectral submanifolds of the Navier-Stokes equations}{Gergely Buza}

\title{Spectral submanifolds of the Navier-Stokes equations}
\author{Gergely Buza\thanks{Department of Applied Mathematics and Theoretical Physics, University of Cambridge, Cambridge CB3 0WA, UK
  (\email{gb643@cam.ac.uk}).}
}

\usepackage{amsopn}

\makeatletter
\newcommand*{\addFileDependency}[1]{% argument=file name and extension
  \typeout{(#1)}% latexmk will find this if $recorder=0 (however, in that case, it will ignore #1 if it is a .aux or .pdf file etc and it exists! if it doesn't exist, it will appear in the list of dependents regardless)
  \@addtofilelist{#1}% if you want it to appear in \listfiles, not really necessary and latexmk doesn't use this
  \IfFileExists{#1}{}{\typeout{No file #1.}}% latexmk will find this message if #1 doesn't exist (yet)
}
\makeatother

\ifpdf
\hypersetup{
  pdftitle={Spectral submanifolds of the Navier-Stokes equations},
  pdfauthor={Gergely Buza}
}
\fi

\begin{document}

\maketitle

% REQUIRED
\begin{abstract}
Spectral subspaces of a linear dynamical system identify a large class of invariant structures that highlight/isolate the dynamics associated to select subsets of the spectrum. The corresponding notion for nonlinear systems is that of spectral submanifolds --  manifolds invariant under the full nonlinear dynamics that are determined by their tangency to spectral subspaces of the linearized system. In light of the recently-emerged interest in their use as tools in model reduction, we propose an extension of the relevant theory to the realm of fluid dynamics. We show the existence of a large (and the most pertinent) subclass of spectral submanifolds and foliations - describing the behaviour of nearby trajectories - about fixed points and periodic orbits of the Navier-Stokes equations. Their uniqueness and smoothness properties are discussed in detail, due to their significance from the perspective of model reduction. The machinery is then put to work via a numerical algorithm developed along the lines of the parameterization method, that computes the desired manifolds as power series expansions. Results are shown within the context of 2D channel flows.
\end{abstract}

% REQUIRED

% REQUIRED

\section{Introduction}

The question of existence concerning 
lower-dimensional representations of evolutionary systems
is 
central to 
many application-oriented subdisciplines of mathematical physics.
Within the engineering literature, the resulting objects are
usually termed \textit{reduced order models} (ROMs)
that 
serve the purpose 
of mitigating computational 
costs.
The concept is particularly topical
in fluid dynamics,
a field restrained by
an ever increasing demand for large-scale, complex numerical calculations
surpassing 
the rate at which computers have evolved.

Current state-of-the-art techniques popular amongst the applied fluids community
are standard Galerkin-type projections (including proper orthogonal decomposition (POD) \cite{PODholmes_lumley_berkooz_1996}),
and Koopman-mode reduction \cite{mezic2013Koopman_review} (via dynamic mode decomposition (DMD) \cite{Schmid2022DMD}).
Galerkin-based approaches project onto a physically relevant 
linear subspace (identified by the specific method, e.g.\ POD) of the phase space 
-- inherently not invariant under the nonlinear dynamics.
To combat the lack of invariance, these techniques rely on the chosen linear subspace having a large enough dimension
to render the discarded dynamics irrelevant.
Koopman/DMD is perhaps a more refined option but still assumes 
linearizability of the underlying dynamics 
(more precisely, the existence of a linearizing semiconjugacy \cite{KVALHEIM2021}), 
which greatly reduces  
its domain of validity within phase space.
Indeed, by design, this domain cannot contain more than a single invariant set (the origin) \cite{page_kerswell_2019}, 
which,
compared with our current understanding of turbulence, seems rather restrictive.

Appropriate nonlinear ROMs are sought as submanifolds of the phase space satisfying some suitable conditions,
which ensure that the resulting ROM is 
robust and representative of all neighboring dynamics.
Invariance, attractivity, persistence and uniqueness are often named 
to be desirable qualities for ROMs to have \cite{Haller2017}.
Of the aforementioned four, persistence and uniqueness
are satisfied by most methods currently in practice.
However, invariance turns out to be the most advantageous one to have (but perhaps the most cumbersome to achieve), 
for even a one-dimensional invariant manifold is an exact trajectory of the full system, thereby eliminating the usual dimensional requirements of non-invariant ROMs recalled above for the case of Galerkin projections.

In the realm of
 analysis concerning fluid dynamics,
the search for such 
invariant
manifolds 
began
shortly after the discovery of gobal attractors in the two-dimensional Navier-Stokes equations \cite{Foias1967,Ladyzhenskaya1975}.
Initial attempts were focused on globally attracting invariant (Lipschitz) manifolds that would encompass the attractor -- termed \textit{inertial manifolds}.
The existence of inertial manifolds was confirmed for a few of evolutionary PDEs including the reaction-diffusion equation \cite{FOIAS1988}, but has remained inconclusive for the case of the Navier-Stokes system, due to an inadequate ratio between the spectral gap and the Lipschitz constant \cite{robinson2001}. 
There exist a few recommendations for numerical procedures, however, that attempt to approximate the inertial manifold (even for the case of the Navier-Stokes system \cite{TITI1990})
and have proven to be useful in
a number of applications, including the Kuramoto-Sivashinsky equation \cite{JOLLY1990}.

If the dynamical region of interest involves only a select few fixed points or periodic orbits (commonly blanket termed simple invariant sets/solutions in the applied fluids literature),
localized approaches may be expected to perform better.
This is often the case in canonical flow geometries (e.g., plane-Couette and channel/pipe flows)
at moderate Reynolds numbers, where even the global attractor can be written as
a finite union of unstable manifolds attached to simple invariant solutions. 
The goal of the present work is to prove the existence of 
local invariant manifolds that 
form appropriate nonlinear ROMs according to the above definition, 
for the case of the Navier-Stokes equations. 
Here, the evolution of the reduced model is obtained upon restricting the Navier-Stokes semiflow to the invariant manifold in question.
We shall utilize invariant manifold theorems of Chen et al.~\cite{chen97} to obtain the ROMs of interest --
 foliation results from the same authors \cite{chen97} then establish the local attractivity of these ROMs in some generalized sense.
By its nature, the local theory allows for control over the Lipschitz constant, and hence the spectral gap issue will no longer hinder the derivation of existence results.
Moreover, it also avoids the need for globally existing solutions (in time), thereby enabling a seamless extension of the theory to the three-dimensional case.
Given their prevalence in standard flow geometries, we will also extend our results to cover invariant manifolds about periodic orbits.

Previous approaches in this direction include the works of Foias and Saut~\cite{FoiasSaut84a,FoiasSaut84b} and Gallay and Wayne~\cite{Gallay2001InvariantMA,Gallay2002_3d}.
The early attempts of Foias and Saut~\cite{FoiasSaut84a,FoiasSaut84b} obtained invariant manifolds about the origin of the unforced Navier-Stokes system (where the global attractor is the origin)
corresponding to the leftmost, strong-stable part of the spectrum.
These would not be suitable candidates for the reduced order modeling purposes herein as  
they are not attracting
-- nonetheless, we obtain them as a special case. 
The procedure of Gallay and Wayne~\cite{Gallay2001InvariantMA} is more in line with the present work,
as it is
similarly
reliant on the results of Chen et al.~\cite{chen97}.
The authors therein consider (unforced) flows on unbounded domains via the vorticity formulation, and compute slow manifolds about the sole fixed point at the origin to obtain asymptotics of the Oseen vortex. 
In this paper, we instead deal with the
(perhaps simpler) setting of bounded domains,
but
consider invariant manifolds about arbitrary
nontrivial fixed points and periodic orbits.

For the case of ODEs,
the local theory of invariant manifolds around fixed points (and invariant tori)
has now reached a fully-developed state with the emergence of the parameterization method \cite{Cabre2003a,Cabre2003b,Cabre2005,HARO2006_rigorous,HARO2006_numerical}.
The technique also benefits from
numerous computational advantages
(over traditional graph transform methods) \cite{haro2016,Mireles2015,BergMireles2016,Gonzalez2022paramfem}, 
and hence has been 
implemented
in our numerical procedure.
Perhaps among its most appealing qualities, the resulting algorithm is constrained to
converge
to a unique smoothest submanifold within the discretized setting (see Theorem~1.1, \cite{Cabre2003a}), thereby eliminating the divergence/ambiguity issues present in e.g.\ center manifold calculations.

The parameterization method has since been adapted to the setting of mechanical systems by Haller and colleagues \cite{Haller2016}, who have
termed the resulting invariant manifolds as \textit{spectral submanifolds}, and have used them for model reduction purposes ever since,
with their range of applicability
now reaching systems 
of up to $O(10^5)$ degrees-of-freedom \cite{Ponsioen2020,shobhit2022}. 
In some sense, our aim - via extending the relevant theory -
is to enable a smooth transition to the setting of Navier-Stokes,
so that one day, this technique can act as a
powerful reduction tool in the environment that perhaps needs it the most.
Work towards the latter has already begun by the same group \cite{kaszas2022,Cenedese2022} with the use of data-driven methods, extrapolating beyond the current theoretical foundations of the field -- we fill the remaining gaps in here.

In particular,
the paper's main purpose is 
to relieve the tension
upheld by concerns regarding the existence of spectral submanifolds in the full, infinite-dimensional phase space of the Navier-Stokes system. 
We provide
rigorous existence and uniqueness results for
spectral submanifolds and foliations about arbitrary fixed points and periodic orbits of the Navier-Stokes equations with time-independent forcing on bounded domains (and also for the case of channel flows).
As a result, a versatile class of ROMs is solidified into the arsenal of tools at the disposal of fluid dynamicists, satisfying the sought-after invariance property.  
Its versatility stems from the fact that the choice of spectral subset gives control over the decay rate towards the associated spectral submanifold.
The methodology is best suited to tackle problems concerning the vicinity of an equilibrium, but the domain of validity of the computed submanifolds could very well include larger subsets of the phase space. 
In the present work, we merely provide a glimpse into the possible range of applications by extracting a few heteroclinic connections between travelling wave solutions in a channel, but we envision the full range to encompass much more.
Notably, we believe, based on the examples herein and our general experience, that spectral submanifolds would perform well in identifying edge states (those mediating between the laminar and more complex states),
and in deciphering 
how trajectories approach a stable equilibrium,
which could serve valuable in the study of mixing flows, for instance.

It should be noted that the numerical procedures to approximate such structures are by no means new, even to the fluids literature.
Similar procedures have also aided the investigation of bifurcation problems in PDEs early on in their development \cite{coullet1983amplitude,arneodo1985dynamics}.
These studies have focused on expressing via polynomial expansions 
some distinguished 
manifolds (for the most part, the center manifold), assuming a priori their existence.
Particular attention should be given to the work of Coullet and Spiegel \cite{coullet1983amplitude}, which could be seen as a precursor to (the approximate theory of) the parameterization method.
These initial efforts
 have since been expanded on through a sequence of works,
 improving both the methodology itself \cite{roberts1989appropriate,roberts1996low} and also its computational aspects \cite{roberts2014book,roberts2014dynamical,roberts2015macroscale}.

The remainder of the paper is organized as follows.
In Section~\ref{sect:functional_setup}, we recall the functional setup required to pose the Navier-Stokes system as an evolution problem, and give the definitions of spectral submanifolds and their corresponding foliations.
Then, Section~\ref{sect:semiflow} establishes the
existence and smoothness of the local semiflow on the specific function space considered (essentially the Sobolev space $W^{1,2}$).
Section~\ref{sect:manifolds} contains the main theoretical body of the paper, with all existence and uniqueness results concerning the desired manifolds and their foliations (see Theorem~\ref{thm:summary} for a summary).
To introduce the foundations underpinning our numerical procedure, the parameterization method is recalled in Section~\ref{sect:parameterization}.
Three examples are then shown in Section~\ref{sect:channelflow}: two in the context of Newtonian channel flows and one for an Oldroyd-B (viscoelastic) fluid.
In Section~\ref{sect:channelflow}, we simply approximate within a discretized setting the submanifolds constructed in Section~\ref{sect:manifolds}, and hence it can be read independently of all prior sections (apart from the notational conventions of Section~\ref{sect:functional_setup}).
Finally, a brief discussion follows detailing the computational aspects and a few alternative routes of implementation (Section~\ref{sect:discussion}).

\section{Preliminaries} 
\label{sect:functional_setup}

We begin by transforming the Navier-Stokes system to an evolution equation on an appropriate function space \cite{Fujita1964}. 
Take $\Omega \subset \mathbb{R}^d$, $d = 2,3$, to be a bounded domain with piecewise smooth, rigid boundary $\partial \Omega$.
Fluid enclosed within $\Omega$ obeys the Navier-Stokes equations
\begin{subequations}
\begin{gather}
    \partial_t v - \frac{1}{Re} \Delta v + (v \cdot \nabla ) v + \nabla p = f \quad \text{in } \Omega, \\
    \nabla \cdot v = 0 \quad \text{in } \Omega, \\
    v = 0 \quad \text{on } \partial \Omega.
\end{gather}\label{eq:NSE}\end{subequations}
Here, $Re$ denotes the Reynolds number, defined as the ratio of inertial forces to viscous forces acting within the fluid; $v$ and $p$ are the sought velocity and pressure fields, respectively; $f$ denotes the forcing (assumed to be time-independent in this work).
We examine the behaviour of solutions in the vicinity of a steady state 
$(U,P)$\footnote{In what follows, we will dispose of the explicit reference to the pressure field $P$ when referring to steady states to keep the notation uncluttered.
It should nonetheless be understood that there is always an associated pressure field to a steady state $U$.}
(a time-independent solution of \eqref{eq:NSE}), which we assume to be smooth (automatically satisfied if $\partial \Omega$ and $f$ are $C^\infty$).
The perturbation velocity field $u = v - U$
and pressure field $q = p - P$ satisfy
the equation
\begin{subequations}
\begin{gather}
    \partial_t u - \frac{1}{Re} \Delta u + (U \cdot \nabla ) u + (u \cdot \nabla ) U + (u \cdot \nabla ) u + \nabla q = 0 \quad \text{in } \Omega, \label{eq:NSE_local_mom}\\
    \nabla \cdot u = 0 \quad \text{in } \Omega, \\
    u = 0 \quad \text{on } \partial \Omega.
\end{gather}
\label{eq:NSE_local}
\end{subequations}

Denote by $L^p_\sigma$ the Banach space obtained by closing the set of smooth, solenoidal vector fields with compact support with respect to the $L^p(\Omega, \mathbb{R}^d)$ norm, i.e.,\footnote{$W^{k,p}_\sigma$ spaces are defined analogously.}
$$
L^p_\sigma := \overline{ \Big\{ \, u \in C^\infty_c (\Omega, \mathbb{R}^d) \; \big\vert \; \mathrm{div} \, u = 0 \, \Big\} }^{L^p(\Omega, \mathbb{R}^d)},
$$
which,
according to the Helmholtz-Leray decomposition,
is complemented in $L^p(\Omega, \mathbb{R}^d)$ by the closure of gradient functions.
The associated projection $P_L : L^p(\Omega, \mathbb{R}^d) \to L^p_\sigma$
is therefore well-defined, bounded and orthogonal.
Upon applying $P_L$
to both sides of the momentum equation \eqref{eq:NSE_local_mom}, we arrive at
the functional evolution equation
\begin{equation}
    \frac{d u}{dt}  =  \mathcal{A}_{U} u - B(u,u),
    \label{eq:main}
\end{equation}
where
\begin{gather}
    \mathcal{A}_{U} u  = \frac{1}{Re} P_L \Delta u - B (U,u) - B(u,U), \label{eq:AU} \\
    B(v,w) = P_L (v \cdot \nabla ) w. \nonumber
\end{gather}
In particular, $\mathcal{A}_0 = \frac{1}{Re} P_L \Delta$ recovers the (scaled) Stokes operator. 

Denote by $H:= L^2_\sigma$ and $V:= W^{1,2}_\sigma$.
We shall take the latter as the phase space to analyze the Navier-Stokes system \eqref{eq:NSE_local}.
The space $H$ is equipped with the standard $L^2$ inner product
\begin{displaymath}
    \langle v,w \rangle_H := \int_\Omega v \cdot w, \qquad v,w \in H,
\end{displaymath}
whereas $V$ is given the inner product
\begin{equation}
    \langle v,w \rangle_V := \langle \nabla v, \nabla w \rangle_H, \qquad v,w \in V,
    \label{eq:Vnorm}
\end{equation}
and hence both are Hilbert spaces.
Note that \eqref{eq:Vnorm} is equivalent to the standard $W^{1,2}$ metric due to the Poincaré inequality.
We will generally denote both elements of $V$ and $V$-valued functions by $u,v,w$ --
the context will always make the distinction clear.

Denote by $\rho (\mathcal{A}_{U})$ the resolvent set of operator $\mathcal{A}_{U}$, i.e.\ the set of $\lambda \in \mathbb{C}$ for which the resolvent operator
\begin{displaymath}
  R_\lambda ( \mathcal{A}_{U}) := \left( \lambda I - \mathcal{A}_{U} \right)^{-1} : L^p_\sigma \to L^p_\sigma
\end{displaymath}
is defined and bounded.
The spectrum of $\mathcal{A}_{U}$, $\sigma ( \mathcal{A}_{U} ) = \mathbb{C} \setminus \rho (\mathcal{A}_{U})$, is characterized by the following theorems.

\begin{theorem}[Spectrum of $\mathcal{A}_{U}$, \cite{yudovich1989}] \label{thm:spect}
  The operator $\mathcal{A}_{U}: \mathrm{dom}_p (\mathcal{A}_{U}) \to L^p_\sigma$ is closed, has compact resolvent and thus a discrete spectrum consisting of eigenvalues $\lambda_1,\lambda_2,\ldots,$ with $\mathrm{Re} \, \lambda_j \to -\infty$, for all $p \in (1,\infty)$. 
  The spectrum of $\mathcal{A}_{U}$ satisfies
  \begin{equation}
    \mathrm{Re} \, \lambda \leq  c +   \max \, \sigma ( \mathcal{A}_0 )  =: \omega_0; \qquad
    | \mathrm{Im} \, \lambda| \leq a \sqrt{\left(-\mathrm{Re} \, \lambda + c \right) Re}+b,
  \label{eq:spectrum_bounds}
  \end{equation} 
  where
  \begin{displaymath}
    a = \max_{x \in \Omega} | U(x)|; \qquad b = \max_{x \in \Omega} | \mathrm{curl} \, U (x) |; \qquad c = \max_{x \in \Omega} \sqrt{\sum_{i,k=1}^d \left( \partial_k U_i (x) \right)^2}.
  \end{displaymath} 
  Moreover, the sequence of eigenvectors and associated vectors of the operator $\mathcal{A}_U$ forms a complete system in $V$, in $L^p_\sigma$ $(1 < p < \infty)$; and $\mathrm{dom}_p \left(\mathcal{A}_U \right) $  $(1 < p < \infty$, in the sense of $\Vert \cdot \Vert_{W^{2,p}})$. 
\end{theorem}

\begin{theorem}[Sectoriality of $\mathcal{A}_{U}$, \cite{yudovich1989}] \label{thm:sectorial}
  The operator $\mathcal{A}_{U} : \mathrm{dom}_p (\mathcal{A}_{U}) \to L^p_\sigma$ is sectorial 
  for $1 < p < \infty$.
  That is, there exists a pair $(\omega, \theta) \in \mathbb{R} \times (\pi/2,\pi)$ with $\omega > \omega_0$ such that the sector
\begin{displaymath}
    \Sigma_{\omega,\theta} := \big\{ \, \lambda \in \mathbb{C} \; \big\vert \;
    | \mathrm{arg} (\lambda-\omega) | < \theta \, \big\}
\end{displaymath}
is contained in $\rho(\mathcal{A}_{U})$, and
  \begin{displaymath}
        \Vert R_\lambda (\mathcal{A}_{U}) \Vert_{\mathcal{L} (L^p_\sigma)} \leq \frac{C}{|\lambda - \omega|} 
  \end{displaymath}
  for all $\lambda \in \Sigma_{\omega,\theta}$. Moreover,
  \begin{equation}
    \Vert R_\lambda (\mathcal{A}_{U}) \Vert_{\mathcal{L} \left(L^p_\sigma,W^{2,p}(\Omega,\mathbb{R}^d)\right)} \leq C
    \label{eq:ellipticregularity}
  \end{equation}
  holds uniformly with respect to $\lambda \in \Sigma_{\omega,\theta}$.
\end{theorem}

\subsection{Spectral submanifolds}

The following segment is devoted to 
establishing a splitting of $V$ into $\mathcal{A}_U$-invariant subspaces corresponding to subsets of the spectrum.
The envisioned spectral submanifolds will be characterized by their tangency to these subspaces.

\begin{figure}
\includegraphics[width=0.7\textwidth]{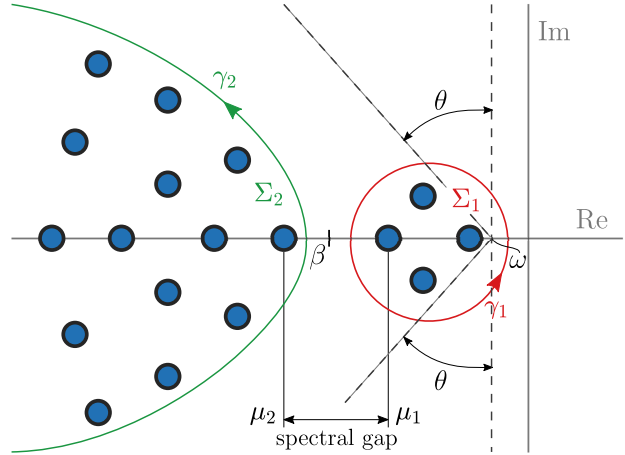}
\centering
\caption{A typical spectrum  of $\mathcal{A}_U$ shown on the complex plane, with $\sigma (\mathcal{A}_U)$ decomposed according to \eqref{eq:sugma}.
}
\label{fig:spectrum_init}
\end{figure}

For any $\beta < \max \mathrm{Re} \, \sigma ( \mathcal{A}_U ),$  $\beta \notin \mathrm{Re} \, \sigma ( \mathcal{A}_U )$, we may
decompose the spectrum into two disjoint components $\sigma ( \mathcal{A}_U ) = \Sigma_1 \cup \Sigma_2$ such that 
\begin{equation}
  \Sigma_1 = \left\{ \, \lambda \in \sigma ( \mathcal{A}_U ) \; \vert \; \mathrm{Re} (\lambda) > \beta_1 \, \right\}, \qquad 
    \Sigma_2 = \left\{ \, \lambda \in \sigma ( \mathcal{A}_U ) \; \vert \; \mathrm{Re} (\lambda) < \beta_2 \, \right\},
    \label{eq:sugma}
\end{equation}
with $\beta_1>\beta>\beta_2$ (see Figure~\ref{fig:spectrum_init}). 
Set $\mu_1:= \min \mathrm{Re} \, \Sigma_1$; $\mu_2:= \max \mathrm{Re} \, \Sigma_2$.
The difference $\mu_1 - \mu_2$ is called the spectral gap, which, roughly speaking, controls the separation of Lyapunov exponents for trajectories inside of, and towards, the invariant manifold to be constructed. 
Define the projection $P_{\Sigma_1} : V \to \mathrm{im}_V ( P_{\Sigma_1})$ via holomorphic functional calculus
\begin{displaymath}
  P_{\Sigma_1} = \frac{1}{2 \pi \mathrm{i}} \int_{\gamma_1} R_z ( \mathcal{A}_U) dz,
\end{displaymath}
with $\gamma_1$ a smooth curve 
encircling $\Sigma_1$ within an isolated neighborhood (see Figure~\ref{fig:spectrum_init}).
Then its image $\mathrm{im}_V (P_{\Sigma_1}  ) \subset \mathrm{dom}_V ( \mathcal{A}_U)$ is finite dimensional and $\mathcal{A}_U$-invariant, the operator $\mathcal{A}_U P_{\Sigma_1}: \mathrm{im}_V (P_{\Sigma_1} ) \to \mathrm{im}_V (P_{\Sigma_1}  )$ is bounded with spectrum $\sigma (\mathcal{A}_U P_{\Sigma_1}) = \Sigma_1$.
Take $P_{\Sigma_2}:= I - P_{\Sigma_1}$, then
\begin{equation}
  V = \mathrm{im}_V (P_{\Sigma_1}) \oplus \mathrm{ker}_V (P_{\Sigma_1}) = P_{\Sigma_1} V \oplus P_{\Sigma_2} V,
  \label{eq:Vdecomposition}
\end{equation}
and the unbounded operator $\mathcal{A}_U P_{\Sigma_2}: P_{\Sigma_2} \mathrm{dom}_V ( \mathcal{A}_U) \to  P_{\Sigma_2}  V$ has spectrum $\Sigma_2$ \cite{buhler2018functional}.
The subspaces $P_{\Sigma_1} V$ and $ P_{\Sigma_2} V$ are usually termed spectral subspaces of $\mathcal{A}_U$.

It is convenient to introduce a norm of box type, with respect to the splitting $P_{\Sigma_1} V \oplus P_{\Sigma_2} V$, by setting
\begin{displaymath}
   \Vert v \Vert_b : = \sup \left\{ \Vert P_{\Sigma_1} v \Vert_V , \Vert P_{\Sigma_2} v \Vert_V \right\}.
\end{displaymath}
This norm is equivalent to $\Vert \cdot \Vert_V$, as
\begin{displaymath}
   \Vert v \Vert_b \leq \Vert v \Vert_V \leq 2\Vert v \Vert_b.
\end{displaymath}
In what follows, we shall distinguish open balls with respect to the new norm $\Vert \cdot \Vert_b$ via the hat symbol, e.g.\ $\widehat{B}_r(u) = \{ v \in V \, | \, \Vert v-u \Vert_b < r \}$.
Moreover, we shall denote open balls in $P_{\Sigma_1}V$ (resp.\ $P_{\Sigma_2}V$) by $B_r^1(u_1)$ for $u_1 \in P_{\Sigma_1}V$ (resp.\ $B_r^2(u_2)$ for $u_2 \in P_{\Sigma_2}V$), where $B_r^1(u) =  \{ v_1 \in P_{\Sigma_1}V \, | \, \Vert v_1-u_1 \Vert_V < r \}$ ($B_r^2(u) =  \{ v_2 \in P_{\Sigma_2}V \, | \, \Vert v_2-u_2 \Vert_V < r \}$).

The class of spectral submanifolds we plan to extract in this work are all according to spectral decompositions of the form \eqref{eq:sugma}.
The resulting manifolds will hence be denoted 
by either
$W^\beta(U)$ or $W^{\Sigma_1}(U)$ interchangeably, depending on whichever is more convenient.
Here, $U$ stands for the base point of the manifold.
In most of what follows, we work in the setting of the perturbation equation around $U$, \eqref{eq:NSE_local}, in which case, the appropriate notation would be $W^\beta (0)$ or $W^{\Sigma_1}(0)$ for manifolds about $U$ -- we shall omit the $(0)$ to avoid confusion.
We now define their local version, $W^{\Sigma_1}_{\mathrm{loc},r} \subset \widehat{B}_r(0)$.

\begin{definition}[Spectral submanifolds]
\label{def:wloc}
Let $\varphi_t$ denote the semiflow generated by the Navier-Stokes system on $V$ (to be shown in Section~\ref{sect:semiflow}). 
We say that $W^{\Sigma_1}_{\mathrm{loc},r} \subset \widehat{B}_r(0) \subset V$ is a local $C^k$ spectral submanifold corresponding to the spectral subset $\Sigma_1$ if it satisfies the following:
\begin{enumerate}[label=\upshape{(\roman*)}]
    \item $W^{\Sigma_1}_{\mathrm{loc},r}$ is locally invariant under $\varphi_t$, i.e., if $u \in W^{\Sigma_1}_{\mathrm{loc},r}$, then $\varphi_t (u ) \in W^{\Sigma_1}_{\mathrm{loc},r}$ for those $t \geq 0$ with
    $
    \varphi_\tau (u) \in \widehat{B}_r(0)
    $
    for all $\tau \in [0,t]$;
    \item \label{def:wloc_2} $W^{\Sigma_1}_{\mathrm{loc},r}$ is tangent to the spectral subspace $P_{\Sigma_1}V$ at the origin;
    \item $W^{\Sigma_1}_{\mathrm{loc},r}$ is a $C^k$ submanifold of $\widehat{B}_r(0)$.
\end{enumerate}
\end{definition}

The corresponding local (spectral) foliation, a family of submanifolds denoted by\footnote{We generally dispose of the explicit reference to the fixed point about which the foliation was computed ($0$ in the upper-right index).}
$$
\{ M_{\mathrm{loc},r}^{\Sigma_2,0} (u) \, | \, u \in \widehat{B}_r(0) \} = \{ M_{\mathrm{loc},r}^{\beta,0} (u) \, | \, u \in \widehat{B}_r(0) \},
$$
is defined as follows.

\begin{definition}
\label{def:foliation}
Let $\varphi_t$ denote the semiflow generated by the Navier-Stokes system on $V$ (to be shown in Section~\ref{sect:semiflow}).
A family of submanifolds $
\{ M_{\mathrm{loc},r}^{\Sigma_2}(u) \subset \widehat{B}_r(0) \, | \, u \in \widehat{B}_r(0) \},
$ parameterized by $u \in \widehat{B}_r(0) $ is said to be a $C^k \times C^\ell$ foliation for $\widehat{B}_r(0)$ corresponding to the spectral subset $\Sigma_2$ if the following conditions are satisfied:
\begin{enumerate}[label=\upshape{(\roman*)}]
    \item $\{ M_{\mathrm{loc},r}^{\Sigma_2}(u) \, | \, u \in \widehat{B}_r(0) \}$ form a locally positively invariant foliation of $\widehat{B}_r(0)$.
    That is,
    $
    \varphi_t (M_{\mathrm{loc},r}^{\Sigma_2}(u)) \cap \widehat{B}_r(0) \subset   M_{\mathrm{loc},r}^{\Sigma_2}(\varphi_t(u))  
    $
    for those $t \geq 0$ with
    $
    \varphi_\tau (u) \in \widehat{B}_r(0)
    $
    for all $\tau \in [0,t]$;
    \item $u \in M_{\mathrm{loc},r}^{\Sigma_2}(u)$ for each $u \in \widehat{B}_r(0)$;
    \item \label{def_fol:prop3} $M_{\mathrm{loc},r}^{\Sigma_2}(u)$ and $M_{\mathrm{loc},r}^{\Sigma_2}(v)$ are either identical or disjoint for each $u,v \in \widehat{B}_r(0)$;
    \item $M_{\mathrm{loc},r}^{\Sigma_2}(0)$ is tangent to the spectral subspace $P_{\Sigma_2}V$; 
    \item The set $\{ (u,v) \, | \, u \in \widehat{B}_r(0), \; v \in  M_{\mathrm{loc},r}^{\Sigma_2}(u)  \}$ is a $C^k \times C^\ell$ submanifold of $\widehat{B}_r(0) \times \widehat{B}_r(0) $.
\end{enumerate}
\end{definition}

The role of the foliation $M_{\mathrm{loc},r}^{\Sigma_2}(u)$
is to synchronize trajectories in their approach to the manifold $W^{\Sigma_1}_{\mathrm{loc},r}$, and determine the rate of attraction. 
In particular, we will show that $W^{\beta}_{\mathrm{loc},r}$ attracts all nearby trajectories remaining within $\widehat{B}_r(0)$ at a rate of $\mathcal{O}(e^{\beta t})$ (see Corollary~\ref{cor:Wloc}; also Theorem~\ref{thm:summary} for a summary of all results).

The objective throughout 
the following few sections
is to show that spectral submanifolds and their respective foliations exist in the phase space $V$ of the Navier-Stokes system.
The first step is to establish the existence and smoothness of the (local) semiflow $\varphi_t$, content of Section~\ref{sect:semiflow}, after a brief recollection of results regarding fractional powers of $\mathcal{A}_U$.

\subsection{Fractional powers of sectorial operators}

There are two ingredients to establishing that \eqref{eq:main} generates a semiflow on $V$.
One is the fact that $B$ is a bounded operator from $V \times V$ into $L^p_\sigma$, for some appropriately chosen $p$, as we shall see later in Lemma~\ref{lemma:F}.
The other is the smoothing property of $e^{\mathcal{A}_U t}$, in particular, that it acts boundedly from $L^p_\sigma$ into $V$, for the same $p$ as above.
We will address these questions in full generality (in terms of $p$), then conclude that there exists $p \in (6/5,3/2]$ for which both hold.

For the smoothing property, the notion of fractional powers is necessary.
Choose $\omega > \omega_0$ so that 
$$
\mathcal{A}_{U,\omega}:= \mathcal{A}_U - \omega I
$$
has spectrum contained entirely in the negative half plane, $\mathrm{Re} \, \sigma \left( \mathcal{A}_{U,\omega} \right) < 0 $.
Then for $\alpha > 0$ we may write
\begin{equation}
    \left( -\mathcal{A}_{U,\omega} \right)^{-\alpha} = \frac{1}{\Gamma (\alpha)} \int_0^\infty t^{\alpha-1} e^{t  \mathcal{A}_{U,\omega} } dt,
    \label{eq:-alpha}
\end{equation}
where $\Gamma$ denotes the Euler $\Gamma$ function.
Expression \eqref{eq:-alpha} defines a bounded linear operator on $L^p_\sigma$ that is injective \cite{lunardi1995}. 
For positive powers, set $\mathrm{dom}_p \left( \left( -\mathcal{A}_{U,\omega} \right)^{\alpha}\right) := \mathrm{im}_p \left( \left( -\mathcal{A}_{U,\omega} \right)^{-\alpha} \right)$ and
\begin{equation}
    \left( -\mathcal{A}_{U,\omega} \right)^{\alpha} : = \left(\left( -\mathcal{A}_{U,\omega} \right)^{-\alpha} \right)^{-1}. \label{eq:alpha}
\end{equation}
The operator $\left( -\mathcal{A}_{U,\omega} \right)^{\alpha}$ is closed, thus the space $\mathrm{dom}_p \left( \left( -\mathcal{A}_{U,\omega} \right)^{\alpha}\right)$ endowed with the norm $\Vert  v \Vert_{\left( -\mathcal{A}_{U,\omega} \right)^{\alpha}} = \Vert \left( -\mathcal{A}_{U,\omega} \right)^{\alpha} v \Vert_{L^{p} }$ (equivalent to the graph norm) defines a Banach space, independent of the choice of $\omega$ (\cite{henry1981}, Theorem 1.4.6).

The following embedding property of 
domains of fractional powers
will be of use.

\begin{theorem}[Theorem 1.6.1, \cite{henry1981}] \label{thm:fractdom}
    Suppose $\Omega \subset \mathbb{R}^d$ is as above, $1 \leq p < \infty$, and $A$ is a sectorial operator in $L^p(\Omega)$ such that $\mathrm{dom}_p(A)$ is continuously embedded in $W^{m,p}(\Omega)$ for some $m \geq 1$. 
    Then for $0 \leq\alpha \leq 1$
    \begin{displaymath}
        \mathrm{dom}_p((-A)^\alpha) \subset W^{k,q}(\Omega) \quad \text{when} \quad k - \frac{d}{q} < m \alpha - \frac{d}{p}, \quad q \geq p,
    \end{displaymath}
    with continuous inclusion.
\end{theorem}

By \eqref{eq:ellipticregularity}, $\mathrm{dom}_p \left( \mathcal{A}_{U,\omega} \right)$ is continuously embedded in $W^{2,p}_\sigma$ for all $p \in (1,\infty)$. 
  Taking 
  \begin{equation}
  \frac{2d}{d+2} < p < \infty, \qquad  \text{and} \qquad \frac12 + \frac{d}{2} \left( \frac1p - \frac12 \right) < \alpha < 1, \label{eq:alpharange}
  \end{equation} 
  Theorem~\ref{thm:fractdom} guarantees that the inclusion $ \mathrm{dom}_p((-\mathcal{A}_{U,\omega} )^\alpha) \xhookrightarrow{} V$ is bounded, which we shall repeatedly use in the following sections.

Through sectoriality, $\mathcal{A}_{U}$ generates an analytic semigroup $e^{\mathcal{A}_{U} t}$ on $L^p_\sigma$ for all $p \in (1,\infty)$, which satisfies 
\begin{gather}
   \Vert e^{\mathcal{A}_{U} t} \Vert_{\mathcal{L}(L^p_\sigma)} \leq C_0 e^{\omega t}; \label{eq:eAU0}\\
   \Vert \left( -\mathcal{A}_{U,\omega} \right)^{\alpha} e^{\mathcal{A}_{U,\omega} t} \Vert_{\mathcal{L}(L^p_\sigma)} \leq C_\alpha t^{-\alpha}  \label{eq:alphabound}
\end{gather}
for any $\omega > \omega_0$ and $t>0$. 
Combining \eqref{eq:alphabound} with the above observation, we obtain
\begin{equation}
    \Vert e^{t \mathcal{A}_U} \Vert_{\mathcal{L}(L^p_\sigma,V)} \leq C \Vert  \left( -\mathcal{A}_{U,\omega} \right)^{\alpha} e^{t \mathcal{A}_{U,\omega} + t \omega } \Vert_{\mathcal{L}(L^p_\sigma)} \leq C_\alpha t^{-\alpha} e^{\omega t}
    \label{eq:commonplace_bound}
\end{equation}
for appropriate $p$ and $\alpha$ chosen according to \eqref{eq:alpharange}.

\section{Existence and smoothness of local semiflow} 
\label{sect:semiflow}

A family of maps $\{\varphi_t\}_{t \geq 0}$, $\varphi_t:V \to V$, is called an (autonomous) semiflow on $V$ if
\begin{subequations}
    \begin{align}
    \varphi_0 &= I, \label{eq:semiflow1}\\ 
    \varphi_t \circ \varphi_s &= \varphi_{t+s}, \qquad \forall t,s \geq 0, \label{eq:semiflow2}
\end{align}\label{eq:semiflow}\end{subequations}
with $t \mapsto \varphi_t(u_0)$ continuous for each initial condition $u_0 \in V$.
A local semiflow is a semiflow $\varphi:\mathcal{D} \to V$ defined on an open set $\mathcal{D} \subset \mathbb{R}^{\geq 0} \times V$ such that $(0,v) \in \mathcal{D}$ for all $v \in V$.

The majority of this section is concerned with establishing that equation \eqref{eq:main} generates a smooth (local) semiflow.
In doing so, we also consider a slightly altered system given by
\begin{equation}
    \frac{d u}{dt}  =  \mathcal{A}_{U} u - \chi_\rho (u) B(u,u) 
    \label{eq:main_cutoff}
\end{equation}
constructed to conform to requirements 
normally posed by invariant manifold theories 
on the Lipschitz constant of the nonlinearity (see \eqref{eq:lipschitz_requirement} and \eqref{eq:lipN}).
Here, $\chi : V \to [0,1]$ is a $C^\infty$ smooth cutoff function satisfying
\begin{equation}
\chi (u) = 
\begin{cases}
1, & \Vert u \Vert_V \leq \frac12, \\
0, & \Vert u \Vert_V \geq 1,
\end{cases}
\qquad
\text{with }
\Vert D \chi (u) \Vert_{\mathcal{L}(V)} \leq 3 
\text{ for all }
u \in V,
\label{eq:chi}
\end{equation}
and $ \chi_\rho (u): =  \chi(u/\rho)$, with $\rho > 0$ to be chosen in accordance with  the Lipschitz condition on the nonlinearity (see \eqref{eq:lipschitz_requirement}).
Whether such a smooth cutoff function exists is a nontrivial issue in general Banach spaces (see \cite{FRY2002}, for instance), but it certainly is the case for Hilbert spaces. 
This serves as the primary reason to exclusively treat the problem on $V$ -- 
the analysis to follow could otherwise be extended to a larger class of $W^{k,p}$ spaces.

For a map $f$ between two Banach spaces $X,Y$, and a subset $O \subset X$ contained in its domain, denote by
\begin{displaymath}
    \mathrm{Lip}_{X \to Y} ( f \, | \,  O )
    = \sup_{\substack{x,y \in O \\ x \neq y}} \frac{\Vert f(x) - f(y) \Vert_Y}{\Vert x-y \Vert_X}
\end{displaymath}
the Lipschitz constant of $f$ on $O$. 
$f$ is called globally Lipschitz from $X$ to $Y$ if $O$ can be chosen to be the entire space $X$ and $\mathrm{Lip}_{X \to Y} ( f) := \mathrm{Lip}_{X \to Y} ( f \, | \,  X)< \infty$;
and locally Lipschitz if $\mathrm{Lip}_{X \to Y} ( f \, | \, \overline{B_r (0)} ) < \infty$ for all $r > 0$,
where $B_r (0)$ denotes the open ball of radius $r$ centered at the origin.
The Gateaux derivative of $f$ is denoted by $D^G f$; its Fréchet derivative is denoted by $Df$.

Let us also introduce shorthand notation for the nonlinear term in the equation:
\begin{equation}
    F (u) : = - B(u,u) , \qquad \text{and} \qquad F_\rho (u) : = - \chi_\rho (u) B(u,u) .
    \label{eq:F_Frho}
\end{equation}

The following lemma characterizes $F$ and $F_\rho$.

\begin{lemma} \label{lemma:F}
Let $F$ and $F_\rho$ be as in \eqref{eq:F_Frho} and let  $ p \in [1,3/2]$ if $d = 3$, and $p \in [1,2)$ if $d = 2$.
Then $F, F_\rho \in C^{\infty}(V,L^p_\sigma)$. 
Moreover, $F_\rho: V \to L^p_\sigma$ is globally Lipschitz with 
\begin{displaymath}
  \Vert F_\rho (u) -   F_\rho (v) \Vert_{L^p}   \leq \rho C_{p,d} \Vert u - v \Vert_V, \qquad u,v \in V.
\end{displaymath}
\end{lemma}
\begin{proof}
  We show that $B : V \times V \to L^p_\sigma$ is a bounded bilinear map. Take $u,v \in V$, then
  \begin{align}
      \Vert B(u,v) \Vert_{L^p} &\leq C_p \Vert u \Vert_{L^{2p/(2-p)}} \Vert \nabla v \Vert_{L^2} \nonumber \\
      &\leq C_{p,d} \Vert u \Vert_V \Vert v \Vert_V
      \label{eq:Bbound}
  \end{align}
  using Hölder's inequality and the Sobolev embedding theorem (which determines the range of admissible $p$). 
  In particular, the map $u \mapsto B(u,u)$ is in $C^{\infty}(V,L^p_\sigma)$ and thus $F \in  C^{\infty}(V,L^p_\sigma)$.
  The same holds for $F_\rho$, given that $\chi_\rho \in C^{\infty}(V,[0,1])$.
  The last assertion is now just an application of the 
  mean value theorem:
  \begin{align*} 
      \mathrm{Lip}_{V \to L^p_\sigma } (F_\rho) &= \mathrm{Lip}_{V \to L^p_\sigma } (F_\rho \, | \, \overline{B_\rho (0)}) \\
      &\leq \sup_{u \in \overline{B_\rho (0)}} \Vert D F_\rho (u) \Vert_{\mathcal{L}(V,L^p_\sigma)} \\
      &= \sup_{u \in \overline{B_\rho (0)}} \sup_{v \in V \setminus \{0 \}}  \frac{\Vert D \chi_\rho (u) [v] B(u,u) +\chi_\rho (u) B(u,v) + \chi_\rho (u) B(v,u) \Vert_{L^p_\sigma}}{\Vert v \Vert_V} \\
      & \leq \frac{3}{\rho} C_{p,d} \rho^2 + 2 C_{p,d} \rho,
  \end{align*}
  where the last line used \eqref{eq:chi} and \eqref{eq:Bbound}.
\end{proof}

We now recall an existence and smoothness result, which in essence is
 obtained upon combining results of Weissler~\cite{WEISSLER1979,Weissler1980TheNI} and Henry~\cite{henry1981}.
Its proof is included in Appendix~\ref{appendix} for the sake of completeness.

\begin{proposition}
\label{prop:semiflow_summary}
    The perturbation equation for the Navier-Stokes system \eqref{eq:main}
    generates a local semiflow $\varphi: \mathcal{D} \to V$  that is jointly $C^\infty$ on $\mathcal{D} \cap (\mathbb{R}^{>0} \times V)$.
    In two spatial dimensions ($d = 2$), the domain can be extended to define a global semiflow on $\mathcal{D} = \mathbb{R}^{\geq 0} \times V$.
    For the case of the cut-off system \eqref{eq:main_cutoff}, this property continues to hold for $d = 3$.
\end{proposition}

\begin{proof}
    The proof is deferred to Appendix~\ref{appendix}.
\end{proof}

\section{Construction of invariant manifolds and foliations}
\label{sect:manifolds}

In this section, we apply results of Chen et al.~\cite{chen97} to obtain  
spectral submanifolds and foliations of interest (according to Definitions \ref{def:wloc} and \ref{def:foliation})
about the base state $U$ in $V$.

As is standard in invariant manifold theory,
we shall initially work with the cut-off version of the Navier-Stokes semiflow determined by
\begin{equation}
    \varphi_t^{\rho} (u) = e^{t \mathcal{A}_{U} } u  + \int_{0}^t e^{(t-s)\mathcal{A}_{U} } F_\rho \circ \varphi_s^{\rho} (u) \, ds, \qquad u \in V,
    \label{eq:cutoff_semiflow}
\end{equation}
in order to apply the abstract results of Chen et al.~\cite{chen97}, 
then subsequently discuss how these pertain to the real system, 
using the fact that $\varphi$ and $\varphi^\rho$ agree locally.
As a preliminary step, we state an explicit bound for the Lipschitz constant associated to the
nonlinear part of $\varphi^\rho_t$, which we denote by
\begin{equation}
  N_\rho (t,u) := \int_0^t e^{(t-s)\mathcal{A}_{U} }  F_\rho \circ \varphi^\rho_s(u)   ds, \qquad t > 0, \quad u \in V.
  \label{eq:Nrho}
\end{equation}

\begin{lemma}
\label{lemma:lipsic2}
Fix any pair $(p,\alpha)$ satisfying \eqref{eq:alpharange} and the assumptions of Lemma~\ref{lemma:F}. Then, the
 operator $ N_\rho (t,\cdot):V \to V$ is globally Lipschitz for $\rho>0$ sufficiently small, with
\begin{equation}
 \Vert N_\rho (t,u) - N_\rho (t,v) \Vert_V  \leq  \rho  C_{\alpha,p,d} t^{1-\alpha} \max \left\{ 1, e^{2 t \omega } \right\} \Vert u -v \Vert_{V},
 \label{eq:lipN}
\end{equation}
for all $u,v \in V$ and $t >0$.
\end{lemma}
\begin{proof}
With $(p,\alpha)$ as in the assumption, we have, using \eqref{eq:commonplace_bound}
\begin{align*}  
\Vert \varphi_t^\rho (u) - &\varphi_t^\rho (v) \Vert_{V}  =\left\Vert e^{t \mathcal{A}_{U} } (u -v) + \int_0^t e^{(t-s)\mathcal{A}_{U} } \left[  F_\rho \circ \varphi_s^\rho (u) -   F_\rho \circ \varphi_s^\rho (v) \right] ds \right\Vert_{V}  \\
&\leq C_0 e^{t \omega } \Vert u -v \Vert_{V} +  \int_0^t  (t-s)^{-\alpha} C_\alpha e^{(t-s) \omega } \left\Vert F_\rho \circ \varphi_s^\rho (u) -   F_\rho \circ \varphi_s^\rho (v) \right\Vert_{L^{p}_\sigma}  ds \\
& \leq C_0 e^{t \omega } \Vert u -v \Vert_{V} + t^{1-\alpha} \frac{C_{\alpha,p,d}  \max \left\{ e^{ t \omega },1 \right\}}{1- \alpha} \rho  \sup_{s \in (0,t)} \Vert \varphi_s^\rho (u) - \varphi_s^\rho (v) \Vert_V.
\end{align*}  
Taking the supremum over $t \in (0,T)$ and choosing $\rho$ small enough such that
\begin{displaymath}
  \rho < \frac{1-\alpha}{2 C_{\alpha,p,d} \max \left\{ e^{T \omega },1 \right\}} T^{\alpha-1},
\end{displaymath} 
we obtain
\begin{displaymath}
  \sup_{t \in (0,T)} \Vert \varphi_t^\rho (u) - \varphi_t^\rho (v) \Vert_V \leq 2 C_0 \max \left\{ e^{T \omega },1 \right\} \Vert u -v \Vert_{V},
\end{displaymath}
and thus
\begin{displaymath}
 \Vert N_\rho (t,u) - N_\rho (t,v) \Vert_V \leq \frac{2 C_0 C_{\alpha,p,d} \max \left\{ 1, e^{2 t \omega } \right\}}{1- \alpha } t^{1-\alpha}  \rho \Vert u -v \Vert_{V} \text{ for } t >0. 
\end{displaymath} 
\end{proof}

The result to follow is a straightforward application of Theorem~1.1 of \cite{chen97}, which
establishes the existence (and uniqueness) of a globally-defined invariant manifold and an associated invariant foliation for the cut-off semiflow $\varphi^\rho$.
The constructed invariant manifold corresponds to the rightmost eigenvalues of $\mathcal{A}_U$ (those contained in $\Sigma_1$) and is characterized by Lyapunov exponents of the backwards orbits.

Let us first introduce shorthand notation for forward/backward orbits.
For the Navier-Stokes semiflow given by $\varphi_t$, we denote the forward orbit of a point $u \in V$ as
\begin{displaymath}
   \mathcal{O}^+(u) : = \left\{ \, \varphi_t (u) \; \vert \; t \geq 0 \, \right\}.
\end{displaymath}
For the full 3D Navier-Stokes system, whenever writing the orbit notation, we are implicitly assuming that the solution starting from $u$
exists for all $t \geq 0$.
Backward orbits can be defined analogously -- their existence, however, is not guaranteed for any semiflow (not even for $\varphi_t^\rho$, where global forward existence is known).
If, for $u \in V$, $\varphi_t (u)$ can be defined for all $t \leq 0$ such that  $\varphi_t \circ \varphi_s (u) = \varphi_{t+s} (u)$ continues to hold over all of $\mathbb{R}$,\footnote{$\mathcal{O}^-(u)$ is uniquely defined for the case considered here: Indeed, for a time-independent forcing, solutions to the Navier-Stokes system are analytic in time on the domain of their existence (see Corollary~3.4.6 and the discussion on page~208 of \cite{henry1981}).}
\begin{displaymath}
   \mathcal{O}^-(u) : = \left\{ \, \varphi_t (u) \; \vert \; t \leq 0 \, \right\}
\end{displaymath}
is called a negative semiorbit. 
The full orbit is defined as $\mathcal{O}(u) : = \mathcal{O}^+(u) \cup \mathcal{O}^-(u)$.

Similarly, we denote by 
\begin{displaymath}
   \mathcal{O}^+_\rho(u) : = \left\{ \, \varphi_t^\rho (u) \; \vert \; t \geq 0 \, \right\}
\end{displaymath}
and 
\begin{displaymath}
   \mathcal{O}^-_\rho(u) : = \left\{ \, \varphi_t^\rho (u) \; \vert \; t \leq 0 \, \right\}
\end{displaymath}
the corresponding sets for the cut-off system.

\begin{figure}
\includegraphics[width=0.8\textwidth]{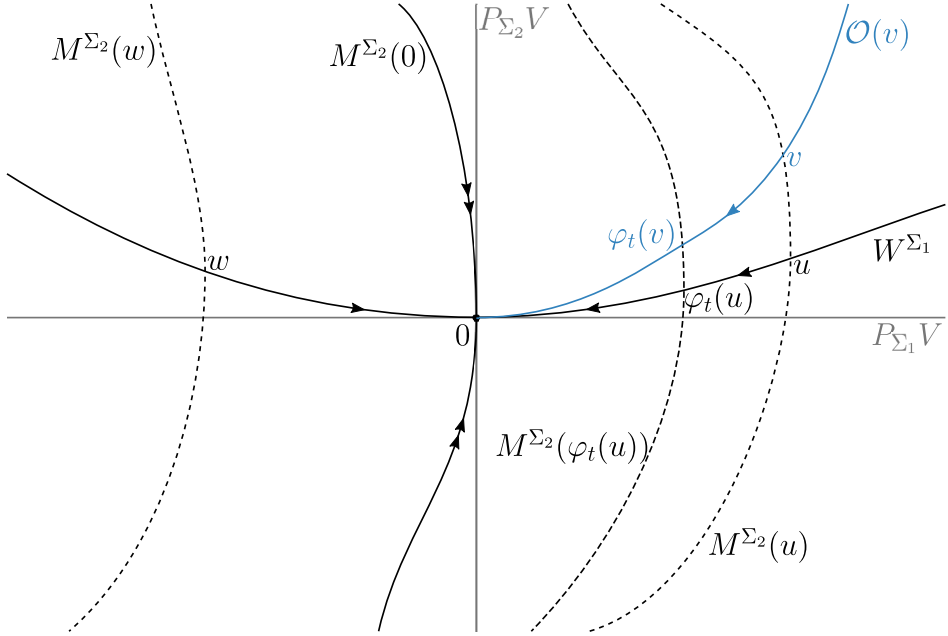}
\centering
\caption{Sketch of what the results of Theorem~\ref{thm:main} imply 
in the special case of a stable equilibrium,
for an operator $\mathcal{A}_U$ with spectrum according to Figure~\ref{fig:spectrum_init}.
Here, $W^{\Sigma_1}$ is a spectral submanifold corresponding to $\Sigma_1$ and $M^{\Sigma_2}$ is the associated foliation. 
Due to the structure of the spectrum in Figure~\ref{fig:spectrum_init}, $M^{\Sigma_2}(0)$ is the strong-stable manifold of $0$ (equivalent to $U$ of the original system), and $W^{\Sigma_1}$ is a a slow manifold (cf.\ Remark~\ref{remark:slow_mfd}).
}
\label{fig:thm}
\end{figure}

\begin{theorem}
\label{thm:main}
Let $\varphi_t^\rho$ denote the cut-off Navier-Stokes semiflow \eqref{eq:cutoff_semiflow}.
Take $\Sigma_1$, $\Sigma_2$, $\beta_1$, $\beta_2$ as in \eqref{eq:sugma}, let $P_{\Sigma_1}$, $P_{\Sigma_2}$ denote the associated projections,
and take $\rho>0$ sufficiently small. 
Then there is a globally Lipschitz and $C^1$ map 
$\phi_\rho : P_{\Sigma_1}V \to P_{\Sigma_2}V$ with $\phi_\rho(0) = 0$, $D\phi_\rho (0) = 0$, such that the submanifold 
\begin{displaymath}
W^{\Sigma_1}_\rho : = \mathrm{graph} ( \phi_\rho) = \left\{ \, u_1 + \phi_\rho(u_1) \; \vert \; u_1 \in P_{\Sigma_1}V  \, \right\}
\end{displaymath}
of $V$ satisfies the following properties:
\begin{enumerate}[label=\upshape{(\roman*)}]
  \item \upshape \textbf{(Invariance)} \itshape The restriction of the semiflow $\varphi_t^\rho$ to $W^{\Sigma_1}_\rho$ leaves the submanifold invariant and can be uniquely extended to a Lipschitz flow $\{\varphi_t^\rho \vert_{W^{\Sigma_1}_\rho} \}_{t \in \mathbb{R}}$ defined for all times.
  \item \upshape \textbf{(Lyapunov exponent)} \itshape Negative semiorbits
  of $\varphi_t^\rho $ are characterized by the following property.
  If $\mathcal{O}_\rho^-(u) \subset W^{\Sigma_1}_\rho$, then
  \begin{displaymath}
    \limsup_{t \to -\infty} \frac{1}{|t|} \ln \Vert \varphi^\rho_t (u) \Vert_V \leq - \beta_1
  \end{displaymath}
  Conversely, if $\mathcal{O}_\rho^-(u) \subset V$ exists and satisfies
  \begin{displaymath}
    \limsup_{t \to -\infty} \frac{1}{|t|} \ln \Vert \varphi^\rho_t (u) \Vert_V < - \beta_2
  \end{displaymath}
  then $\mathcal{O}_\rho^-(u) \subset W^{\Sigma_1}_\rho$.

  \item \label{mainthm:stat3} \upshape \textbf{(Invariant foliation)} \itshape There exists a continuous map $\psi_\rho : V \times P_{\Sigma_2} V \to P_{\Sigma_1}V$ such that for each $u \in W^{\Sigma_1}_\rho$, $\psi_\rho(u,P_{\Sigma_2} u ) = P_{\Sigma_1} u$ and the manifold $M^{\Sigma_2}_\rho(u) = \{ \, \psi_\rho(u,v_2)+v_2 \; | \; v_2 \in P_{\Sigma_2}V \, \}$ passing through $u$ satisfies
  \begin{gather}
      \varphi_t^\rho ( M^{\Sigma_2}_\rho(u) ) \subset M^{\Sigma_2}_\rho(\varphi_t^\rho(u)), \qquad t \geq 0, \label{eq:fibrepreserving} \\
      M^{\Sigma_2}_\rho(u) = \left\{ \, v \in V \; \Big\vert \; \limsup_{t \to \infty} \frac1t \ln \Vert \varphi_t^\rho(u) - \varphi_t^\rho (v) \Vert_V \leq  \beta_2 \,  \right\}. \label{eq:Foliationdecay}
  \end{gather}
  Moreover, $\psi_\rho : V \times P_{\Sigma_2} V \to P_{\Sigma_1}V$ is uniformly Lipschitz and
  $C^1$ with respect to its second argument. 
  
  \item \label{mainthm:stat4} \upshape \textbf{(Completeness)} \itshape For every $u \in V$, $M^{\Sigma_2}_\rho(u) \cap W^{\Sigma_1}_\rho$ is a single point. 
  In particular, 
  \begin{displaymath}
    M^{\Sigma_2}_\rho(u) \cap M^{\Sigma_2}_\rho(v) = \emptyset \quad \text{for } u,v \in W^{\Sigma_1}_\rho, u \neq v; \qquad \bigcup_{u \in W^{\Sigma_1}_\rho} M^{\Sigma_2}_\rho(u) = V;
  \end{displaymath}
  that is, $\{ M^{\Sigma_2}_\rho(u)  \}_{u \in W^{\Sigma_1}_\rho}$ form a foliation of $V$ over $W^{\Sigma_1}_\rho$.
\end{enumerate}
\end{theorem}

\begin{proof}
We shall verify hypotheses (H.1)-(H.4) of \cite{chen97}
for the case of 
$\varphi_t^\rho$, recalled below.
\begin{enumerate}[label =(\textbf{H.\arabic*})]
  \item $\varphi_t^\rho (u)$ is continuous as a map  $\mathbb{R}^{>0} \times V \to V$ and there exists a constant $q > 0$ such that 
  \begin{displaymath}
    \sup_{0 \leq t \leq q} \mathrm{Lip}_V ( \varphi_t^\rho ) <\infty.
  \end{displaymath}
  
  \item There exists a $\tau \in (0,q]$ such that $\varphi_\tau^\rho$ can be decomposed as $\varphi_\tau^\rho = L + N$ where $L \in \mathcal{L}(V)$ and $N: V \to V$ is globally Lipschitz.
  
  \item There are subspaces $V_i \subset V$, $i =1,2$ and continuous projections $P_i : V \to V_i$ such that $P_1 + P_2 = I$, $V = V_1 \oplus V_2$, $L$ leaves $V_i$ invariant and $L$ commutes with $P_i$, $i = 1,2$. 
  Denoting by $L_i:V_i \to V_i$ the restrictions of $L$, $L_1$ has bounded inverse and there exist constants $\alpha_1 > e^{\beta_1} > e^{\beta_2} > \alpha_2 \geq 0$, $C_1 \geq 1$, and $C_2 \geq 1$ such that
  \begin{gather}
      \Vert L_1^{-k} P_1\Vert_{\mathcal{L}(V)} \leq C_1 \alpha_1^{-k}, \qquad k \geq 0, \label{eq:requiredLbounds1}\\
      \Vert L_2^{k} P_2\Vert_{\mathcal{L}(V)} \leq C_2 \alpha_2^{k},  \qquad k \geq 0.
      \label{eq:requiredLbounds2}
  \end{gather}
  
  \item The operators $L$ and $N$ satisfy
  \begin{equation}
    \frac{(\sqrt{C_1}+\sqrt{C_2})^2}{\alpha_1 - \alpha_2} \mathrm{Lip}_V ( N ) < 1.
    \label{eq:lipschitz_requirement}
  \end{equation}
\end{enumerate}
Points (\textbf{H.1})-(\textbf{H.2}) were contents of Lemmas \ref{lemma:F} and \ref{lemma:lipsic2}, with $\tau = 1$, $L= e^{\mathcal{A}_U}$ and $N(u) = N_\rho(1,u)$ via definition \eqref{eq:Nrho}.
A decomposition such as the one required by (\textbf{H.3}) 
follows from \eqref{eq:Vdecomposition} with $V_i = P_{\Sigma_i} V$ and $L$ as above.

For $\varepsilon >0$ small enough such that $\beta_1+\varepsilon<\mu_1$ and $\beta_2 -\varepsilon > \mu_2$,
the restricted semigroup satisfies the bounds 
\begin{displaymath}
    \left\Vert \left( e^{ t \mathcal{A}_U }  P_{\Sigma_1} \right)^{-1} \right\Vert_{\mathcal{L}(V)} 
    \leq C_1 e^{- t (\beta_1 +\varepsilon) }
\end{displaymath}
and 
\begin{displaymath}
    \Vert e^{t \mathcal{A}_U } P_{\Sigma_2} \Vert_{\mathcal{L}(V)} \leq C_2 e^{t(\beta_2 - \varepsilon) }
\end{displaymath}
by the Hille-Yosida theorem,
  which imply \eqref{eq:requiredLbounds1}-\eqref{eq:requiredLbounds2} with $\alpha_1 = e^{ \beta_1 +\varepsilon }$ and $\alpha_2 = e^{\beta_2 - \varepsilon }$. 
 Condition (\textbf{H.4}) is achieved by adjusting $\mathrm{Lip} (N_\rho (1, \cdot))$ through $\rho$ appropriately (cf.~\eqref{eq:lipN}).
  With $\alpha_1$ and $\alpha_2$ as above, we may choose the bounds on Lyapunov exponents to be $\beta_1$ and $\beta_2$ (labeled as $\gamma_1$ and $\gamma_2$ in the conclusions of Theorem~1, \cite{chen97}).
\end{proof}

In the following, we shall assume that $\rho>0$ is selected such that
\begin{equation}
    \mathrm{Lip}_{P_{\Sigma_1}V \to P_{\Sigma_2}V} (\phi_\rho) <1 \qquad \text{and} \qquad \mathrm{Lip}_{P_{\Sigma_2}V \to P_{\Sigma_1}V} (\psi_\rho (v,\cdot)) <1
   \label{eq:lipschitz_assumption}
\end{equation}
hold for each $v\in V$.

The constructed manifolds are generally smoother than $C^1$ --
the question of smoothness, however, is postponed to the latter parts of this section, once the uniqueness issue is settled (see Remark~\ref{remark:smoothness}).

Note that 
the foliation in statements \ref{mainthm:stat3} and \ref{mainthm:stat4} 
may be alternatively characterized via a (nonlinear) projection map $\pi_\rho: V \to W^{\Sigma_1}_\rho$. 
Indeed, an application of the (parametric) Banach fixed point theorem to the map 
\begin{gather*}
     g: V \times P_{\Sigma_1} V \to  P_{\Sigma_1} V \\ 
    (v,v_1) \mapsto \psi_\rho(v,\phi_\rho(v_1))
\end{gather*}
shows that, for each $v \in V$, the map $g(v,\cdot) : P_{\Sigma_1} V \to P_{\Sigma_1} V$ has a unique fixed point.
Moreover, if we denote this fixed point by $f (v)$, then the map $ f: V \to P_{\Sigma_1} V$ is continuous.
Now, defining $\pi_\rho: V  \to  W^{\Sigma_1}_\rho$ as
\begin{equation}
   \pi_\rho(v): = f(v)+\phi_\rho \circ f (v)
   \label{eq:pi_rho}
\end{equation}
yields the desired projection map. 
In this terminology, $M^{\Sigma_2}_\rho(u) = \pi_\rho^{-1}(u)$ are the leaves of the foliation, and the following diagram commutes
\begin{equation}
\begin{tikzcd}
    V \arrow[r, "\varphi_t^\rho"] \arrow[d, "\pi_\rho"']
    & V \arrow[d, "\pi_\rho"] \\
    W^{\Sigma_1}_\rho \arrow[r, "\varphi_t^\rho \vert_{W^{\Sigma_1}_\rho}"']
    &  W^{\Sigma_1}_\rho
    \label{eq:fibrepreserving2}
\end{tikzcd}
\end{equation}
for $t \geq 0$ according to \eqref{eq:fibrepreserving}.

Theorem~\ref{thm:main} gives global and unique $C^1$ submanifolds of $V$ with the desired properties -- but only for the modified system governed by $\varphi^\rho$.
The goal in the following segment is to extend these results as best as possible to the full semiflow $\varphi$,
making use of the fact that $\varphi \equiv \varphi^\rho$ on $B_{\rho/2}(0)$.
As the foregoing argument already indicates,
the majority of results deduced for the full system $\varphi$ will only hold locally, in a small neighborhood of the origin (i.e., the base state $U \in V$).

The following corollary of Theorem~\ref{thm:main} shows that spectral submanifolds $W^{\beta}_{\mathrm{loc},r}$ (Definition~\ref{def:wloc}) and their corresponding foliations $M^\beta_{\mathrm{loc},r}$ (Definition~\ref{def:foliation})
exist for each $\beta \in \mathbb{R}$ bisecting the spectrum, 
in the present context for the Navier-Stokes semiflow. Uniqueness, however, is not guaranteed.

\begin{corollary}
\label{cor:Wloc}
Let $\varphi_t$ denote the semiflow generated by the Navier-Stokes equations on $V$, and let $\mathcal{A}_U$ be as above.
For any $\beta < \max \mathrm{Re} \, \sigma ( \mathcal{A}_U ),$  $\beta \notin \mathrm{Re} \, \sigma ( \mathcal{A}_U )$  
and for $r>0$ small enough, there exists 
a locally invariant manifold $W^{\Sigma_1}_{\mathrm{loc},r} \subset \widehat{B}_r(0)$ 
with the following properties: 
\begin{enumerate}[label=\upshape{(\roman*)}]
    \item \label{cor1:stat1} $W^{\Sigma_1}_{\mathrm{loc},r}$ is a $C^1$ spectral submanifold corresponding to the spectral subset $\Sigma_1 : = \{ \lambda \in \sigma (\mathcal{A}_U) \, | \, \mathrm{Re} \, \lambda > \beta \}$, in the sense of Definition~\ref{def:wloc}.
    \item \label{cor1:stat2} 
    If, for $u \in V$, $\varphi_t (u )$ remains defined for all $t \geq 0$ and $\mathcal{O}^+(u) \subset B_{r_1}(0)$ (with $r_1>0$ small enough), then there exists a unique $v \in W_{\Sigma_1}^{\mathrm{loc},r}$ such that  $\mathcal{O}^+(v) \subset W_{\Sigma_1}^{\mathrm{loc},r}$ and 
    \begin{equation}
        \Vert \varphi_t (u) - \varphi_t (v) \Vert_V \leq C e^{\beta t}, \qquad  t \geq 0.
        \label{eq:cor1_2}
    \end{equation}
\end{enumerate}
Moreover, there exists a $C^0 \times C^1$ foliation for $\widehat{B}_r(0)$ corresponding to the spectral subset $\Sigma_2 = \sigma (\mathcal{A}_U) \setminus \Sigma_1$ (as in Definition~\ref{def:foliation})
that may be represented by a map ${\psi}:\widehat{B}_r(0) \times B^2_{r} (0) \to B^1_{r_1}(0)$ for some $r_1>0$, such that $M_{\mathrm{loc},r}^{\Sigma_2}(u)=\mathrm{graph}({\psi}(u,\cdot)) \cap \widehat{B}_r(0)$ and $\mathrm{Lip}_{P_{\Sigma_2}V \to P_{\Sigma_1}V}({\psi}(u,\cdot))<1$ for each $u \in \widehat{B}_r(0)$.

\end{corollary}

\begin{proof}
We can choose $\mu_2$ and $\mu_1$ to be elements of $\mathrm{Re} \, \sigma ( \mathcal{A}_U )$ adjacent to $\beta$, so that $ \mu_2 < \beta < \mu_1 $ and the spectrum splits into a disjoint union of $\Sigma_1$ and $\Sigma_2$ as in \eqref{eq:sugma},
by the assumptions on $\beta$ and the discreteness of $\sigma ( \mathcal{A}_U ) $. 
Take a cutoff function $\chi_\rho$ such that the assumptions of Theorem~\ref{thm:main} are met. For any $r < \rho/4$, we
define 
\begin{equation}
    W^{\Sigma_1}_{\mathrm{loc},r}: = W^{\Sigma_1}_\rho \cap \widehat{B}_r(0);
    \label{eq:Wlocdef}
\end{equation}
now statement \ref{cor1:stat1} follows, since $\varphi \equiv \varphi^\rho$ on $\widehat{B}_r(0)$.

With $\chi_\rho$ fixed as above, set $v : = \pi_\rho (u)$, with $\pi_\rho:V \to W^{\Sigma_1}_\rho$ as in \eqref{eq:pi_rho}. 
Using continuity of $\pi_\rho$ at $0$, we may choose $r_1 \in (0,r)$ such that $\Vert \varphi_t^\rho \circ \pi_\rho  (u) \Vert_b = \Vert \pi_\rho \circ \varphi_t^\rho (u) \Vert_b \leq r$ whenever $ \varphi_t (u)$ and thus $\varphi_t^\rho (u)$ is contained in $B_{r_1}(0)$.
This also implies that $\mathcal{O}^+(v) = \mathcal{O}_\rho^+(v) \subset \widehat{B}_r(0)$.
Clearly, $v \in W^{\Sigma_1}_{\mathrm{loc},r}$ and hence, by statement \ref{cor1:stat1}, 
$\mathcal{O}^+(v) \subset W^{\Sigma_1}_{\mathrm{loc},r}$.
Equation \eqref{eq:cor1_2} follows from the characterization \eqref{eq:Foliationdecay} of $\pi_\rho^{-1}(v)$ and the choice of $\mu_2$, since both $\mathcal{O}^+(v)$ and $\mathcal{O}^+(u)$ are contained in $\widehat{B}_r(0)$ where $\varphi$ and $\varphi^\rho$ agree.

As for the 'moreover' part of the statement, it is clear that the properties outlined in Definition~\ref{def:foliation} hold if the local family of manifolds is defined by
\begin{displaymath}
   M_{\mathrm{loc},r}^{\Sigma_2}(u) := M_\rho^{\Sigma_2}(u) \cap \widehat{B}_r(0),
\end{displaymath}
with $r$ chosen as above. 
In particular, the map $\psi$ that characterizes the local foliation is given by simply restricting $\psi_\rho$ to the appropriate domain -- the Lipschitz condition follows from \eqref{eq:lipschitz_assumption}.
\end{proof}

The most important aspect of the above corollary - from the point of view of model reduction - is that, subject to adjustments to $\Sigma_1$, \eqref{eq:cor1_2} allows for control over the decay rate towards the invariant manifold $W^{\Sigma_1}_{\mathrm{loc},r}$ of trajectories that remain in a neighborhood of the base state.

The source of non-uniqueness of $W^{\Sigma_1}_{\mathrm{loc},r}$ stems from its dependence on the choice of cutoff function, which crucially affects $W^{\Sigma_1}_\rho$
in 
its definition \eqref{eq:Wlocdef} (see \cite{Sijbrand85} for a much more detailed discussion).
It is clear, however, that
once a cutoff function is fixed, the manifold $W^{\Sigma_1}_{\mathrm{loc},r}$ is uniquely determined by \eqref{eq:Wlocdef}.
In fact, the converse is also true:
If $W^{\Sigma_1}_{\mathrm{loc},r}$ is any manifold satisfying Definition~\ref{def:wloc}
then there exists a cutoff function $\chi_\rho$ such that $W^{\Sigma_1}_{\mathrm{loc},r} = W^{\Sigma_1}_\rho \cap \widehat{B}_r(0)$ (this fact is discussed for the case of center manifolds in \cite{Sijbrand85}).
In particular, for any given $W^{\Sigma_1}_{\mathrm{loc},r}$, we may conclude that Corollary~\ref{cor:Wloc}\ref{cor1:stat2} holds with $\beta > \sup_{\lambda \in \Sigma_2} \; \mathrm{Re} \lambda$.

\subsection{The question of uniqueness}
\label{sect:uniqueness}

Whenever invariant manifolds are constructed with the objective of being used as reduced order models,
uniqueness
becomes a crucial 
aspect of discussion.
Indeed, if two or more manifolds exist with the same defining properties, 
the dynamics they predict on a global scale may be vastly different (see the introduction to \cite{Haller2016}, for instance).
Therefore, in this section,
we take the time to give a 
detailed account of the uniqueness issue.
A classical reference for such considerations for finite dimensional manifolds is the work of Hirsch, Pugh and Shub \cite{hirsch1970invariant}.

It is also worth noting that the results to follow are independent of the specific PDE studied, as long as $\varphi_t$ is a sufficiently smooth semiflow and its cut-off version $\varphi_t^\rho$ satisfies the assumptions of Theorem~\ref{thm:main}.

In preparation, 
let us give a few different characterizations of $M_\rho^{\Sigma_2}(u)$ and $W^{\Sigma_1}_\rho$.

\begin{lemma}
\label{lemma:mfd_charact}
Let $\varphi_t^\rho$ denote semiflow on $V$ associated to the cut-off dynamical system and let $F_\rho$, $\mathcal{A}_U$, $\mu_1$, $\mu_2$ be as above.
Assume that $\rho>0$ is chosen such that $\mathrm{Lip}_{V \to L^p_\sigma} ( F_\rho)$ is sufficiently small, and fix $u \in V$.
Then, the following characterizations of the leaf 
of the foliation passing through $u$ are equivalent
\begin{align*}
    M_\rho^{\Sigma_2}(u) &= \left\{ \, v \in V \, \Big\vert \, \limsup_{t \to \infty} \frac{1}{t} \ln{ \Vert \varphi_t^\rho (u ) - \varphi_t^\rho (v )\Vert_V} \leq \mu, \text{ for some } \mu \in [\mu_2,\mu_1) \, \right\} \\
    &= \left\{ \, v \in V \, \Big\vert \, \limsup_{t \to \infty} \frac{1}{t} \ln{ \Vert \varphi_t^\rho (u ) - \varphi_t^\rho (v )\Vert_V} \leq \mu, \text{ for all } \mu \in [\mu_2,\mu_1) \, \right\} \\
    &= \left\{ \, v \in V \, \Big\vert \, \Vert \varphi_t^\rho (u ) - \varphi_t^\rho (v ) \Vert_V \leq C e^{\mu t} \Vert u-v \Vert_V \text{ for all } t \geq 0 \text{ and } \mu \in (\mu_2,\mu_1) \, \right\}.
\end{align*}
\end{lemma}

\begin{proof}
We first note that the first two sets on the right hand side are equivalent.
Recall that Theorem~\ref{thm:main}\ref{mainthm:stat3} provides a continuous map $\psi_\rho : V \times P_{\Sigma_2}V \to P_{\Sigma_1}V$ such that
\begin{equation}
    \mathrm{graph}(\psi_\rho(u,\cdot)) = \left\{ \, v \in V \, \Big\vert \, \limsup_{t \to \infty} \frac{1}{t} \ln{ \Vert \varphi_t^\rho (u ) - \varphi_t^\rho (v )\Vert_V} \leq \mu\, \right\},
    \label{eq:foliation_lem}
\end{equation}
with $\mathrm{Lip}_{P_{\Sigma_2}V \to P_{\Sigma_1}V} (\psi_\rho (u,\cdot)) <1$ (up to possibly shrinking $\rho>0$) for all $u \in V$.
The construction does not depend on the choice of $\mu$ (in place of $\beta_2$ of Theorem~\ref{thm:main}), and hence \eqref{eq:foliation_lem} continues to hold over $\mu \in (\mu_2,\mu_1)$, and thus also for $\mu = \mu_2$.

Next, we show that the first set on the right hand side is contained in the third one.
Fix $\psi_\rho$ as above and take $v \in V$ from the first set.
Given that $\Vert P_{\Sigma_2} e^{\mathcal{A}_U t} \Vert_{\mathcal{L}(V)} \leq C e^{(\mu_2+\varepsilon)t}$ holds for any $\varepsilon>0$,
we may take $0<\varepsilon < \mu - \mu_2$ so that 
\begin{multline*}   
\left\Vert e^{-\mu t} P_{\Sigma_2} \left( \varphi_t^\rho(u) - \varphi_t^\rho(v) \right) \right\Vert_V \leq C_0 e^{(\mu_2 + \varepsilon - \mu)t} \Vert u - v \Vert_V \\
+ e^{-\mu t} \int_0^t C_\alpha e^{(\mu_2+\varepsilon) (t-s)} (t-s)^{-\alpha} \mathrm{Lip}_{V \to L^p_\sigma} ( F_\rho) e^{\mu s} \left\Vert e^{-\mu s}  \left(  \varphi_s^\rho(u) - \varphi_s^\rho(v) \right) \right\Vert_V ds
\end{multline*}
is satisfied.
Now, since $v \in M^{\Sigma_2}_\rho (u)$, we have that $\varphi^\rho_t(v) \in M^{\Sigma_2}_\rho (\varphi_t^\rho(u))$ and hence
\begin{align*}
    \left\Vert  \varphi_s^\rho(u) - \varphi_s^\rho(v)  \right\Vert_V &\leq (1+ \mathrm{Lip}_{P_{\Sigma_2}V \to P_{\Sigma_1}V} (\psi_\rho (\varphi_t^\rho(u),\cdot)))  \left\Vert P_{\Sigma_2}  \left(  \varphi_s^\rho(u) - \varphi_s^\rho(v) \right) \right\Vert_V \\
    &< 2 \left\Vert P_{\Sigma_2}  \left(  \varphi_s^\rho(u) - \varphi_s^\rho(v) \right) \right\Vert_V.
\end{align*}
Adjusting $\rho>0$ further so that
\begin{displaymath}
  k_\rho := 2 \mathrm{Lip}_{V \to L^p_\sigma} ( F_\rho) \int_0^\infty C_\alpha e^{(\mu_2+\varepsilon-\mu) y} y^{-\alpha} dy < 1
\end{displaymath}
holds, we obtain
 \begin{displaymath}
  \sup_{t \in (0,T)}     \left\Vert e^{-\mu t} P_{\Sigma_2} \left( \varphi_t^\rho(v) - \varphi_t^\rho(w) \right) \right\Vert_V \leq C_0  (1-k_\rho )^{-1} \Vert v - w \Vert_V.
\end{displaymath}
Since the right hand side is independent of $T$, we may take $T \to \infty$,
which yields 
 \begin{align*}
     \left\Vert  \varphi_s^\rho(u) - \varphi_s^\rho(v)  \right\Vert_V &< 2 \left\Vert  P_{\Sigma_2} \left( \varphi_t^\rho(v) - \varphi_t^\rho(w) \right) \right\Vert_V \\ &\leq 2 C e^{\mu t} \Vert v - w \Vert_V, \qquad \qquad t \geq 0,
\end{align*}
as desired.

Lastly, if $v$ is an element of the third set, then
\begin{displaymath}
   \limsup_{t \to \infty} \frac{1}{t} \ln \left\Vert \varphi_t^\rho (u) -  \varphi_t^\rho(v) \right\Vert_V \leq \mu
\end{displaymath}
holds for all $\mu \in (\mu_2,\mu_1)$, and thus $v$ is also contained in the first set.
\end{proof}

The corresponding result for $W^{\Sigma_1}_\rho$ is derived analogously, and hence is stated here without proof.

\begin{lemma}
\label{lemma:mfd1_charact}
Let $\varphi_t^\rho$ denote semiflow on $V$ associated to the cut-off dynamical system and let $F_\rho$, $\mathcal{A}_U$, $\mu_1$, $\mu_2$ be as above. 
Assume that $\rho>0$ is chosen such that $\mathrm{Lip}_{V \to L^p_\sigma} ( F_\rho)$ is sufficiently small.
Then the following characterizations of the constructed manifold $W^{\Sigma_1}_{\rho}$ are equivalent
\begin{align*}
    W^{\Sigma_1}_{\rho} &= \left\{ \, v \in V \, \Big\vert \, \limsup_{t \to -\infty} \frac{1}{|t|} \ln{ \Vert \varphi_t^\rho (v ) \Vert_V} \leq \mu, \text{ for all } \mu \in [-\mu_1,-\mu_2) \, \right\} \\
    &= \left\{ \, v \in V \, \Big\vert \, \limsup_{t \to -\infty} \frac{1}{|t|} \ln{ \Vert \varphi_t^\rho (v ) \Vert_V} \leq \mu, \text{ for some } \mu \in [-\mu_1,-\mu_2) \, \right\} \\
    &= \left\{ \, v \in V \, \Big\vert \, \Vert \varphi_t^\rho (v ) \Vert_V \leq C e^{-\mu t} \Vert v \Vert_V \text{ for all } t \leq 0 \text{ and } \mu \in (-\mu_1,-\mu_2) \, \right\}.
\end{align*}
\end{lemma}

\begin{proof}
An argument analogous to the proof of the preceding lemma applies.
\end{proof}

We now turn to outline two special cases
in which the spectral submanifolds $W^{\Sigma_1}_{\mathrm{loc},r}$  and foliations $M^{\Sigma_2}_{\mathrm{loc},r}$ constructed for the full Navier-Stokes system  $\varphi_t$ (from Corollary~\ref{cor:Wloc}) are unique.
In particular, the following lemma shows that the
foliation for $\varphi_t^\rho$, generated by $\psi_\rho$, yields a unique foliation for $\varphi_t$ over forward orbits that remain in a neighborhood of the origin, provided that $\Sigma_2$ is contained in the left half-plane of $\mathbb{C}$ (equivalently if $\beta < 0$).
In this case, the leaf of the foliation passing through the origin, $M^{\Sigma_2}_{\mathrm{loc},r}(0)$, represents what is commonly termed a strong-stable manifold.

\begin{lemma} 
\label{lemma:Mu_loc}
Let $\varphi_t$ denote the semiflow generated by the Navier-Stokes equations on $V$, and let $\mathcal{A}_U$, $\mu_2, \mu_1$ be as above.
Suppose the forward orbit of $u \in V$ remains in a sufficiently small neighborhood of the origin, i.e.\ $\mathcal{O}^+(u) \subset B_{r_2}(0)$, with $r_2>0$ to be determined in the course of the proof.
Assume further that $\mu_2 < 0$.
Then, for $r>0$ small enough, 
the family of submanifolds constructed in the last part of Corollary~\ref{cor:Wloc} are uniquely determined over $\mathcal{O}^+(u)$.
More precisely,
there exist unique $C^1$ submanifolds embedded in $V$ and parameterized by $u \in V$ of codimension $\mathrm{dim} (P_{\Sigma_1}V)$, 
\begin{align*}
    M_{\mathrm{loc},r}^{\Sigma_2}(u)     &= \left\{ \, v \in \widehat{B}_r(u) \, \Big\vert \, \limsup_{t \to \infty} \frac{1}{t} \ln{ \Vert \varphi_t (u ) - \varphi_t (v )\Vert_V} \leq \mu, \text{ for some } \mu \in [\mu_2,\mu_1) \, \right\} \\
    &= \left\{ \, v \in \widehat{B}_r(u) \, \Big\vert \, \limsup_{t \to \infty} \frac{1}{t} \ln{ \Vert \varphi_t (u ) - \varphi_t (v )\Vert_V} \leq \mu, \text{ for all } \mu \in [\mu_2,\mu_1) \, \right\} \\
    &= \left\{ \, v \in \widehat{B}_r(u) \, \Big\vert \, \Vert \varphi_t (u ) - \varphi_t (v ) \Vert_V \leq C e^{\mu t} \Vert u-v \Vert_V \text{ for all } t \geq 0 \text{ and } \mu \in (\mu_2,\mu_1) \, \right\},
\end{align*}
that form a local foliation over $\mathcal{O}^+(u)$.
In particular, if $r>0$ is chosen as above, then
\begin{displaymath}
   \varphi_t ( M_{\mathrm{loc},r}^{\Sigma_2}(u) ) \cap \widehat{B}_r(0) \subset M_{\mathrm{loc},r}^{\Sigma_2}(\varphi_t(u)), \qquad \text{for all } t \geq 0.
\end{displaymath}
\end{lemma}

\begin{proof}
Pick $\rho>0$ and a cutoff function $\chi_\rho$ such that the conditions of both Theorem~\ref{thm:main} and Lemma~\ref{lemma:mfd_charact} are met.
Rewriting the invariant foliation of points \ref{mainthm:stat3} and \ref{mainthm:stat4}, $\psi_\rho : V \times P_{\Sigma_2} V \to P_{\Sigma_1}V$, we have a $C^1$ embedding $i_u^\rho: P_{\Sigma_2}V \to V$ for each $u \in V$ given by 
\begin{displaymath}
   i_u^\rho(v_2) = (\psi_\rho(u,v_2),v_2),
\end{displaymath}
so that $\mathrm{im}(i_u^\rho) = M_\rho^{\Sigma_2}(u)$.
By \eqref{eq:lipschitz_assumption}, we have that $\mathrm{Lip}_{P_{\Sigma_2}V \to P_{\Sigma_1}V}(\psi_\rho(u,\cdot)) < 1$, and thus
\begin{displaymath}
   i_u^\rho \big({B}^2_r(u) \big) = \widehat{B}_r(u) \cap M_\rho^{\Sigma_2}(u),
\end{displaymath}
where we have also used that $\psi_\rho(u,P_{\Sigma_2}u) = P_{\Sigma_1}u$.

We claim that 
\begin{equation}
  M_{\mathrm{loc},r}^{\Sigma_2}(u) = \widehat{B}_r(u) \cap M_\rho^{\Sigma_2}(u),
  \label{eq:MulocBR}
\end{equation}
for $r>0$ small enough.
Indeed, by the third characterization of $M_\rho^{\Sigma_2}(u)$ we have that
\begin{displaymath}
   \Vert \varphi_t^\rho (u ) - \varphi_t^\rho (v ) \Vert_V \leq C e^{\mu t} \Vert u-v \Vert_V
\end{displaymath}
for $v \in M_\rho^{\Sigma_2}(u)$, for some $\mu_2 < \mu <0$.
Take $r_2 \in (0,\rho/2)$ -- we assume $\mathcal{O}^+(u) \subset B_{r_2}(0)$.
If 
$r>0$ is chosen such that $Cr < \rho/4-r_2/2$, then
\begin{displaymath}
   \left\Vert  \varphi_t^\rho (v ) \right\Vert_V 
   \leq \left\Vert \varphi_t^\rho (u) \right\Vert_V +  C e^{\mu t} \Vert u-v \Vert_V 
   \leq  \left\Vert \varphi_t^\rho (u) \right\Vert_V + 2 C e^{\mu t} \Vert u-v \Vert_b,
\end{displaymath}
and thus, if  $v \in \widehat{B}_r(u) \cap M_\rho^{\Sigma_2}(u)$, then $\mathcal{O}_\rho^+(v) \subset B_{\rho/2}(0)$.
Consequently, since $\varphi^\rho \equiv \varphi$ on $B_{\rho/2}(0)$, $\mathcal{O}^+(v) \subset B_{\rho/2}(0)$ and the
characterizations in Lemma~\ref{lemma:mfd_charact} continue to hold for the original semiflow $\varphi_t$ locally.
This implies that
\begin{displaymath}
   M_{\mathrm{loc},r}^{\Sigma_2}(u)  \supset \widehat{B}_r(u) \cap M_\rho^{\Sigma_2}(u);
\end{displaymath}
the converse is obvious, using the third characterization of $M_{\mathrm{loc},r}^{\Sigma_2}(u) $ and that $\varphi \equiv \varphi^\rho$ on $B_{\rho/2}(0)$ once again.
Thus \eqref{eq:MulocBR} follows, and hence
\begin{displaymath}
   M_{\mathrm{loc},r}^{\Sigma_2}(u)  = i_u^\rho \big({B}^2_r(u) \big),
\end{displaymath}
which shows that $M_{\mathrm{loc},r}^{\Sigma_2}(u) $ is an embedded $C^1$ submanifold of $V$ diffeomorphic to the ball ${B}^2_r(u)$.
\end{proof}

The counterpart of Lemma~\ref{lemma:Mu_loc} for the main manifold of interest shows that $W^{\Sigma_1}_{\mathrm{loc},r}$ 
is uniquely determined whenever $\Sigma_1$ is contained in the right half-plane of $\mathbb{C}$ (or $\beta >0$).
In this case, the constructed submanifold is customarily referred to as a strong-unstable manifold, since it comprises trajectories that 
depart from the origin at the fastest rates.

\begin{lemma}
\label{lemma:uniq_specialcases}
Let $\varphi_t$ denote the semiflow generated by the Navier-Stokes equations on $V$, and let $\mathcal{A}_U$, $\mu_1$ be as above.
Assume, additionally, that $\mu_1 > 0$.
Then, for $r>0$ small enough, the $C^1$ spectral submanifold constructed in Corollary~\ref{cor:Wloc} can be alternatively characterized as
\begin{align*}
    W^{\Sigma_1}_{\mathrm{loc},r}     &= \left\{ \, v \in \widehat{B}_r(0) \, \Big\vert \, \limsup_{t \to -\infty} \frac{1}{|t|} \ln{ \Vert \varphi_t(v ) \Vert_V} \leq \mu, \text{ for some } \mu \in [-\mu_1,-\mu_2) \, \right\} \\
    &=  \left\{ \, v \in \widehat{B}_r(0) \, \Big\vert \, \limsup_{t \to -\infty} \frac{1}{|t|} \ln{ \Vert \varphi_t (v ) \Vert_V} \leq \mu, \text{ for all } \mu \in [-\mu_1,-\mu_2) \, \right\} \\
    &= \left\{ \, v \in \widehat{B}_r(0) \, \Big\vert \, \Vert \varphi_t (v ) \Vert_V \leq C e^{-\mu t} \Vert v \Vert_V \text{ for all } t \leq 0 \text{ and } \mu \in (-\mu_1,-\mu_2) \, \right\}.
\end{align*}
In particular, it is unique and negatively invariant in the following sense:
If $u \in W^{\Sigma_1}_{\mathrm{loc},r}$, there exists $T<0$ such that $\varphi_t(u) \in W^{\Sigma_1}_{\mathrm{loc},r}$ for all $t \leq T$.
\end{lemma}

\begin{proof}
Once again, the proof is completely analogous to that of the preceding lemma and is therefore safely omitted.
\end{proof}

Outside the realm of these special cases -
despite not being able to
guarantee uniqueness -
the local behaviour of the system along different submanifolds remains unique, in some weakened sense.
In particular, we show 
that any two local submanifolds satisfying Definition~\ref{def:wloc} 
have locally conjugate reduced dynamics,
in line with the work of Burchard et al.~\cite{burchard1992} for center manifolds.

\begin{lemma}
\label{lemma:conj}
Let $\varphi_t$ denote the semiflow generated by the Navier-Stokes equations on $V$. 
Suppose that $W^{\Sigma_1}_{\mathrm{loc},r}$ is non-unique, and denote by $W^{\Sigma_1,a}_{\mathrm{loc},r}$ and $W^{\Sigma_1,b}_{\mathrm{loc},r}$ any two such manifolds. 
Then,
there is a neighborhood $O \subset V$ of the origin and a homeomorphism
$h:W^{\Sigma_1,a}_{\mathrm{loc},r} \cap O \to W^{\Sigma_1,b}_{\mathrm{loc},r} \cap O$
satisfying
\begin{equation}
    \varphi_t \circ h (u) = h \circ \varphi_t (u)
    \label{eq:conjugacy}
\end{equation}
for all $u \in W^{\Sigma_1,a}_{\mathrm{loc},r} \cap O$ and all $t$ satisfying $\varphi_t (u) \in  W^{\Sigma_1,a}_{\mathrm{loc},r} \cap O$.
\end{lemma}

\begin{proof}
First note that, by point~\ref{def:wloc_2} of Definition~\ref{def:wloc}, we may choose $r_1>0$ small enough such that both $W^{\Sigma_1,a}_{\mathrm{loc},r} \cap \widehat{B}_{r_1}(0)$ and $W^{\Sigma_1,b}_{\mathrm{loc},r} \cap \widehat{B}_{r_1}(0)$
may be represented as graphs of functions $\phi_a,\phi_b:P_{\Sigma_1}V \to P_{\Sigma_2}V$ with Lipschitz constants less than one.

Using Corollary~\ref{cor:Wloc}, there exists a neighborhood $\widehat{B}_{r_0}(0) \subset \widehat{B}_{r_1}(0)$ for which a local foliation exists, in the form of a map ${\psi}:\widehat{B}_{r_0}(0) \times B^2_{r_0} (0) \to B^1_{\tilde{r}_0}(0)$, for some $\tilde{r}_0>0$,
with $\mathrm{Lip}_{P_{\Sigma_2}V \to P_{\Sigma_1}V}({\psi}(v,\cdot))<1$ for each $v \in \widehat{B}_{r_0}(0)$.
By further restricting the domain of its first entry, we can ensure that
\begin{equation}
\psi(v,\phi_a (B^1_{r_0}(0))) \subset B^1_{r_0}(0), \qquad \text{and} \qquad \psi(v,\phi_b (B^1_{r_0}(0))) \subset B^1_{r_0}(0)
\label{eq:psiphiab}
\end{equation}
are both satisfied for each $v \in \widehat{B}_{r_2}(0) \subset \widehat{B}_{r_0}(0)$,
so that the contraction maps given by
\begin{equation}
v_1 \mapsto \psi(v,\phi_a (v_1)) , \qquad \text{and} \qquad v_1 \mapsto \psi(v,\phi_b (v_1))
\label{eq:psiphiab2}
\end{equation}
have unique fixed points in $B^1_{r_0}(0) \subset P_{\Sigma_1}V$ that depend continuously on $v$.
Thus, we may proceed as in the case of the cut-off system and define continuous projections $\pi_a:\widehat{B}_{r_2}(0) \to W^{\Sigma_1,a}_{\mathrm{loc},r} \cap \widehat{B}_{r_0}(0)$ and $\pi_b:\widehat{B}_{r_2}(0) \to W^{\Sigma_1,b}_{\mathrm{loc},r} \cap \widehat{B}_{r_0}(0)$ analogously to \eqref{eq:pi_rho}.

The above, \eqref{eq:psiphiab} in particular, also implies that any leaf of the local foliation that passes through $\widehat{B}_{r_2}(0)$ will have exactly one intersection point with both $W^{\Sigma_1,b}_{\mathrm{loc},r}$ and $W^{\Sigma_1,b}_{\mathrm{loc},r}$
within $\widehat{B}_{r_0}(0)$.

Our candidates for the conjugacy are $f_a:= \left. \pi_a \right\vert_{\widehat{B}_{r_2}(0) \cap W^{\Sigma_1,b}_{\mathrm{loc},r}}$
and $f_b:= \left. \pi_b \right\vert_{\widehat{B}_{r_2}(0) \cap W^{\Sigma_1,a}_{\mathrm{loc},r}}$
-- their domains, however, need to be adjusted.
We claim that
\begin{displaymath}
   h:= \left. f_a \right\vert_{\mathrm{im}(f_b) \cap \widehat{B}_{r_2}(0)}
\end{displaymath}
gives the desired homeomorphism, with inverse given by $\left. f_b \right\vert_{\mathrm{im}(f_a) \cap \widehat{B}_{r_2}(0)}$.

Injectivity of $h:\mathrm{im}(f_b) \cap \widehat{B}_{r_2}(0) \to \mathrm{im}(f_a) \cap \widehat{B}_{r_2}(0)$ follows from the fact that $\pi_a^{-1}(u) \cap (  \widehat{B}_{r_2}(0) \cap W^{\Sigma_1,b}_{\mathrm{loc},r})$ is a single point for each $u \in \mathrm{im}(f_a) \cap \widehat{B}_{r_2}(0)$, using \eqref{eq:psiphiab2}.
To show surjectivity, take $v \in \mathrm{im}(f_a) \cap \widehat{B}_{r_2}(0)$.
Then, there exists $u \in \widehat{B}_{r_2}(0) \cap W^{\Sigma_1,b}_{\mathrm{loc},r}$ such that $f_a(u) = v$, and thus we need only show that $u \in \mathrm{im}(f_b)$ --
but this is clear since both $u$ and $v$ are on the same leaf within $\widehat{B}_{r_2}(0)$, and hence \eqref{eq:psiphiab2} implies that $f_b(v) = u$.
Therefore, $h$ is indeed a homeomorphism.

The neighborhood $O$ may be chosen arbitrarily subject to the constraint that it intersects the manifolds $W^{\Sigma_1,a}_{\mathrm{loc},r}$ and $W^{\Sigma_1,b}_{\mathrm{loc},r}$ precisely in the domains of $h$ and $h^{-1}$.
The conjugacy \eqref{eq:conjugacy} follows from the local invariance of the manifolds in question (cf.\ Corollary~\ref{cor:Wloc}), along with the foliation-preserving property of $\varphi_t$, similarly to \eqref{eq:fibrepreserving2}.
\end{proof}

The smoothness of the conjugacy map $h$ may be improved whenever $\psi_\rho$ is continuously Fréchet differentiable as a map $V \times P_{\Sigma_2}V \to P_{\Sigma_1}V$, upon imposing the implicit function theorem in place of the Banach fixed point theorem, wherever appropriate \cite{burchard1992}.

\begin{remark}[Smoothness]
\label{remark:smoothness}
In many cases, the smoothness of $W^{\Sigma_1}_{\mathrm{loc},r}$ and/or 
$M_{\mathrm{loc},r}^{\Sigma_2}(0)$ can be improved.
In the subcases where we could conclude uniqueness, $C^\infty$ smoothness of the corresponding manifold ensues \cite{henry1981}.
Problems for $W^{\Sigma_1}_{\mathrm{loc},r}$ (resp.\ $M_{\mathrm{loc},r}^{\Sigma_2}(0)$) arise whenever $\Sigma_1$ contains elements with negative real part (resp.\ $\Sigma_2$ crosses the imaginary axis), i.e., when $W^{\Sigma_1}_{\mathrm{loc},r}$ is the pseudo-unstable manifold (resp.\ $M_{\mathrm{loc},r}^{\Sigma_2}(0)$ is the pseudo-stable manifold).
In these instances, smoothness may be improved subject to spectral gap conditions, see, e.g., \cite{irwin1980new}. 
\end{remark}

\subsection{Globalization of invariant manifolds}

The standard procedure of constructing global invariant manifolds from local ones usually involves successive iterations of the time-one map $\varphi_1$.
For fluid flows on $\Omega \subset \mathbb{R}^3$, $\varphi_1$ might not even be defined throughout the local manifold, and hence, most of the following discussion concerns only the two-dimensional case.

The corresponding theory 
is fully worked out in \cite{henry1981} --
we summarize it here due to its importance from the perspective of model reduction.

For a local manifold $W_{\mathrm{loc}}$ that is negatively invariant, its global counterpart can be obtained as
\begin{equation}
   W = \mathcal{O}^+ \left( W_{\mathrm{loc}} \right) = \bigcup_{n \in \mathbb{N}} \varphi_n  \left( W_{\mathrm{loc}} \right).
   \label{eq:globMfd_fw}
\end{equation}
For $W$ to be an immersed submanifold of $V$, 
$\varphi_n$ needs to meet the requirements posed by the inverse function theorem (see, e.g., \cite{lang2001fundamentals}),
namely that $D\varphi_n(u)$ is injective and splits for all $u \in  W_{\mathrm{loc}}$.
Since $V$ is a Hilbert space, the latter condition is equivalent to $D\varphi_n(u)$ having a closed image -- satisfied if and only if $W_{\mathrm{loc}}$ is finite dimensional (using the closed image theorem and compactness of $D\varphi_n(u)$).
Injectivity of $\varphi_t$ and hence $D\varphi_n(u)$ is always given due to the time-analiticity of Navier-Stokes solutions (see, e.g., \cite{foias2001} or \cite{henry1981}).
Finally, it is clear from \eqref{eq:globMfd_fw} that the resulting manifold $W$ is  invariant under $\varphi_t$.

Conversely, if $W_{\mathrm{loc}}$ is positively invariant, the global version is given by
\begin{equation}
   W = \mathcal{O}^- \left( W_{\mathrm{loc}} \right) = \bigcup_{n \in \mathbb{N}} \varphi_{-n}  \left( W_{\mathrm{loc}} \right).
   \label{eq:globMfd_bw}
\end{equation}
To see that $W$ is immersed in $V$, we
appeal to the implicit function theorem, for which we need in addition that $\varphi_n$ is transversal over $W_{\mathrm{loc}}$ (in the sense of \cite{lang2001fundamentals}, p29).
This can be ensured if $W_{\mathrm{loc}}$ has finite codimension -- in this case, for each $u \in \varphi_{-n}  \left( W_{\mathrm{loc}} \right)$, the composite map 
\begin{displaymath}
   T_uV \xrightarrow{D \varphi_n(u)} T_{\varphi_n(u)}V \rightarrow T_{\varphi_n(u)}V \big/ T_{\varphi_n(u)} W_{\mathrm{loc}} 
\end{displaymath}
is surjective (since in general $D \varphi_n (u)$ has dense image, see \cite{henry1981}) and its kernel splits.

For instance, global strong stable and strong unstable manifolds may be produced upon
applying
\eqref{eq:globMfd_bw} and \eqref{eq:globMfd_fw} to
$M^{\Sigma_2}_{\mathrm{loc},r}(0)$ and $W^{\Sigma_1}_{\mathrm{loc},r}$
respectively.

\subsection{Summary of results}

In the following theorem, we summarize the results obtained concerning 
spectral submanifolds and foliations about fixed points.

\begin{theorem}[Summary]
\label{thm:summary}
Let $\varphi_t$ denote the semiflow generated by the Navier-Stokes system on $V$,
and let $\mathcal{A}_U$ be as above.
For any $\beta < \max \mathrm{Re} \, \sigma ( \mathcal{A}_U ),$  $\beta \notin \mathrm{Re} \, \sigma ( \mathcal{A}_U )$  
and for $r>0$ small enough,
there exist a spectral submanifold $W^{\beta}_{\mathrm{loc},r}$ and a complementing foliation $M^{\beta}_{\mathrm{loc},r}$ for $\widehat{B}_r(0)$, according to Definitions \ref{def:wloc} and \ref{def:foliation}.
The manifold $W^{\beta}_{\mathrm{loc},r}$
attracts trajectories within $\widehat{B}_r(0)$
at a rate of $\mathcal{O}(e^{\beta t})$ (in the sense of \eqref{eq:cor1_2}), synchronized along leaves of the foliation. 

Moreover, for $\beta>0$, $W^{\beta}_{\mathrm{loc},r}$ is unique and can be extended to a $C^\infty$ smooth global submanifold defined by the following property
\begin{displaymath}
     W^{\beta}=  \left\{ \, v \in V \, \Big\vert \, \limsup_{t \to -\infty} \frac{1}{|t|} \ln{ \Vert \varphi_t (v ) \Vert_V} \leq \beta \, \right\}.
\end{displaymath}
In all other cases,
smoothness (above $C^1$) is subject to further conditions (c.f.\ Remark~\ref{remark:smoothness}), and uniqueness is only guaranteed in a weakened sense:
 any two local spectral submanifolds have topologically conjugate reduced dynamics on some neighborhood of the origin (Lemma~\ref{lemma:conj}).

Similarly, for $\beta < 0$, the resulting foliation $M^{\beta}_{\mathrm{loc},r}$ is unique over trajectories that remain in a small enough neighborhood of the origin.
In particular, $M^{\beta}_{\mathrm{loc},r}(0)$ is unique, $C^\infty$ smooth and can be globalized to satisfy
\begin{displaymath}
     M^{\beta}(0) =  \left\{ \, v \in V \, \Big\vert \, \limsup_{t \to \infty} \frac{1}{t} \ln{ \Vert \varphi_t (v )\Vert_V} \leq \beta  \, \right\}.
\end{displaymath}
\end{theorem}

Slow manifolds of stable fixed points
are a particularly common case of interest for physical applications,
and hence warrant some further discussion.

\begin{remark}[Slow manifolds]
\label{remark:slow_mfd}
If $0 \in V$ is stable in the Lyapunov sense,
then Lemma~\ref{lemma:Mu_loc} guarantees the existence of a unique foliation for the whole of $\widehat{B}_{r_1}(0)$, where $r_1>0$ is chosen such that $u \in \widehat{B}_{r_1}(0)$ implies $\mathcal{O}^+(u) \subset \widehat{B}_{r_2}(0)$ (with $r_2$ as in Lemma~\ref{lemma:Mu_loc}).
This means that slow manifolds 
- despite not necessarily being unique - 
represent  \textit{all} nearby trajectories up to an exponential order discrepancy (if there are no other invariant sets in phase space, this property holds globally).
To be more precise,
for any given $\beta < 0$, $\beta \notin \mathrm{Re} \, \sigma(\mathcal{A}_U)$, 
the reduced dynamics on $W^{\beta}_{\mathrm{loc},r_1}$ 
- a finite system of ODEs with dimension equal to $\mathrm{dim}(\mathrm{im}(P_{\Sigma_1}))$ -
describe the long-term behaviour $\varphi_t(u)$, for any $u \in \widehat{B}_{r_1}(0)$, to order $\mathcal{O}(e^{\beta t})$.
\end{remark}

\subsection{Invariant manifolds around periodic orbits}

Periodic orbits - and invariant structures attached to them -
are perhaps the most studied class of invariant sets in the context of viscous fluids.
In channel flows, for instance, the stable manifold attached to the lower branch of travelling wave solutions (a special class of periodic orbits) serves as the separatrix between the basin of attraction of 
the laminar state and more complex states.
Considering their importance,
it is desirable to extend the above generalized
invariant manifold/foliation results
to the case of periodic orbits.

The strategy is to apply the corresponding
(map) version
of Theorem~\ref{thm:main}
to the Poincaré map,
as is standard procedure in the dynamical systems literature.
We first recall the notion of the Poincaré map
and confirm its existence in the present setting.

Assume now
that
there exists a non-constant periodic orbit $\Gamma = \mathcal{O}(u)$ with period $T$, defined as the minimal $T>0$ for which $\varphi_{T}(u) = u$ holds,
and let $S \subset V$ be a codimension one submanifold transverse to $\Gamma$ at $u$.
The Poincaré map of $\Gamma$ is a mapping $\mathbb{P}:O_0 \to O_1$ defined by
\begin{equation}
    \mathbb{P}(v):= \varphi_{T+\delta(v)}(v), \qquad \text{for} \; v \in O_0,
    \label{eq:poincaremap}
\end{equation}
where $O_0,O_1 \subset S$ are open neighborhoods of $u \in S$; $\delta:O_0 \to \mathbb{R}$ is a function such that for all $v \in O_0$, $(T+\delta(v),v) \in \mathcal{D}$ (the domain of $\varphi$); and if $t \in (0,T+\delta(v))$, then $\varphi_t(v) \notin O_0$.

The existence 
of such $\delta$ (and hence $\mathbb{P}$) is subject to the joint smoothness of the semiflow,
which was
confirmed in the course of Corollary~\ref{cor:jointsmoothness} for the case of the Navier-Stokes semiflow over all of $\mathcal{D} \cap (\mathbb{R}^{>0} \times V)$ in the $C^\infty$ sense.
It follows that the Poincaré map $\mathbb{P}$ as in \eqref{eq:poincaremap} is $C^\infty$ (see \cite{marsden1976hopf,abraham1978foundations}, or \cite{henry1981} for a direct implicit function theorem approach).
Moreover, $\mathbb{P}$ is unique up to conjugacy between different choices of $S$ \cite{abraham1978foundations}.
In particular, if $\mathbb{P}:O_0 \to O_1$ and $\mathbb{P}':O_0' \to O_1'$ are two Poincaré maps on $S$ (transversal to $\Gamma$ at $u$) and $S'$ (transversal to $\Gamma$ at $u'$) respectively,
then there exist neighborhoods $O_2 \subset S$ of $u$ and $O_2' \subset S'$ of $u'$ and a $C^\infty$ map $h:O_2 \to O_2'$ such that $O_2 \subset O_0 \cap O_1$, $O_2' \subset O_0'$
and the diagram 
\begin{equation} 
\begin{tikzcd}
    \mathbb{P}^{-1}(O_2) \cap O_2 \arrow[r, "\mathbb{P}"] \arrow[d, "h"']
    & O_2  \arrow[d, "h"] \\
    O_2' \arrow[r, "\mathbb{P}'"']
    &  S'
    \label{eq:poincare_uniq}
\end{tikzcd}
\end{equation}
commutes.
Note that $h$ in \eqref{eq:poincare_uniq} is constructed by propagating along trajectories starting from $S$ up until reaching $S'$ \cite{abraham1978foundations} -- in particular, $h$ and $Dh(u)$ are injective because $\varphi_t$ is.

The spectrum of the linearized Poincaré map is given by $\sigma ( D \mathbb{P} (u)) = \sigma  \left( \left. D \varphi_T(u) \right\vert_{T_uS} \right)$ with
$D \varphi_T(u)$ compact, and hence $\sigma ( D \mathbb{P} (u)) = P\sigma ( D \mathbb{P} (u)) \cup \{ 0 \}$.
The conjugacy \eqref{eq:poincare_uniq} yields
\begin{equation}
   \left( D \mathbb{P}' (u') - \lambda I \right) Dh(u) = Dh(u) \left( D \mathbb{P} (u) - \lambda I \right),
   \label{eq:spectrum_invariance}
\end{equation}
which - along with the injectivity of $Dh(u)$ - implies that $P\sigma ( D \mathbb{P} (u) )  = P\sigma ( D \mathbb{P}' (u') )$, i.e.\ $\sigma(\Gamma) : = \sigma ( D \mathbb{P} (u))$ is independent of the choice of $u$ and $S$ along $\Gamma$.
In the following, we assume $S$ is chosen to be the normal space $N_u \Gamma := T_u V \big/ T_u \Gamma$ for each $u \in \Gamma$, with norm and inner product inherited from $V$.

Similarly to \eqref{eq:sugma}, we decompose the spectrum $\sigma(\Gamma)$ 
into two disjoint subsets 
\begin{equation}
    \Sigma_1 = \left\{ \, \lambda \in \sigma ( \Gamma ) \; \vert \; | \lambda | > \beta_1 \, \right\}, \qquad 
    \Sigma_2 = \left\{ \, \lambda \in \sigma ( \Gamma ) \; \vert \; |\lambda| < \beta_2 \, \right\},
    \label{eq:sugma2PO}
\end{equation}
with $\beta_1 > \beta_2$ -- note by compactness $\Sigma_1$ is finite.
At each point $u$ along the orbit $\Gamma$, we may define a projection $P_{\Sigma_1}(u) \in \mathcal{L}(N_u \Gamma)$ according to 
\begin{displaymath}
  P_{\Sigma_1}(u) = \frac{1}{2 \pi \mathrm{i}} \int_{\gamma_1} R_z ( D \mathbb{P}(u)) dz,
\end{displaymath}
for a curve $\gamma_1$ in $\mathbb{C}$ encircling $\Sigma_1$.
If $h_t: N_u\Gamma \to N_{\varphi_t(u)}\Gamma$ denotes the conjugacy map from \eqref{eq:poincare_uniq} between different choices of base points $\varphi_t(u) \in \Gamma$ parameterized by $t$, we have that $h_t$ depends smoothly on $t$ (see Lemma (2B.3) in \cite{marsden1976hopf}).
Via \eqref{eq:spectrum_invariance}, we obtain 
\begin{equation}
   P_{\Sigma_1}(\varphi_t(u)) D h_t (u) = D h_t (u) P_{\Sigma_1}(u),
   \label{eq:proj_t}
\end{equation}
which implies that $P_{\Sigma_1}:\Gamma \to \mathcal{L}(N\Gamma )$ is smooth in the operator norm,
and hence extends to define a bundle morphism $\mathcal{P}_{\Sigma_1}:N\Gamma \to N\Gamma$  over $\mathrm{id}_\Gamma$ given by $(u,v) \mapsto (u,P_{\Sigma_1}(u)v)$.
Consequently, we obtain a splitting of $N\Gamma$ into spectral subbundles $E^{\Sigma_1}:= \mathrm{im} (\mathcal{P}_{\Sigma_1})$ and $E^{\Sigma_2}:= \mathrm{ker} (\mathcal{P}_{\Sigma_1})$, i.e.
\begin{equation}
   N\Gamma = E^{\Sigma_1} \oplus E^{\Sigma_2}.
   \label{eq:bundle_split}
\end{equation}
Note $E^{\Sigma_1}$ and $E^{\Sigma_2}$ are invariant under the semiflow linearized about the periodic orbit.
The same splitting can be obtained by applying the theory of Sacker and Sell~\cite{SACKER199417}
to the shifted semiflow $\beta^{t/T} \varphi_t$ (for $\beta \in (\beta_2,\beta_1)$)
restricted to the normal bundle.

The following definition is the analogue of Definition~\ref{def:wloc} for periodic orbits.

\begin{definition}
\label{def:wlocPO}
We say that $W^{\Sigma_1}_{\mathrm{loc},r}(\Gamma) \subset B_r(\Gamma)\footnote{Here, $B_r(\Gamma)$ is understood as a radius $r$ open neighborhood of $\Gamma$ in $V$.} \subset V$ is a local $C^k$ spectral submanifold about a periodic orbit $\Gamma$ corresponding to the spectral subset $\Sigma_1$ if it satisfies the following: 
\begin{enumerate}[label=\upshape{(\roman*)}]
    \item $W^{\Sigma_1}_{\mathrm{loc},r}(\Gamma)$ is locally invariant under $\varphi_t$, i.e., if $u \in W^{\Sigma_1}_{\mathrm{loc},r}(\Gamma)$, then $\varphi_t (u ) \in W^{\Sigma_1}_{\mathrm{loc},r}(\Gamma)$ for those $t \geq 0$ with
    $
    \varphi_\tau (u) \in B_r(\Gamma)
    $
    for all $\tau \in [0,t]$;
    \item  $W^{\Sigma_1}_{\mathrm{loc},r}(\Gamma)$ is tangent to $T \Gamma \oplus E^{\Sigma_1}$ at $\Gamma$;
    \item $W^{\Sigma_1}_{\mathrm{loc},r}(\Gamma)$ is a $C^k$ submanifold of $B_r(\Gamma)$.
\end{enumerate} 
\end{definition} 

The corresponding definition for the local foliation (Definition~\ref{def:foliation}) can be extended to periodic orbits analogously.

To conclude existence of these structures,
we now apply (the local version of) Theorem~3.1 of \cite{chen97} to $\mathbb{P}$. 

\begin{figure}
\includegraphics[width=0.7\textwidth]{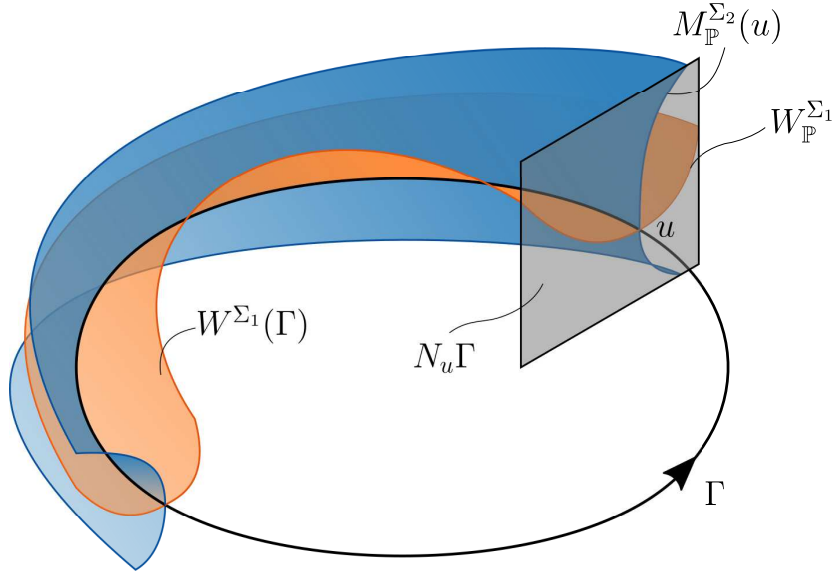}
\centering
\caption{Sketch of the results of Theorem~\ref{thm:PO}, showing a spectral submanifold $W^{\Sigma_1}(\Gamma)$ about the periodic orbit $\Gamma$, and a single leaf of the foliation $M^{\Sigma_2}_{\mathbb{P}}$ integrated about $\Gamma$. 
}
\label{fig:PO}
\end{figure}

\begin{theorem}
\label{thm:PO}
For a periodic orbit $\Gamma \subset V$ of the Navier-Stokes semiflow $\varphi_t$, let $\mathbb{P}:N_u\Gamma \to N_u \Gamma$ denote its associated Poincaré map and let $\sigma(\Gamma)$ denote the spectrum of $D \mathbb{P}(u)$ for some $u \in \Gamma$. 
For any $\beta < \max | \sigma ( \Gamma ) |,$  $\beta \notin | \sigma ( \Gamma ) |$ 
and for $r>0$ small enough, there exists 
a locally invariant manifold $W^{\Sigma_1}_{\mathrm{loc},r}(\Gamma) \subset B_r(\Gamma)$ such that
\begin{enumerate}[label=\upshape{(\roman*)}]
    \item $W^{\Sigma_1}_{\mathrm{loc},r}(\Gamma)$ is an immersed $C^1$ spectral submanifold corresponding to the spectral subset $\Sigma_1 : = \{ \lambda \in \sigma (\Gamma) \, | \, | \lambda | > \beta \}$, in the sense of Definition~\ref{def:wlocPO}.
    \item \label{thm:PO_item2} If $v \in N_{u}\Gamma$ and $\varphi_t (v )$ remains defined for all $t \geq 0$ with $\mathcal{O}^+(v) \subset B_{r_1}(\Gamma)$ (for some $r_1 > 0$ small enough), then there exists a unique $w \in W^{\Sigma_1}_{\mathrm{loc},r}(\Gamma) \cap N_u\Gamma$ such that  $\mathcal{O}^+(w) \subset W^{\Sigma_1}_{\mathrm{loc},r}(\Gamma)$ and 
    \begin{displaymath}
        \Vert \mathbb{P}^n (v) - \mathbb{P}^n (w) \Vert_V \leq C \beta^n, \qquad  n \in \mathbb{N}.
    \end{displaymath}
\end{enumerate}
Moreover, there exists a locally positively invariant foliation for $B_r(\Gamma)$ corresponding to the spectral subset $\Sigma_2 = \sigma (\Gamma) \setminus \Sigma_1$, satisfying properties analogous to Definition~\ref{def:foliation}.
\end{theorem}

\begin{proof}
We merely outline the main steps of the proof, as it would otherwise be a repetition of the above procedure for the case of a simple fixed point.

We consider the Poincaré map $\mathbb{P}:N_u\Gamma \to N_u \Gamma$ and decompose it according to the requirements of Theorem~3.1 in \cite{chen97}, with $L:= D \mathbb{P}(u)$ as the linear part and $N:= \mathbb{P} - D \mathbb{P}(u)$ as the nonlinear part.
The dichotomy condition on $L$ \eqref{eq:requiredLbounds1}-\eqref{eq:requiredLbounds2} is satisfied due to the splitting of the (discrete) spectrum \eqref{eq:sugma2PO}.
For $N$, we need to be able to control its (global) Lipschitz constant.
This is achieved by means of a cutoff function, just as above --
note that
$DN(u) = 0$ implies, due to smoothness and the mean value theorem, that there exists a neighborhood on which $\mathrm{Lip}_{N_u\Gamma}(N)$ can be kept small.

With an application of Theorem~3.1, \cite{chen97}, we obtain a $C^1$ manifold $W_{\mathbb{P}}^{\Sigma_1}$ and a $C^1$ foliation $ M_{\mathbb{P}}^{\Sigma_2}$ on $N_u\Gamma$ invariant under $\mathbb{P}$ that separates trajectories according to their Lyapunov exponents in the desired fashion (see Theorem~\ref{thm:main}).

We now integrate this manifold around the periodic orbit, which is
standard procedure
-- but we will nonetheless be thorough here to ensure the resulting manifold remains immersed in $N\Gamma$.

For each $t> 0$, let $\mathcal{W}_t : = h_t \left( W_{\mathbb{P}}^{\Sigma_1} \right) \subset N_{\varphi_t(u)}\Gamma$ denote the time-$t$ immersed submanifold along the orbit, where $h_t:N_u\Gamma \to N_{\varphi_t(u)}\Gamma$ is the conjugacy map from \eqref{eq:proj_t} obtained by propagating along trajectories until reaching $N_{\varphi_t(u)}\Gamma$.
Define $\theta: (-\varepsilon, \varepsilon) \times O \to \mathrm{im}(\theta) \subset V$, with $0 < \varepsilon < t$ and $O$ an open neighborhood of $\varphi_t(u)$ in $ \mathcal{W}_t$, as
\begin{displaymath}
   \theta (s, \cdot ) := \varphi_s \vert_{O},
\end{displaymath}
with $\varepsilon$ small enough to ensure $\varphi_{-\varepsilon}$ is defined on $O \subset h_t \left( W_{\mathbb{P}}^{\Sigma_1} \right)$.
Note that $\theta (0, \cdot ) = \mathrm{id}_{O}$ implies
\begin{displaymath}
    D_2 \theta (0,\varphi_t(u)) = \mathrm{id}_{T_{\varphi_t(u)}\mathcal{W}_t},
\end{displaymath}
and, by definition,
\begin{displaymath}
   D_1 \theta (0,\varphi_t(u)) = \frac{d}{ds} \bigg\vert_{s=t} \varphi_s (u).
\end{displaymath}
Since $\mathcal{W}_t$ is transversal to $\Gamma$,
$D \theta (0,\varphi_t(u)) : T_0 \mathbb{R} \oplus T_{\varphi_t(u)}\mathcal{W}_t \to  T_{\varphi_t(u)} \Gamma \oplus T_{\varphi_t(u)}\mathcal{W}_t $ is an isomorphism. 
Consequently, up to shrinking  $\varepsilon$ and $O$, we may assume $\theta$ maps its domain diffeomorphically onto its image in $V$.
The resulting manifold $\mathcal{W}^{\varepsilon}_t := \mathrm{im}(\theta)$ is immersed in $V$, as
$T_{\varphi_t(u)} \Gamma \oplus T_{\varphi_t(u)}\mathcal{W}_t$ is finite dimensional and hence complemented in $T_{\varphi_t(u)} \Gamma \oplus N_{\varphi_t(u)}\Gamma \simeq V$.

Take a finite subcovering of $\Gamma$ consisting of $\left\{ \, \mathcal{W}^{\varepsilon_i}_{t_i} \cap \Gamma \; \big| \; i = 1,\ldots,n \, \right\}$, with $t_i \in [\delta, T + \delta]$ for some $\delta > 0$.
Then, there exists $r>0$ such that
\begin{displaymath}
W_{\Sigma_1}^{\mathrm{loc},r}(\Gamma): = B_r(\Gamma) \cap \left( \bigcup_{i = 1,\ldots,n} \mathcal{W}^{\varepsilon_i}_{t_i} \right)
\end{displaymath}
is an immersed spectral submanifold in the sense of Definition~\ref{def:wlocPO} (local invariance is clear from the construction, tangency follows from \eqref{eq:proj_t}).

The corresponding result for the foliation and \ref{thm:PO_item2} follows upon integrating $M_{\mathbb{P}}^{\Sigma_2}$ backwards in time via the implicit function theorem.
\end{proof}

\section{Parameterization}
\label{sect:parameterization}

In the preceding text, we sought 
the manifold $W^{\Sigma_1}$
as a graph over its tangent space at $U$, $P_{\Sigma_1}V$.
This is perhaps the most common approach in the classical literature, and, despite its simplicity, it is completely reasonable, since any manifold can be locally viewed as a graph over its own tangent space. 
However, as 
the manifold proceeds to develop a fold further away from the base point (in our case, $U$),
the above approach breaks down, as the envisioned function (in our case, $\phi:P_{\Sigma_1}V \to P_{\Sigma_2} V$) would now be ill-defined.

At the level of numerical computations, invariant manifolds are obtained upon enforcing an invariance relation in an approximate manner, achieved by recursively tuning the function ($\phi$) that represents the manifold.
The specific form of invariance relation could vary:
If we were to proceed with 
the above classical route, we could either use the Lyapunov-Perron integrals from the proof of Theorem~\ref{thm:main} directly, 
or substitute $\phi$ back into the equations - while assuming that trajectories remain on $\mathrm{graph}(\phi)$ - to obtain a PDE for the unknown function $\phi$, which is perhaps the most popular choice.

Instead, we will implement a fairly recent technique called the \textit{parameterization method} -- developed in the early 2000s \cite{Cabre2003a,Cabre2003b,Cabre2005}.
We merely utilize it here for its benefits as a computational tool (see \cite{haro2016} for a recent review detailing its numerical aspects),
but the parameterization method has otherwise created a blooming subject area in 
dynamical systems, fueling much of today's research \cite{haro2016,Mireles2015,BergMireles2016,Gonzalez2022paramfem}. 
Crucially, the method envisions the manifold as the image of an embedding - rather than the above detailed
graph approach - circumventing the issue regarding folds, and hence possibly enlarging its (initial)\footnote{The obtained manifolds can be subsequently globalized further via time-integration.} domain of existence. 
This is especially important from the point of view of model reduction -- 
to obtain a
distinctive edge over Galerkin-type reductions that are based on the linearized dynamics,
nonlinear reduced models have to be reliable a sizeable distance away from the fixed point (or periodic orbit).

We now recall the basic principles of the method 
adapted to the setting of the Navier-Stokes system, but the reader is encouraged to obtain the full picture from the work of Haro et al.~\cite{haro2016}.
Assume, as before, that $U$ is a fixed point of the Navier-Stokes equations \eqref{eq:NSE}, and consider
\begin{equation}
    \mathcal{X}: u \mapsto \mathcal{A}_U u - B(u,u)
    \label{eq:mathXdef}
\end{equation}
as a $C^\infty$ vector field on $\mathrm{dom}_V(\mathcal{A}_U)$, so that $\mathcal{X}(u) \in T_u\mathrm{dom}_V(\mathcal{A}_U) \simeq V$. 
The equation for the perturbation velocity $u$ now reads
\begin{equation}
    \frac{d u}{dt} = \mathcal{X}(u),
    \label{eq:p_fullsys}
\end{equation}
with $\mathcal{X}(0) = 0$ as a fixed point;
 the full flow state may be recovered as $v = U + u$.

The parameterization method seeks invariant manifolds of \eqref{eq:p_fullsys} as an embedding of a \textit{model manifold} $\Theta$ into the full phase space $ V$ via the \textit{parameterization} $K: \Theta \to \mathrm{dom}_V(\mathcal{A}_U)  \subset V$.
In this work, the model manifold $\Theta$ will always be taken to be a hypercube in some Euclidean space.

The submanifold $\mathrm{im}(K) \subset \mathrm{dom}_V(\mathcal{A}_U)$ will be invariant under the evolution of \eqref{eq:p_fullsys} if there exists a vector field $R: \Theta \to T\Theta$ such that\footnote{Here, $DK:T\Theta \to T\mathrm{dom}_V(\mathcal{A}_U)$ is the usual fiberwise derivative map along $K$.}
\begin{equation}
    \mathcal{X} \circ K = DK [R]
    \label{eq:invariance}
\end{equation}
holds.
In other words, \eqref{eq:invariance} asserts that $\mathcal{X} \vert_{\mathrm{im}(K)}$ should be equal to the pushforward of $R$.
The obtained reduced model on $\Theta$ is governed by
$$
\dot{\theta} = R(\theta),
$$
for some choice of local coordinates $\theta$.
The invariance equation \eqref{eq:invariance} also implies that
the following conjugacy relationship
$$
    \varphi_t \circ K = K \circ \varphi^R_t
$$
holds between
the flow of $R$, $\varphi^R_t$, and
full Navier-Stokes semiflow restricted to $\mathrm{im}(K)$.

\begin{figure}
\includegraphics[width=0.9\textwidth]{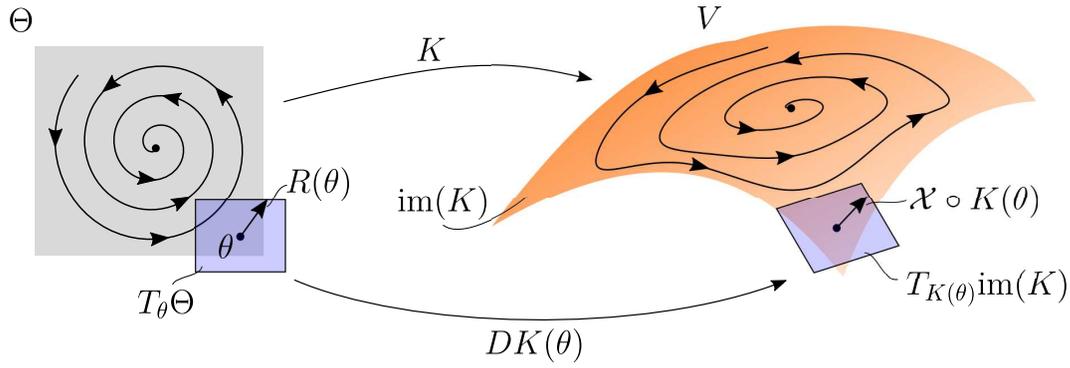}
\centering
\caption{Geometrical interpretation of the invariance equation \eqref{eq:invariance}.
}
\label{fig:param}
\end{figure}

\subsection{Solving the invariance equation}

With the use of parameterization, we obtain an additional
degree of freedom over traditional methods,
which results from the invariance equation \eqref{eq:invariance} being underdetermined. 
In particular, we are free to choose the dynamics on $\Theta$ up to a conjugacy.
Indeed,
if $(K,R)$ is a solution pair of \eqref{eq:invariance}, then so is
$(K \circ h, (Dh)^{-1}   [R \circ h])$ 
for any diffeomorphism $h$ on $\Theta$. 

The classical graph transform method may be recovered as a special case - referred to as the \textit{graph style} of parameterization \cite{haro2016} - upon selecting $K$ and $\Theta$ to match the set of assumptions determined by the prior.
Within the present framework, this amounts to taking
$\Theta \subset P_{\Sigma_1}V$ to be an open subset
and $K = (\mathrm{id}_{\Theta},\phi)$, where $\phi: \Theta \to P_{\Sigma_2}V$ is the sought function that represents the manifold as $\mathrm{graph}(\phi) = \mathrm{im}(K)$.
Substituting back into the invariance equation \eqref{eq:invariance} and applying $P_{\Sigma_2}$ yields
\begin{equation}
    P_{\Sigma_2} \mathcal{X} (\theta,\phi(\theta)) =  D \phi (\theta) [R(\theta)], \qquad \theta \in \Theta,
    \label{eq:PDEinvariance}
\end{equation}
where $R(\theta) = P_{\Sigma_1}\mathcal{X}(\theta,\phi(\theta))$ governs the reduced dynamics.
With equation \eqref{eq:PDEinvariance}, we have recovered the well-known PDE formulation of the invariance equation.
Notice that the particular choice of $K$ fixes the diffeomorphism $h$ (from the preceding paragraph), and hence we have had no room to adjust the reduced dynamics $R$ -- they were obtained as a byproduct of the graph transform procedure.

The general form of parameterization \eqref{eq:invariance} allows for much more.
In particular, it
leaves room to simultaneously perform a normal form transformation \cite{murdock2003}  with the invariant manifold computation.
The exact procedure 
is easier to detail once power series expansions
have already been carried out, which is the subject of the following section.

\subsubsection{Power series expansions}

From now on, we shall assume without loss of generality that $\Theta$ is a finite dimensional Euclidean space $\mathbb{R}^r$ or some nonmeagre subset thereof.
Let us also fix an orthonormal basis $\{e_i \}_{i=1,\ldots,r}$ of $\Theta$.

Following the methodology of Haro et al.~\cite{haro2016},
we seek the parameterization $K: \Theta \to \mathrm{dom}_V (\mathcal{A}_U)$ and the reduced dynamics $R: \Theta \to \Theta$ as
power series in $\theta = (\theta_1,\ldots,\theta_r)$ of the form
\begin{equation}
    K(\theta) = \sum_{j \geq 1} K_j(\theta); \qquad R(\theta) = \sum_{j \geq 1} R_j(\theta),
    \label{eq:expansions}
\end{equation}
where each $K_j$ (and $R_j$, respectively) is an $r$-variate homogeneous polynomial of order $j$ 
with values in $\mathrm{dom}_V (\mathcal{A}_U)$ (resp.\ $\Theta$).

Using expansions of the form \eqref{eq:expansions}, the invariance equation \eqref{eq:invariance} is solved recursively on an order-by-order basis.
At order $j$, \eqref{eq:invariance} reads
\begin{equation}
    (\mathcal{X} \circ K )_j = \sum_{k = 1}^j D K_{j-k+1}[R_k], \qquad j = 1, \ldots, 
    \label{eq:orderj}
\end{equation}
where the left hand side can be written explicitly using \eqref{eq:mathXdef} as
$$
(\mathcal{X} \circ K )_j (\theta)= \mathcal{A}_U K_j (\theta) - \sum_{r = 1}^{j-1} B(K_{j-r}(\theta),K_{r}(\theta)).
$$

At order $1$, \eqref{eq:orderj} yields 
\begin{equation}
    \mathcal{A}_U K_1 (\theta) = DK_1(\theta)[R_1(\theta)] = K_1 \circ R_1 (\theta),
    \label{eq:firstorder}
\end{equation}
which implies that $\mathcal{A}_U$ leaves $\mathrm{im}(K_1)$ invariant.
Moreover, we necessarily have that $\mathrm{im}(K_1)$ spans the tangent space of the desired manifold at $0$, $T_0 \mathrm{im}(K)$, which should further equal $\mathrm{im}(P_{\Sigma_1})$.
A convenient choice is to have eigenfunctions (and generalized eigenfunctions) $w_i$, $i = 1,\ldots,r$ of $\mathcal{A}_U$ corresponding to $\Sigma_1$ constitute the basis of $\mathrm{im}(P_{\Sigma_1})$,
and then define $K_1:\Theta \to \mathrm{dom}_V(\mathcal{A}_U)$ by requiring that $K_1(e_i) = w_i$, $i = 1,\ldots,r$.
Corresponding to the basis $\{e_i \}_{i=1,\ldots,r}$, $R_1$ is then the Jordan normal form of $\mathcal{A}_U P_{\Sigma_1}$.

At subsequent orders, the invariance equation \eqref{eq:orderj} reads
\begin{multline}
    \mathcal{A}_U K_j (\theta) - DK_j(\theta)[R_1(\theta)] -K_1 \circ R_j(\theta) = \\ \sum_{r = 1}^{j-1} B(K_{j-r}(\theta),K_{r}(\theta)) +  \sum_{k = 2}^{j-1} D K_{j-k+1}(\theta)[R_k(\theta)] =: \eta_j(\theta), \qquad j = 2,\ldots, 
    \label{eq:higherorders}
\end{multline}
where the right hand side is composed solely of lower order terms that have already been computed by this stage.
In \eqref{eq:higherorders}, the highest order
term of the reduced dynamics, $R_j$, only contributes to the $P_{\Sigma_1}V$
part of the splitting $P_{\Sigma_1}V \oplus P_{\Sigma_2}V$ of $V$.
Hence, projecting onto the second factor, we obtain the following functional equation
\begin{equation}
      \mathcal{A}_U P_{\Sigma_2} K_j  - D(P_{\Sigma_2}K_j)[R_1]  = P_{\Sigma_2}\eta_j,
    \label{eq:proj2invariance} 
\end{equation}
which determines $P_{\Sigma_2} K_j$ provided that the operation $M \mapsto \mathcal{A}_U M  - DM[R_1]$ is invertible between appropriate spaces of $j$-multilinear maps \cite{Cabre2003a}. 
Due to the two space approach used in setting up the invariance equation,
the solvability of \eqref{eq:proj2invariance} is not obvious and,
to our knowledge, not yet known.
Since we are only using parameterization in this work to accommodate computations (the existence of the sought manifolds is already known from Section~\ref{sect:manifolds}),
we will address this only at the level of numerics -- once the function space $V$ has been discretized.
In this case, there is no longer a distinction between $\mathrm{dom}_V(\mathcal{A}_U)$ and $V$; and $\mathcal{A}_U$ acts as a bounded operator, which permits the use of
proposition~3.2 of \cite{Cabre2003a}.
We arrive at the standard non-resonance conditions:\footnote{The operation $\mathrm{vec}$ sends $\Sigma_1$ to a vector composed of all its elements, repeated according to their algebraic multiplicity (as an eigenvalue).}
\begin{equation}
    \langle \mathsf{a}, \mathrm{vec}( \Sigma_1) \rangle_{\mathbb{R}^r} \notin \Sigma_2, \qquad \text{for all } \mathsf{a} \in \mathbb{N}^r \text{ with } |\mathsf{a} |_{\ell^1} = j,
    \label{eq:cross_resonance}
\end{equation}
which serve as a necessary and sufficient conditions to solve (the discretized version of) \eqref{eq:proj2invariance} at order $j$.
On their occurrence (i.e.\ whenever $\langle \mathsf{a}, \mathrm{vec}( \Sigma_1) \rangle_{\mathbb{R}^r} \in \Sigma_2$ for some $\mathsf{a} \in \mathbb{N}^r$ with $|\mathsf{a} |_{\ell^1} = j$), they are termed \textit{cross resonances}. 
These conditions appear as a result of seeking $K$
in the form of a power series expansion, and hence any other technique (such as the graph transform) would suffer from them equivalently, once the relevant expansions are assumed.\footnote{Clearly, the proof of Theorem~\ref{thm:main} circumvented such expansions -- but for computational purposes, power series are ideal.}

The other half of \eqref{eq:higherorders} reads
\begin{equation}
      \mathcal{A}_U P_{\Sigma_1} K_j - D(P_{\Sigma_1}K_j)[R_1] -K_1 \circ R_j = P_{\Sigma_1}\eta_j.
     \label{eq:proj1invariance}
\end{equation}
A suitable solution to \eqref{eq:proj1invariance} would be to set $P_{\Sigma_1}K_j = 0$ and adjust $R_j$ accordingly, since $K_1:\Theta \to \mathrm{im}(P_{\Sigma_1})$ is bijective -- doing so at every order would once again recover the graph transform approach.
Alternatively,
if the corresponding set of non-resonance conditions (termed \textit{internal resonance} conditions) permit, i.e.
\begin{equation}
    \langle \mathsf{a}, \mathrm{vec}( \Sigma_1) \rangle_{\mathbb{R}^r} \notin \Sigma_1, \qquad \text{for all } \mathsf{a} \in \mathbb{N}^r \text{ with } |\mathsf{a} |_{\ell^1} = j,
    \label{eq:internal_resonance}
\end{equation}
holds, $R_j$ can be chosen arbitrarily, and $P_{\Sigma_1}K_j$ will be uniquely determined as above.
The choice $R_j = 0$ essentially recovers the standard normal form procedure \cite{murdock2003} for the reduced dynamics, and hence is termed the \textit{normal form style} of parameterization.
To reiterate, condition \eqref{eq:internal_resonance} is not necessary for the manifold computation -- but merely for the normal form procedure.
In particular, if \eqref{eq:internal_resonance} holds at each order, the dynamics on the manifold $\mathrm{im}(K)$ are linearizable.

\begin{remark} \label{remark:param_style} 
Whether the graph or normal form style of parameterization is more suitable depends on the problem at hand.
Generally speaking, unless there are near internal resonances present, the normal form style
is expected to perform better - computationally at least - as there will be less coefficients to extract \cite{haro2016}.
If there are near internal resonances present, 
enforcing $R_j = 0$ as part of the normal form procedure in \eqref{eq:proj1invariance} may result in small divisors, resulting in ill-conditioned matrices when solving for $P_{\Sigma_1}K_j$.
The ideal route is to mix the two styles, and only set coefficients of $R_j$ to zero that correspond to non-resonant subspaces.
Moreover, the domain of validity of the normal form transformation may be smaller than that of the invariant manifold, and hence one has to be careful, especially when the objective is to extract other nearby invariant sets. 
In our numerical implementation (based on that of Jain et al.~\cite{shobhit2021code}), \eqref{eq:cross_resonance} and \eqref{eq:internal_resonance} are checked on an example-by-example basis, with an adjustable tolerance controlling the degree of permissible internal resonances,
which ensures that the ideal mix of the graph and normal form approaches is achieved.
A relevant example is given in Section~\ref{sect:example_lam}, highlighting the importance of choosing said tolerance parameter carefully.
\end{remark}

Note that
if \eqref{eq:cross_resonance} holds in the discretized system - which is necessary to compute the manifold -
the obtained manifold will be unique in some subclass $C^L$ of sufficiently smooth manifolds (for some $L$ large enough, possibly depending on the discretization), according to the main theorem of parameterization \cite{Cabre2003a}.
Therefore, within the realm of the discretized system, the power series will converge to the unique smoothest manifold, even in cases where we could not achieve theoretical uniqueness in Section~\ref{sect:uniqueness}.
Despite having no theoretical relevance, uniqueness in this 'numerical' sense eliminates the ambiguity and convergence issues that plague center manifold calculations, for instance (see the introduction to \cite{Haller2016} for a review of this issue).

Finally, it is clear that if a $C^k$ map $K:\Theta \to \mathrm{dom}_V(\mathcal{A}_U)$ satisfying \eqref{eq:invariance} and \eqref{eq:firstorder} 
exists,
then $\mathrm{im}(K) \cap \widehat{B}_r(0)$ defines 
a $C^k$ spectral submanifold $W^{\Sigma_1}_{\mathrm{loc},r}$ in the sense of Definition~\ref{def:wloc}. 
Hence, the obtained manifold is locally attracting according to Corollary~\ref{cor:Wloc} (and the discussion following it), with exponent $\beta > \sup_{\lambda \in \Sigma_2} \; \mathrm{Re} \lambda$.

\section{Example: 2D channel flow}
\label{sect:channelflow}

We take pressure-driven, two dimensional channel flow of a Newtonian fluid as our main example.
In this context, invariant manifold theory will
be of most use in understanding
heteroclinic transitions between travelling wave solutions
(for similar examples, see e.g.\ 
\cite{halcrow2009heteroclinic,budanur2017heteroclinic,van2011matrix}).
Travelling waves appear as fixed points in a moving frame, and are usually termed \textit{exact coherent states} in the applied fluids literature.

Before discussing any technical peculiarities of the problem or the numerical set-up,
let us give a sample application
concerned with the simplest example of travelling wave solutions, the laminar state.

\subsection{Invariant manifolds about the laminar state}
\label{sect:example_lam}

The laminar flow is a stationary solution of \eqref{eq:NSE} in the appropriate, channel flow set-up (see Section~\ref{sect:channel_setup}), depicted in Figure~\ref{fig:lam}.
The coordinate system $(x_1,x_2)$ is aligned in all examples such that $x_1$ denotes the streamwise component and $x_2$ the wall-normal component.

\begin{figure}
\includegraphics[width=0.5\textwidth]{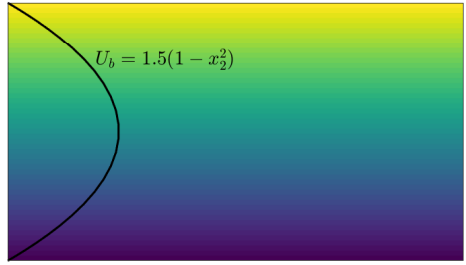}
\centering
\caption{The laminar base state $U_b$ in a channel flow. 
Colored contours correspond to the vorticity; the solid black line shows the streamwise ($x_1$) component of the velocity (the wall-normal ($x_2$) component is trivially $0$).
}
\label{fig:lam}
\end{figure}

It is well known, in the context of 2D channel flows,
that the laminar state turns unstable 
as the \textit{Reynolds number} 
is increased above a certain threshold $Re_c$ \cite{orszag_1971}.
The loss of stability occurs via a subcritical Hopf bifurcation,
and hence,
there exists a branch of (unstable) periodic orbits emanating from it
that may be traced down to
Reynolds numbers below the critical value ($Re < Re_c$).
This situation is pictured in Figure~\ref{fig:kcrit_branch} over a single Poincaré section.

\begin{figure}
\includegraphics[width=0.5\textwidth]{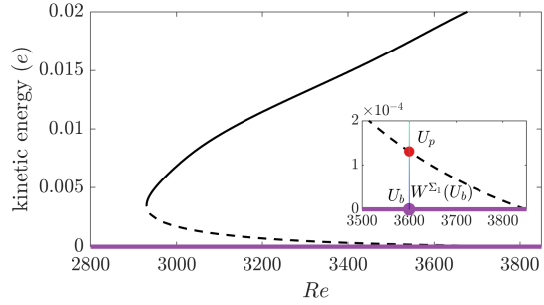}
\centering
\caption[Bifurcation diagram of 2D channel flow]{Bifurcation diagram of 2D channel flow, obtained via
branch continuation 
techniques of Section~\ref{sect:TWformulation}, with wavenumber fixed as $k = k_{crit} = 1.02056$ 
(the critical wavenumber for which the laminar state turns unstable at minimal $Re$). 
On the vertical axis, $e$ is defined as the perturbation kinetic energy above the laminar state, see \eqref{eq:pert_kin_e}. The laminar solution is shown in purple, the nontrivial branch emanating from the bifurcation at $Re_{crit} \approx 3848$\footnotemark is shown in solid black.
}
\label{fig:kcrit_branch}
\end{figure}
\footnotetext{Note that this differs from the well-known $Re_{crit} = 5772$ \cite{orszag_1971} by a factor of 1.5 due to our definition of $Re$ \eqref{eq:Reynolds}.} 

The initial, unstable, lower branch of solutions in Figure~\ref{fig:kcrit_branch} 
serves the purpose of mediating between the laminar and more complex states.
States as such are customarily termed \textit{edge states}.
Due to their significance, 
it is desirable to extract them directly,
without having to rely on continuation techniques that are launched from the laminar bifurcation. 
Indeed, in many scenarios
it is cumbersome to perform the latter (see Example~\ref{sect:example_visco}), or even outright impossible (e.g., in pipe and Couette flows, where the laminar state remains stable).

\begin{figure}
\includegraphics[width=0.5\textwidth]{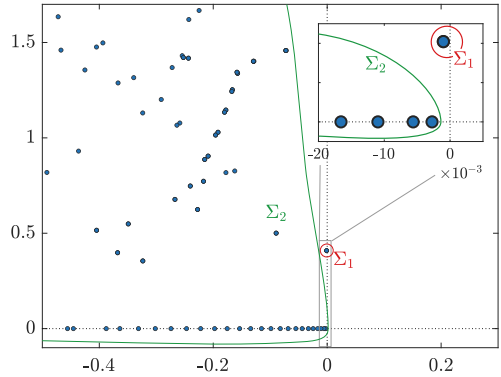}
\centering
\caption{Spectrum of the linearized operator $\mathcal{A}_{U_b}$ about the laminar base state $U_b$ shown on the complex plane, partitioned according to \eqref{eq:sugma}. 
}
\label{fig:spectrum_lam}
\end{figure}

We shall do so at a single point $Re = 3600 < Re_c$, by computing the slow manifold at the (linearly stable) laminar state $U_b$, whose associated spectrum $\sigma(\mathcal{A}_{U_b})$ is shown in Figure~\ref{fig:spectrum_lam}.
Even though such a slow manifold is generally non-unique,
the one we obtain via the parameterization procedure will  
generally
contain the nontrivial periodic orbit $U_p$, at least in the discretized, finite-dimensional setting.
Indeed, numerical experiments suggest that $U_p \notin W^{\Sigma_2}(U_b)$ and that $W^u(U_p)$ connects back to $U_b$\footnote{These can be verified without needing to rely on experiments in the vicinity (in terms of $Re$) of the laminar bifurcation.},
in which case, $U_p$ has to lie on a $C^\infty$ smooth slow manifold $W^{\Sigma_1}(U_b)$,
unique amongst a sufficiently smooth subclass of manifolds tangent to $\mathrm{im}(P_{\Sigma_1})$ 
\cite{Cabre2003a} and precisely the one we obtain via parameterization.

To discretize the system, we have used $40$ Chebyshev modes in the wall-normal direction and $30$ Fourier modes in the streamwise direction -- a detailed account of the formulation is deferred to the following sections.

Once the spectrum and its eigenvectors are computed (Figure~\ref{fig:spectrum_lam}), $K_1$ and $R_1 $ are readily available. 
The rest of the coefficients are computed order-by-order according to the invariance relation \eqref{eq:higherorders} -- up to order $3$, for this preliminary example.
The resulting invariant manifold $W^{\Sigma_1}(U_b) = \mathrm{im}(K)$ is shown in Figure~\ref{fig:main_lam}, embedded in a $3$-dimensional space spanned by
the perturbation kinetic energy
 \begin{equation}
    e = \frac{1}{2} \Vert u \Vert_{L^2(\Omega,\mathbb{R}^2)}^2,
    \label{eq:pert_kin_e}
 \end{equation}
the mean wall-normal velocity
\begin{equation}
    \mathrm{mwnv} = \int_{-1}^1 u_2 (0,x_2) dx_2
    \label{eq:mwnv}
\end{equation}
and the streamwise velocity fluctuation 
\begin{equation}
    \mathrm{svf} = u_1(0,\hat{x}_2).
    \label{eq:svf}
\end{equation}
Here $u := v - U_b$ denotes the perturbation velocity field above the laminar state, the domain $\Omega$ is specified in \eqref{eq:channel_domain}, and $\hat{x}_{2}$ is some distinguished point in the wall-normal direction.
In this example, we choose $\hat{x}_2 = 0$, so that $\mathrm{svf}$ agrees with the streamwise centerline velocity.

\begin{figure}
\includegraphics[width=1\textwidth]{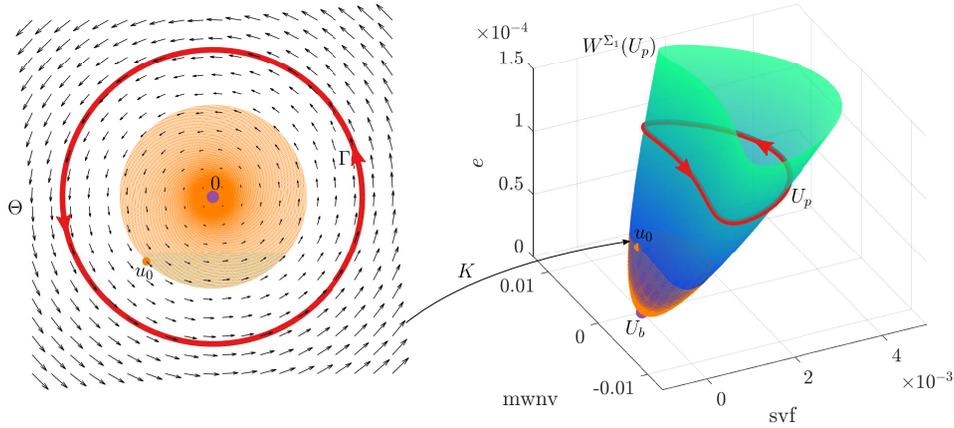}
\centering
\caption{Parameterization of the slow manifold $W^{\Sigma_1}(U_b)$ about the laminar state. (Left) The model space $\Theta$ along with its reduced dynamics described by $R$. (Right) The manifold $W^{\Sigma_1}(U_b) = \mathrm{im}(K)$ embedded in a three-dimensional phase space spanned by $e$, $\mathrm{mwnv}$ and $\mathrm{svf}$.
A single, stable solution with initial condition $u_0$ is shown on both panels, slowly evolving back to $U_b$ ($0$ in the perturbation coordinate system). 
}
\label{fig:main_lam}
\end{figure}

The reduced dynamics $\dot{\theta} = R(\theta)$ on the model space $\Theta$ are governed by 
the simple complex eigenvalue $ \lambda_1 = -0.00105 + 0.409\mathrm{i}$ (the only element of $\Sigma_1$) at linear order.
At order $3$, an additional term appears (via the graph style of parameterization),
resulting in the following 2-dimensional ODE in polar coordinates:
\begin{subequations}
\begin{align}
    \dot{r} &=  -0.00105r + 9.23 r^3,\\
    \dot{\vartheta} &= 0.409 + 52.4 r^2.
\end{align}\label{eq:LAMEX_red_dyn}\end{subequations}
The corresponding vector field and 
invariant sets at $\{(0,0)\}$ and $ \Gamma(t) =\{(r,\vartheta(t)) \; \vert \; r = r_p \}$, with $r_p \approx 0.0107$,
are shown on the left panel of Figure~\ref{fig:main_lam}.
Now $U_p = K(\Gamma)$ yields our desired periodic orbit
on the lower branch of bifurcated solutions (c.f.\ Figure~\ref{fig:kcrit_branch}).
Note that $\dot{\vartheta} = \mathrm{const}.$ on $r = r_p$, and hence $\Gamma$ is a constant speed periodic orbit, i.e., a travelling wave.

The higher order terms in the vector field $R$ \eqref{eq:LAMEX_red_dyn} were computed due to the near internal resonance (cf.\ \eqref{eq:internal_resonance} and Remark~\ref{remark:param_style})
\begin{equation}
    2 \lambda_1 + \bar{\lambda}_1 \approx \lambda_1.
    \label{eq:resonance_example}
\end{equation}
If the goal is to extract other nontrivial states on the manifold $W^{\Sigma_1}$, one should always disallow near resonances of the form \eqref{eq:resonance_example} via either adjusting the tolerance parameter  or opting for the graph style of parameterization.

We now turn to address the outstanding issue we have conveniently skipped -- the technical aspects of implementation.
The formulation will be introduced in a more general framework that permits
travelling wave solutions to be time-independent, so that invariant manifolds along the nontrivial branches become easier to compute (the setup of the above example is retained as a special case with phase speed $0$).

\subsection{Changes to the functional setup}
\label{sect:channel_setup}

First, to accommodate the channel flow setting,
some minor changes are necessary to the
functional setup.
We 
follow the recommendations
outlined in the work of Foias et al.~\cite{foias2001} (see Section~2 of Chapter~II therein).

We assume that the fluid occupies a rectangular domain
\begin{equation}
    \Omega = \left( 0, 2 \pi / k \right) \times \left( -1,1 \right)
    \label{eq:channel_domain}
\end{equation}
between two parallel, stationary, rigid plates that are modeled via no-slip boundary conditions, i.e.\
$$
v = 0 \quad \text{for} \quad  x_2 = \pm 1.
$$
In the remaining, streamwise direction $e_1$, the flow state $(v,p)$ is assumed to be periodic with period $2 \pi / k$ -- here, $k$ is the assumed streamwise wavenumber, which will be treated as a model parameter. 

The fluid flow is sustained by a constant pressure drop $P_1$ between the two ends of the channel, which is \textit{not} periodic and hence has to be absorbed into the forcing term $ f := P_1 e_1$.
Anticipating the Reynolds number formulation, $P_1$ will be treated as an auxiliary variable, so that the volume flux can be kept constant as\footnote{Note that the left hand side of \eqref{eq:volflux} is independent of $x_1$ by incompressibility.}
\begin{equation}
   \frac{1}{2}  \int_{-1}^{1} v_1(x_1,x_2) \, dx_2 = Q,
    \label{eq:volflux}
\end{equation}
where the desired volume flux $Q$ is determined by the laminar state:
\begin{equation}
    Q = \frac{1}{2}\int_{-1}^{1} U_{b}(x_2) \, dx_2.
    \label{eq:volflux_lamstate}
\end{equation}
Condition \eqref{eq:volflux} ensures that the Reynolds number defined by
\begin{equation}
    Re := \frac{ Q}{ \nu}
    \label{eq:Reynolds}
\end{equation}
can be used to non-dimensionalize the Navier-Stokes equations as in \eqref{eq:NSE}, where $\nu$ denotes the kinematic viscosity of the fluid.

The function spaces $H$ and $V$ can now be defined analogously to Section~\ref{sect:functional_setup} with the altered boundary conditions. 
Note that condition \eqref{eq:volflux} also ensures that the Poincaré inequality continues to hold in the space of perturbation functions $u$.
From then on, the analysis/setup can be performed just as in Section~\ref{sect:semiflow}, with little to no adjustments.

For the sake of clarity, we define explicitly the function space intended for the setting of the perturbative equations \eqref{eq:NSE_channel_pert} below as
\begin{equation}
    V := \overline{ \Big\{ \, u \in C^\infty_{per \times c} (\Omega, \mathbb{R}^2) \; \big\vert \; \int_\Omega u = 0, \: \mathrm{div} \, u = 0 \, \Big\} }^{W^{1,2}(\Omega, \mathbb{R}^2)},
    \label{eq:V_channelflow}
\end{equation} 
where $C^\infty_{per \times c} (\Omega, \mathbb{R}^2)$ denotes the space of smooth functions over $\Omega$ with values in $\mathbb{R}^2$ that are periodic in the $e_1$ direction and have compact support with respect to $e_2$.

\subsection{Travelling wave formulation and discretization}
\label{sect:TWformulation}

The numerical setup introduced in this section is commonly employed in the applied fluids literature, see the work of Dijkstra et al.\ \cite{dijkstra_14} for a general overview.
As a preliminary step, the formulation will be altered slightly, allowing travelling waves to be captured as fixed points (rather than periodic orbits) -- in a Galilean frame.
Introducing $\tilde{x}_1 := x_1 - ct$ as the new frame of reference, the equations are reduced to
\begin{subequations}
\begin{gather}
     \left( \partial_t - c \partial_{\tilde{x}_1} \right) v - \frac{1}{Re} \Delta v + (v \cdot \nabla ) v + \nabla p = f \quad \text{in } \Omega, \\
    \nabla \cdot v = 0 \quad \text{in } \Omega, \\
    v = 0 \quad \text{on } \partial \Omega, \label{eq:Dirichlet} \\
    \frac{1}{|\Omega |}\int_\Omega v_1 = 1, \label{eq:nondimensional_flux} \\
    \langle v_1, \sin (k \tilde{x}_1) \rangle_{L^2(\Omega,\mathbb{R}^2)} = 0. \label{eq:poincaresect}
\end{gather}\label{eq:NSE_channel}\end{subequations}
Here, \eqref{eq:nondimensional_flux} - the dimensionless version of \eqref{eq:volflux} - is enforced through adjusting $f$; 
and \eqref{eq:poincaresect} fixes the Poincaré section (or phase).

Next, we introduce the discretization to be used throughout this work.
As is customary in the case of unidirectional wall-bounded flows, 
we project onto  
a finite dimensional subspace spanned by a truncated Fourier-Chebyshev basis, so that 
\begin{displaymath}
    (v_1,v_2,p)(\tilde{x}_1,x_2,t) \approx \sum_{n=-N_1}^{N_1} \sum_{m=0}^{N_2} \mathsf{v}_{nm}(t) \Phi_n(\tilde{x}_1)  T_m(x_2),
\end{displaymath}
where
$$
\Phi_n(\tilde{x}_1):= \sqrt{k/(2 \pi)} \, e^{\mathrm{i}nk\tilde{x}_1} \quad {\rm and} \quad
T_m(x_2) := \cos [m \cos^{-1}(x_2)]
$$
denote the Fourier and Chebyshev basis elements, respectively; and $\mathsf{v}_{nm} = (\mathsf{v}_1,\mathsf{v}_2,\mathsf{p})_{nm}: \mathbb{R} \to \mathbb{C}^3$, is a solution curve
with values in
the space of coefficients.
Since $v$ and $p$ are real in physical space, the negative indices have to satisfy $\mathsf{v}_{(-n)m} =\bar{\mathsf{v}}_{nm}$, and hence the equations need only be solved for half of the Fourier coefficients.

In search of fixed points of \eqref{eq:NSE_channel}, we require that $\mathsf{v}_{nm}$  remains constant, and arrive at
\begin{equation}
  \sum_{m}  \mathsf{A}_n ( \mathsf{v}_{nm} T_m ) - \sum_{j,m,q} \mathsf{B}_{n-q} (\mathsf{v}_{qm} T_m, \mathsf{v}_{(n-q)j} T_j) + \mathsf{f}_n = 0, \quad n = 0,\ldots,N_1,
  \label{eq:discFourier}
\end{equation}
where 
\begin{gather*}
   \mathsf{A}_n g = 
\begin{pmatrix}
c \mathrm{i}nk g_1 - \mathrm{i}nk g_3 +\frac{1}{Re} \left(  \partial_{x_2 x_2}-(nk)^2 \right) g_1    \\
c \mathrm{i}nk g_2 - \partial_{x_2} g_3 +\frac{1}{Re} \left(  \partial_{x_2 x_2}-(nk)^2 \right) g_2 \\
 \mathrm{i}nk g_1 + \partial_{x_2} g_2
\end{pmatrix}; 
\\\mathsf{B}_{n-q}(g,h) =
\begin{pmatrix}
 g_1 \mathrm{i}(n-q) h_1 + g_2 \partial_{x_2} h_1 \\
  g_1 \mathrm{i}(n-q) h_2 + g_2 \partial_{x_2} h_2 \\
  0
\end{pmatrix}, \qquad \mathsf{f}_n = \delta_{0n} \begin{pmatrix}
f \\
0
\end{pmatrix},
\end{gather*}
for functions $g,h: [-1,1] \to \mathbb{C}^3$ in $x_2$.
In \eqref{eq:discFourier}, we have projected system \eqref{eq:NSE_channel} back onto the first $N_1$ Fourier modes; to treat the $x_2$ direction, we use collocation points instead. 
Equation \eqref{eq:discFourier} is evaluated at the so-called Gauss-Lobatto points given by 
\begin{equation}
    x_{2,s} = \cos \left(\frac{s \pi}{N_2} \right) \in [-1,1], \qquad s = 0,\ldots,N_2,
    \label{eq:gauss_lobatto}
\end{equation}
which, along with the volume flux condition \eqref{eq:nondimensional_flux} and the phase condition \eqref{eq:poincaresect},
yields enough equations to solve for the coefficients $\mathsf{v}_{nm}$, $n = 0,\ldots,N_1$, $m = 0,\ldots,N_2$, $f$ and $c$.
The Dirichlet boundary condition~\eqref{eq:Dirichlet} is achieved by replacing the relevant rows of (the $x_{2,s}$-projected version of) \eqref{eq:discFourier} by equations of the form
\begin{equation}
    \sum_m \mathsf{v}_{nm} T_m (x_{2,0}) = 0, \qquad \text{and} \qquad \sum_m \mathsf{v}_{nm} T_m (x_{2,N_2}) = 0,
    \label{eq:bc_implement}
\end{equation}
with $n$ ranging over all Fourier modes.

Later, we will prefer to use the vectorized form of \eqref{eq:discFourier}, which we simply write as
\begin{equation}
    \mathsf{A} \mathsf{v} - \mathsf{B}(\mathsf{v},\mathsf{v})  = 0,
    \label{eq:TWeq_vectorized}
\end{equation}
obtained by stacking \eqref{eq:discFourier} on top of itself according to $n = 1,\ldots,N_1$ and subsequent evaluations at \eqref{eq:gauss_lobatto}.
Here, $\mathsf{v} := (\mathrm{vec}(\{ \mathsf{v}_{nm} \}_{n,m}),f,c)$;
$\mathsf{A}\in \mathrm{GL} (\mathbb{C}^N)$ and $\mathsf{B}:\mathbb{C}^N \times \mathbb{C}^N \to \mathbb{C}^N$ (with $N:= 3 N_1 N_2 + 2$ denoting the total dimension of the discretized system) map from coefficient space corresponding to the basis $\{\Phi_n(\tilde{x}_1) T_m (x_2)\}_{n,m}$ into the 'semi-physical' space \eqref{eq:TWeq_vectorized} lives in (evaluated at the collocation points \eqref{eq:gauss_lobatto} -- but still in terms of Fourier coefficients, corresponding to the 'basis'\footnote{Here, $\delta$ denotes the Dirac measure. This is  not a basis of $V$, its only purpose is to describe the 'semi-physical' space.} $\{\Phi_n(\tilde{x}_1) \delta_{x_2}(x_{2,m})\}_{n,m}$).
The additional conditions
\eqref{eq:nondimensional_flux}, \eqref{eq:poincaresect} and \eqref{eq:bc_implement} have all 
been absorbed into $\mathsf{A}$ along with the contribution of $f$, since all of these become linear within the discretized framework.

Equation \eqref{eq:TWeq_vectorized} is solved via standard Newton-Raphson iteration. 
The resulting solution, which we shall denote by $\mathsf{U} = (\mathsf{U}_1,\mathsf{U}_2)$ ($U = (U_1,U_2)$ in physical space),
corresponds to a wave travelling at phase speed $c$.
Figure~\ref{fig:TWs_U1_U2} shows two typical such solutions
$U_s$ and $U_u$
along stable/unstable parts of the 
branch emanating directly from the bifurcation of the laminar state (at $k = 0.85$).
The following section describes how surrounding invariant manifolds 
are computed.

 \begin{figure}
\includegraphics[width=1\textwidth]{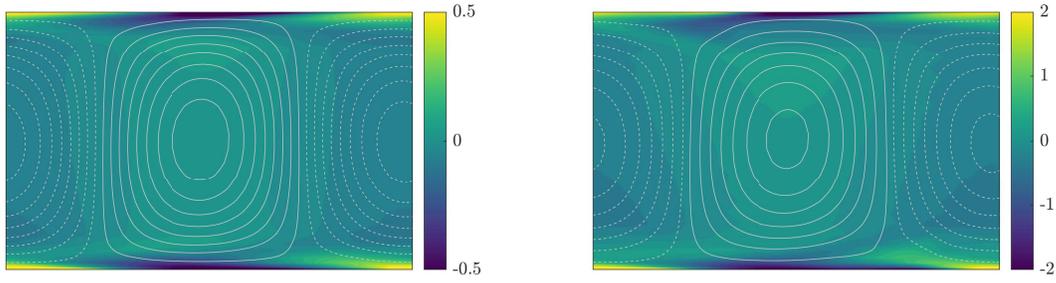}
\centering
\caption{Two travelling wave solutions $ \mathsf{U}_s \approx U_s $ (left) and $ \mathsf{U}_u \approx U_u$ (right) obtained upon solving \eqref{eq:TWeq_vectorized} at $Re = 5440$ and $Re = 5348$, with $k=0.85$ fixed, shown on the periodic domain $[0, 2 \pi / k] \times [-1,1]$. 
Contours show the vorticity; 
lines correspond to level sets of the perturbation stream function.
}
\label{fig:TWs_U1_U2}
\end{figure}

For the purpose of following branches of travelling waves as parameters are varied,
the Newton-Raphson procedure solving \eqref{eq:TWeq_vectorized} is coupled with a pseudo-arclength continuation routine --
Figures~\ref{fig:kcrit_branch}, \ref{fig:main} and~\ref{fig:OB_main} were obtained as such.

\subsection{Computation of the invariant manifold}

We consider the
equation for perturbations about the steady state $U$ obtained in the preceding section,
that reads
\begin{subequations}
\begin{gather}
      \partial_t u = c P_L \partial_{\tilde{x}_1} u + \mathcal{A}_U u -B(u,u) + P_L \tilde{f}, \\
    \langle u_1, \sin (k \tilde{x}_1) \rangle_{L^2(\Omega,\mathbb{R}^2)} = 0, 
\end{gather}\label{eq:NSE_channel_pert}\end{subequations}
with operators $\mathcal{A}_U$ and $B$ as defined in Section~\ref{sect:functional_setup}.
Solutions $u$ are sought in the space $V$ defined in \eqref{eq:V_channelflow}.

The corresponding discretized equation can be written in the form 
\begin{equation}
    \mathsf{M} \frac{d \mathsf{u}}{dt} = \mathsf{A}_\mathsf{U} \mathsf{u} -  \mathsf{B}(\mathsf{u},\mathsf{u}),
    \label{eq:pert_vectorized}
\end{equation}
where 
$\mathsf{u}$ is of the same form as $\mathsf{v}$ above, 
$\mathsf{M}$ is the composition of the map 
\begin{displaymath}
    (\mathrm{vec}(\{\mathsf{u}_1,\mathsf{u}_2,\mathsf{p})_{nm}\}_{n,m},f,c) \mapsto (\mathrm{vec} (\{ \sum_m(\mathsf{u}_1,\mathsf{u}_2,0)_{nm}T_m(x_{2,s})\}_{n,s}),0,0)
\end{displaymath}
with a map that sets the entries with $s=0,N_2$ to $0$ in order to ensure that the boundary conditions enforced in $\mathsf{A}_\mathsf{U}$ are retained;
$\mathsf{B}$ remains the same bilinear contribution as in \eqref{eq:TWeq_vectorized}; $\mathsf{A}_\mathsf{U}$ is the linearization equivalent to $\mathsf{A} - \mathsf{B}(\mathsf{U},\cdot) -\mathsf{B}(\cdot,\mathsf{U})$ (with $\mathsf{A}$ and $\mathsf{B}$ from \eqref{eq:TWeq_vectorized}).

The invariance equation (as developed in Section~\ref{sect:parameterization})
in the context of \eqref{eq:pert_vectorized} reads
\begin{equation}
   \mathsf{M} DK(\cdot) [R(\cdot)] = \mathsf{A}_\mathsf{U} K(\cdot) -  \mathsf{B}(K (\cdot ), K( \cdot))
   \label{eq:invariance_channel}
\end{equation}
for unknowns $K : \Theta \to \mathbb{C}^N$, the parameterization, and $R : \Theta \to \Theta$, the reduced dynamics.

In the examples to follow, we encounter 
two types of sets $\Sigma_1 \subset \sigma (\mathsf{A}_\mathsf{U})$ for which we wish to compute invariant manifolds: either $\Sigma_1$ is a simple real eigenvalue, or $\Sigma_1$ contains a finite number of complex eigenvalues. 
Typical spectra $\sigma (\mathsf{A}_\mathsf{U})$ within the travelling wave configuration are shown in Figure~\ref{fig:spectrum_U1_U2}, with $\Sigma_1$ comprising one real simple eigenvalue -- in this case, we would choose the model space $\Theta$ to be an interval of $\mathbb{R}$ containing $0$.
Otherwise, if $\Sigma_1$ consists of a few complex eigenvalues, we take for the model space $\Theta$ to be a hypercube in $\mathbb{C}^r$ containing $0$,
where $r = \mathrm{dim}_{\mathbb{C}}(\mathrm{im}(P_{\Sigma_1}))$,
 with coordinates $\theta = (\theta_1,\ldots,\theta_r)$.
 
 \begin{figure}
\includegraphics[width=1\textwidth]{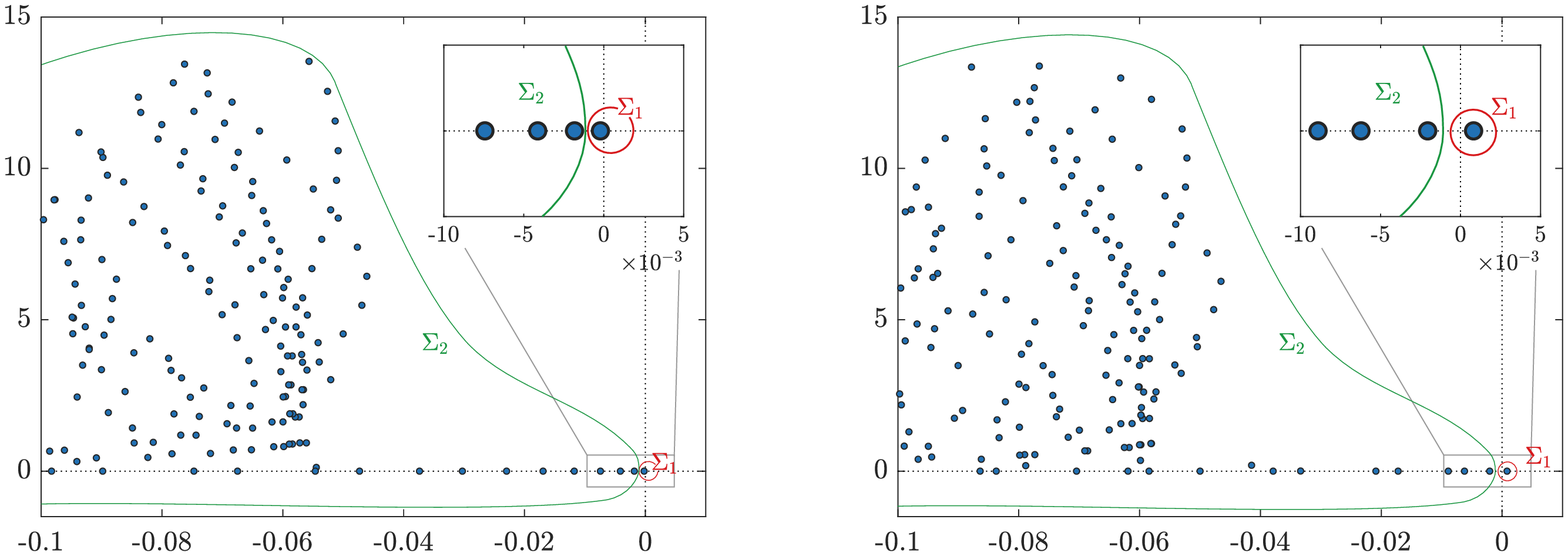}
\centering
\caption{Spectrum of the linearized operators $\mathsf{A}_{\mathsf{U}_s}$ (left) and $\mathsf{A}_{\mathsf{U}_u}$ (right) about travelling wave solutions $\mathsf{U}_s$ and $\mathsf{U}_u$ (as shown in Figure~\ref{fig:TWs_U1_U2}), partitioned according to \eqref{eq:sugma}. 
}
\label{fig:spectrum_U1_U2}
\end{figure}
 
We seek the unknowns in the form of power series expansions
\begin{displaymath}
    K(\theta) = \sum_{j \geq 1} K_j(\theta); \qquad R(\theta) = \sum_{j \geq 1} R_j(\theta),
\end{displaymath}
with $K = (K^1,\ldots,K^N)$ and $R = (R^1,\ldots,R^r)$. Their components are given by
\begin{displaymath}
   K_j^i (\theta)= \sum_{ | \alpha | = j} K^i_{j,\alpha} \theta^\alpha, \qquad i= 1,\ldots,N,
\end{displaymath}
and
\begin{displaymath}
   R_j^i (\theta)= \sum_{ | \alpha | = j} R^i_{j,\alpha} \theta^\alpha, \qquad i= 1,\ldots,r,
\end{displaymath}
where $\alpha \in \mathbb{N}^r$ denotes a multi-index such that $| \alpha | = \alpha_1 + \cdots + \alpha_r$ and $\theta^\alpha = \theta_1^{\alpha_1}\cdots \theta_r^{\alpha_r}$;
the coefficients are real or complex numbers depending on the setup.

At the first order ($j = 1$), \eqref{eq:invariance_channel} returns the eigenvalue problem
\begin{displaymath}
    \mathsf{M} K_1 \circ R_1 = \mathsf{A}_\mathsf{U} K_1,
\end{displaymath}
as discussed below \eqref{eq:firstorder}.
We choose $R_1 \in \mathbb{C}^{r \times r}$ to be the Jordan normal form of $\mathsf{A}_\mathsf{U}P_{\Sigma_1}$ and $K_1$ to be the the $\mathbb{C}^{N \times r}$ matrix composed of eigenvectors (and generalized eigenvectors) corresponding to $\Sigma_1$. 
The $(i,k)-$th entry of these matrices defines the coefficient $R^i_{1,e_k}$ (and $K^i_{1,e_k}$), where $e_k = (0,\ldots,1,\ldots,0)$ with $1$ in the $k-$th entry.

At higher orders ($j \geq 2$), \eqref{eq:invariance_channel} yields an algebraic set equations
 for the coefficients $K^i_{j,\alpha}$ and $R^q_{j,\alpha}$
\begin{equation}
   \mathsf{A}_{\mathsf{U},i}^\ell K^i_{j,\alpha} - \mathsf{M}^\ell_i  K^i_{j,\beta^k}  \alpha_q R^q_{1,e_k} - \mathsf{M}^\ell_i K^i_{1,e_q} R^q_{j,\alpha} = \eta_{j,\alpha}^\ell, \qquad |\alpha| = j,  
   \label{eq:invariance_index}
\end{equation}
with $\beta^k = \alpha - e_k$\footnote{If $\beta^k$ has a negative entry, we set $K^i_{j,\beta^k} = 0$.} and $\eta_j$ as in \eqref{eq:higherorders}:
\begin{equation}
   \eta_{j,\alpha}^\ell  = \left( \sum_{r = 1}^{j-1} \mathsf{B}(K_{j-r}(\cdot),K_{r}(\cdot)) + \mathsf{M}  \sum_{k = 2}^{j-1} D K_{j-k+1}(\cdot)[R_k(\cdot)] \right)_\alpha^\ell.
   \label{eq:eta_j_alpha}
\end{equation}
Note that if $R_1$ is diagonalizable, i.e., of the form
\begin{displaymath}
   \begin{pmatrix}
        \lambda_1 &  &\\
       & \ddots & \\
       & & \lambda_r 
   \end{pmatrix},
\end{displaymath}
then \eqref{eq:invariance_index} reduces to
\begin{equation}
 \left( \mathsf{A}_{\mathsf{U},i}^\ell  - \alpha_q \lambda_q \mathsf{M}^\ell_i \right) K^i_{j,\alpha}   - \mathsf{M}^\ell_i K^i_{1,e_q} R^q_{j,\alpha} = \eta_{j,\alpha}^\ell, \qquad |\alpha| = j,  
  \label{eq:invariance_index_diagable}
\end{equation}
which is block-diagonal in the coefficients $K^i_{j,\alpha}$,
and hence, when solved via a numerical algorithm, can be parallelized over different choices of $\alpha$.

With the numerical procedure, we follow the approach of \cite{shobhit2022}.
In particular, we avoid further diagonalizing the complementary part of $\mathsf{A}_{\mathsf{U}}$, as full-fledged eigenvalue computations are often times the bottleneck of large scale problems (and fluid flows certainly fall into this category).
To compute $R^q_{j,\alpha}$ in the \textit{normal form style} of parameterization, 
\cite{shobhit2022} derive a solvability condition
in the vein of the Fredholm alternative.
Namely, \eqref{eq:invariance_index_diagable}\footnote{If $R_1$ is non-diagonalizable, one can do the same procedure to the full system, rather than the index-by-index approach shown here.} is solvable if and only if
\begin{displaymath}
  \eta_{j,\alpha} + \mathsf{M}_i K^i_{1,e_q} R^q_{j,\alpha} \in  \mathrm{im}\left( \mathsf{A}_{\mathsf{U}}-\alpha_q \lambda_q \mathsf{M} \right) = \left(\mathrm{ker}\left( \mathsf{A}_{\mathsf{U}}-\alpha_q \lambda_q \mathsf{M} \right)^* \right)^\perp
\end{displaymath}
for all $|\alpha| = j$, i.e., if for each $\zeta \in \mathrm{ker}\left( \mathsf{A}_{\mathsf{U}}-\alpha_q \lambda_q \mathsf{M} \right)^*$,
\begin{equation}
 \left\langle \zeta,  (\eta_{j,\alpha} + \mathsf{M}_i K^i_{1,e_q} R^q_{j,\alpha}) \right\rangle_{\mathbb{C}^N}  = 0
 \label{eq:Fredholm}
\end{equation}
holds over all $|\alpha| = j$. 
In fact, it is sufficient to check \eqref{eq:Fredholm} over $\zeta \in \mathrm{ker}\left( \mathsf{A}_{\mathsf{U}}-\alpha_q \lambda_q \mathsf{M} \right)^* \cap \mathrm{im}(\mathsf{M})$, since $\eta_{j,\alpha} \in \mathrm{im}(\mathsf{M})$ by \eqref{eq:eta_j_alpha}.
Note that $\mathrm{ker}\left( \mathsf{A}_{\mathsf{U}}-\alpha_q \lambda_q \mathsf{M} \right)^*$ is nonempty if and only if $\alpha_q {\lambda}_q \in \sigma (\mathsf{A}_{\mathsf{U}}) = \Sigma_1 \cup \Sigma_2$. 

If $\alpha_q {\lambda}_q \in \Sigma_1$, we have a case of internal resonance in the sense of \eqref{eq:internal_resonance}.
In this case, for any $\zeta \in \mathrm{ker}\left( \mathsf{A}_{\mathsf{U}}-\alpha_q \lambda_q \mathsf{M} \right)^* \cap \mathrm{im}(\mathsf{M})$, there exists $k \in \{1,\ldots,r\}$ such that $\langle \zeta, \mathsf{M} K_{1,e_k} \rangle_{\mathbb{C}^N}$ is nonzero, 
and hence, after setting
\begin{displaymath}
   R_{j,\alpha}^k = -\frac{\zeta^\ell \eta_{j,\alpha}^\ell}{\zeta^\ell  \mathsf{M}_i^\ell K^i_{1,e_k}}, \qquad \text{and} \qquad R_{j,\alpha}^q = 0   \qquad \text{for} \; q \in \{1, \ldots , r \} \setminus \{ k\},
\end{displaymath}
 \eqref{eq:Fredholm} is satisfied.

If, on the other hand, $\alpha_q \lambda_q \in \Sigma_2$, we have a case of cross resonance in the sense of \eqref{eq:cross_resonance}, and the system \eqref{eq:invariance_index_diagable} is not solvable by any means, 
as $\langle \zeta, \mathsf{M} K_1 \rangle_{\mathbb{C}^N} = 0$ for any $\zeta \in \mathrm{ker}\left( \mathsf{A}_{\mathsf{U}}-\alpha_q \lambda_q \mathsf{M} \right)^* \cap \mathrm{im}(\mathsf{M})$. 

Provided that the outer nonresonance conditions \eqref{eq:cross_resonance} hold, the above procedure ensures solvability of \eqref{eq:invariance_index_diagable}, and all remaining coefficients $R^q_{j,\alpha}$ can be set to zero. 
The resulting reduced dynamics are hence as simple as possible, which marks the objective of the normal form style.
In the \textit{graph style}, the embedding $K$ is simplified instead by choosing
the $K_{j,\alpha}^i$ such that $\mathrm{im}(K_{j,\alpha}) \subset \left(\mathrm{im}(K_1) \right)^\perp$, and adjusting $R^q_{j,\alpha}$ such that
\begin{displaymath}
   - \mathsf{M}^\ell_i K^i_{1,e_q} R^q_{j,\alpha} = \eta_{j,\alpha}^\ell
\end{displaymath}
holds.

Within both styles of parameterization, the remainder of coefficients $K_{j,\alpha}^i$ are obtained upon solving the nontrivial part of the linear system \eqref{eq:invariance_index_diagable}.

\subsection{Invariant manifolds about travelling wave solutions}
\label{sect:example_main}

In this section, the above
routine is applied to the bifurcation branch problem initiated in Section~\ref{sect:example_lam}, albeit at a slightly different wavenumber $k = 0.85$ (held fixed).
For the discretization, we choose $N_2 = 30$ Chebyshev and $N_1 = 15$ Fourier modes.
We trace down the solution branch emanating from the laminar stability transition at $Re = 5396$, along which both $U_s$ and $U_u$ of Figure~\ref{fig:TWs_U1_U2} appear.
By computing invariant manifolds about these states, the aim is to construct 
heteroclinic connections between different segments of the branch.

The new choice of wavenumber is motivated by the slightly unusual, complex behaviour
the branch exhibits at $k=0.85$ (see Figure~\ref{fig:main} top-right panel, in black).
The initial Hopf bifurcation (of the laminar state) at $Re \approx 5396$ is supercritical, and hence, the initial segment of the travelling wave branch is stable until reaching a saddle-node (fold) bifurcation at $Re \approx 5588$.
From then on, the branch continues to evolve in the expected 
fashion -- with a lower branch of unstable saddles, an upper branch of stable nodes
and a fold bifurcation in-between, at $Re \approx 5312$.

\begin{figure}
\includegraphics[width=1\textwidth]{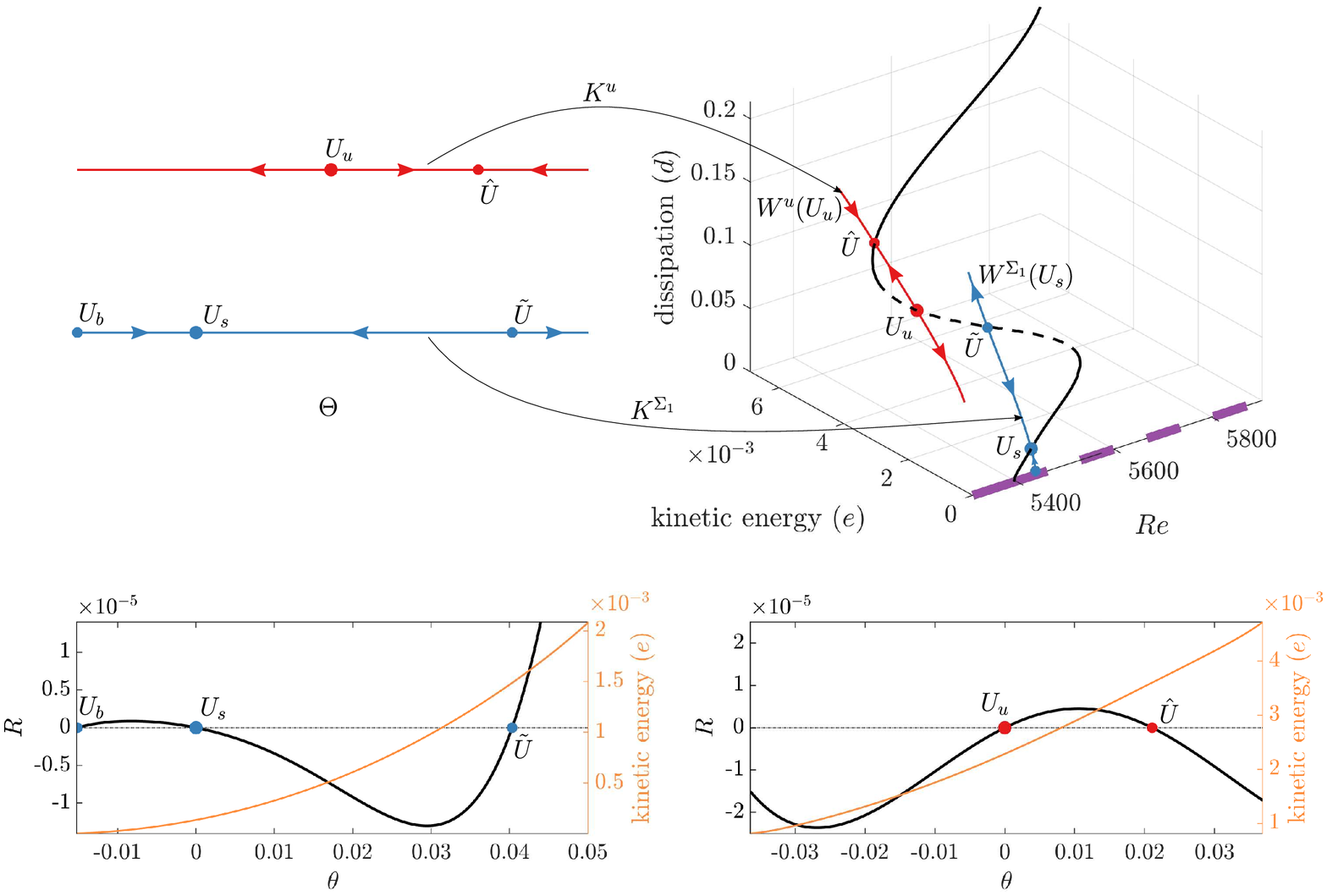}
\centering
\caption{Parameterization of invariant manifolds around the branch of travelling waves at $k = 0.85$. (Top-left) The spaces $\Theta$ on which the invariant manifolds are modeled. (Top-right) The manifolds $W^{\Sigma_1}(U_s) = \mathrm{im}(K^{\Sigma_1})$ and $W^u(U_u) = \mathrm{im}(K^{u})$ embedded in a two-dimensional phase space spanned by $e$ and $d$, along with the full branch of solutions.
Stability of states along the branch is indicated via solid (stable) and dashed (unstable) lines.
(Bottom) The reduced dynamics corresponding to $W^{\Sigma_1}(U_s)$ (left) and $W^u(U_u)$ (right).
}
\label{fig:main}
\end{figure}

 As measures of amplitude, we have chosen the perturbation kinetic energy \eqref{eq:pert_kin_e}
 and the square-root dissipation (motivated by \cite{kaszas2022})
 \begin{displaymath}
    d = \Vert u \Vert_V = \left( \int_{\Omega} | \nabla u|^2  \right)^{\frac{1}{2}}
 \end{displaymath}
 of the perturbation velocity $u = v - U_{\mathrm{lam}}$ above the laminar state.
 
At two points $U_s$ and $U_u$ along the stable/unstable parts of the branch, 
invariant manifolds $W^{\Sigma_1}(U_s)$ (slow-stable) and $W^u(U_u)$ (unstable) were computed via the graph style of parameterization up to order $7$
-- these are shown embedded in $e-d$ space alongside the branch in the top-right panel of Figure~\ref{fig:main}. 
The reduced, $1$-dimensional dynamics $R$ on the model space $\Theta = [-1,1]$ are shown on the left panel. 
The embedding $K$ transfers the dynamics over into the full phase space, and hence the exact dynamics are obtained along the invariant manifold in the vicinity of the travelling wave solution -- or, in some subcases, even up to the nearest exact coherent state.

Moreover, the obtained local dynamics about $U_s$ (resp.\ $U_u$)  
are representative of all nearby solutions,
since these are (locally) attracted towards $W^{\Sigma_1}(U_s)$ (resp.\ $W^{u}(U_u)$) along the remaining, stable directions at some rate $\mathcal{O}(e^{\beta t})$. 
Locally, around $U_s$ ($U_u$), these 'stable directions' are represented by the foliation of Theorem~\ref{thm:PO}, and
the decay rate is specified by Theorem~\ref{thm:PO}\ref{thm:PO_item2} as $\mathcal{O}(e^{\beta t})$ for any $\beta > \lambda_2$
with $\lambda_2$ the least stable eigenvalue of $\Sigma_2 \subset \sigma ( \mathsf{A}_{\mathsf{U}_s} ) $ (resp.\ $\Sigma_2 \subset \sigma ( \mathsf{A}_{\mathsf{U}_u} ) $).
In particular, unstable manifold $W_{u}(U_u)$ isolates 
 the optimal (fastest) way for the flow state to
grow towards more complex solutions (the upper branch)
or descend back towards the laminar state.
Its slow-stable counterpart at $U_s$, $W^{\Sigma_1}(U_s)$ isolates the slowest route for trajectories to approach $U_s$.

With $\tilde{U}$ and $\hat{U}$ as in Figure~\ref{fig:main},
we have that that $W^u(\tilde{U}) = W^{\Sigma_1}(U_s)$ on the arc $[U_s,\tilde{U}]$,
based on the discussion in Section~\ref{sect:example_lam} (also, $W^u(U_u) = W^{\Sigma_1}(\hat{U})$ on $[U_u,\hat{U}]$). 
In general, we observe that slow stable manifolds along stable parts of the branch and unstable manifolds along its unstable segments contain nearby exact coherent states.

Whether or not these
heteroclinic connections are actually observable along 
the computed manifolds (c.f.\ Figure~\ref{fig:main})
comes down to the accuracy of the numerical procedure at the distance of the nearby state.
A relevant concept is the \textit{fundamental domain} of the manifold -- the domain on which the power series expansions satisfy the invariance equation sufficiently well.
This can be quantified upon defining an error over $\theta$, see Section~2.2.5 of \cite{haro2016} for a few recommendations (assuming that full trajectories of the system can be computed).
Here, integration of full-fledged solutions would be unfeasible, and hence we are forced to resort to a simplified approach, and use
\begin{displaymath}
   \mathrm{err}(\theta) =  \left| \sum_{j \geq 2} \eta_j(\theta) \right|
\end{displaymath}
as the error norm, with $\eta_j$ as in \eqref{eq:higherorders}.
The fundamental domain $\Theta_f$ is defined as 
\begin{displaymath}
    \Theta_f : = \left\{ \theta \in \Theta \; \big| \; \mathrm{err}(\theta) < \mathrm{tol} \right\}
\end{displaymath}
for some user-specified tolerance $\mathrm{tol}$ (set to $1.5 \cdot 10^{-2}$ for the manifolds in Figure~\ref{fig:main}) that represents 'sufficient accuracy' for the system in question -- usually determined via trial and error\footnote{As the latter suggests, this approach is rather fiddly, and hence, whenever possible, we recommend using those outlined in \cite{haro2016} instead.}.
The manifolds displayed in Figure~\ref{fig:main} correspond to $K (\Theta_f)$, and hence may or may not contain nearby states depending on their distance from the fixed point (about which the manifold was computed) -- for instance, $W^u(U_u)$ does not contain the laminar state $U_b$.
Note that this is solely a feature of the numerical procedure, and not that of the full system. 

The fundamental domain can be enlarged to some degree (with $\mathrm{tol}$ fixed), by increasing the order up to which the expansions are computed, but the improvements will saturate as higher orders are reached.
From then on, the full extent of the manifolds could be recovered via numerical integration 
(e.g.\ $U_b$ could be reached upon forward integrating $W^u(U_u)$) 
-- however, we will not perform this procedure here due to its computational costs.

\subsection{Elastic instability of an Oldroyd-B fluid}
\label{sect:example_visco}

For the final example, we consider pressure-driven viscoelastic channel flow (in two spatial dimensions), motivated by the recent advancements in the field \cite{Miguel2022}. 
We assume that the evolution of the fluid flow is governed by the Oldroyd-B model\footnote{Note that we have not yet shown the existence of spectral submanifolds for the case of the Oldroyd-B model. We expect them to retain their existence due to the structure of \eqref{eq:OB}, but this example is more meant as an extrapolation towards possible future directions.}
\begin{subequations} \label{eq:OB}
\begin{gather}
    Re \Big[ \partial_t u + \left( u \cdot \nabla \right) u \Big]+ \nabla p = \beta \Delta u +  \nabla \cdot T + \begin{pmatrix}
    f \\ 0
    \end{pmatrix},  \\
    \nabla \cdot u =0, \\
    \partial_t T +  \left( u \cdot \nabla \right) T + \frac{1}{Wi} T = T \cdot \nabla u +   \left( \nabla u \right)^{T} \cdot T +\frac{(1-\beta)}{Wi} \left( (\nabla u)^T + \nabla u \right) +\varepsilon \Delta T. 
\end{gather}
\end{subequations}
Here $T $ is the ($2 \times 2$) tensor representing the polymer stress field and
$\beta:= \nu_s/\nu \in [0,1]$ denotes the viscosity ratio where $\nu_s$ and $\nu_p=\nu-\nu_s$ are the solvent and polymer contributions to the total kinematic viscosity $\nu$. 
The viscoelastic fluid is assumed to occupy the same (periodic) domain $\Omega = \left( 0, 2 \pi / k \right) \times \left( -1,1 \right)$ as in the prior examples.
The dimensionless form of \eqref{eq:OB} was achieved via the same route as in Section~\ref{sect:channel_setup}, upon implementing
the Reynolds number \eqref{eq:Reynolds} and the Weissenberg number
$$
Wi = \tau Q,
$$
where $\tau$ is the polymer relaxation time and $Q$ is the volume flux \eqref{eq:volflux_lamstate}.

Equation \eqref{eq:OB} is supplemented with non-slip boundary conditions on the velocity field.
To treat the stress tensor $T$, we set $\varepsilon = 0$ at the walls of the channel (i.e.\ $x_2 = \pm 1$).
In the streamwise ($x_1$) direction, periodic boundary conditions are imposed on both $u$ and $T$, as in Section~\ref{sect:channel_setup}.

The study of viscoelastic flows forms a very active field within the applied fluids literature,
which hence pose a multitude of novel challenges. 
In particular, there are two new linear instabilities that persist even in the inertialess ($Re = 0$) limit: the center-mode instability discovered by Garg et al.~\cite{Garg2018} and the wall-mode instability of Beneitez et al.~\cite{Miguel2022}.
The central question puzzling the field
is whether these instabilities could be dynamically connected
to nearby (in parameter space) turbulent states known as elastic \cite{Steinberg2021rev} and elasto-inertial turbulence \cite{Samanta2013EIT}.
While this question remains inconclusive
for the case of the center-mode instability, the wall-mode has already produced promising results.
In the work of Beneitez et al.~\cite{Miguel2022}, it is shown that this latter instability
eventually evolves into what is perceived to be elastic turbulence in plane-Couette flows, and similar results have been confirmed for the case of channel flows by the same authors \cite{Miguel2023}.

If this is indeed the case, the exact coherent states
determined by the branch
emanating from the laminar bifurcation 
become of particular interest, as they would serve to
explain
the (laminar-turbulent) transition process.
The initial segment of this branch, however, behaves in a very stiff manner as parameters are varied, and hence standard techniques to initialize the branch continuation (e.g., weakly nonlinear analysis) become almost impossible to use.

\begin{figure}
\includegraphics[width=0.5\textwidth]{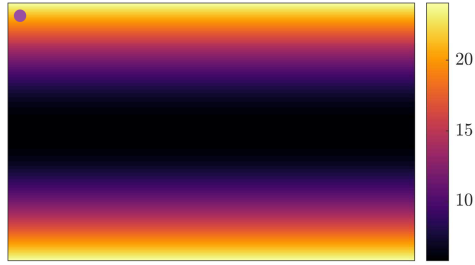}
\centering
\caption[Laminar state of Oldroyd-B fluid]{The laminar state $\Psi_b$ of an Oldroyd-B fluid at $(Wi,\beta,\varepsilon,Re) = (13.6,0.9,10^{-2},0)$. 
Colored contours correspond to $\mathrm{tr} \, T_b $. 
The base velocity field $U_b$ remains almost equal to that of Figure~\ref{fig:lam}, with only minor adjustments due to the $\varepsilon \Delta T$ term.
}
\label{fig:OB_laminar}
\end{figure}

Here, we use spectral submanifolds (as in Section~\ref{sect:example_lam}) as an alternative route of initialization, which does not even rely on the existence of a nearby laminar bifurcation.
While doing so, we also explore an elementary transition process between two segments of the branch connected to the wall-mode instability of \cite{Miguel2022}.
Let us fix $(\beta,\varepsilon,Re) = (0.9,10^{-2},0)$, a fairly standard set of parameters to be considered in the inertialess limit.
For the discretization, we choose $N_2 = 50$ Chebyshev and $N_1 = 10$ Fourier modes.
The laminar state $\Psi_b = (U_b,T_b)$ 
(Figure~\ref{fig:OB_laminar})
bifurcates at $Wi \approx 13.5$
for a critical streamwise wavenumber of $k = 2.3$.
The spectrum of the linearized operator about $\Psi_b$ is shown in Figure~\ref{fig:OB_spect}.

\begin{figure}
\includegraphics[width=0.5\textwidth]{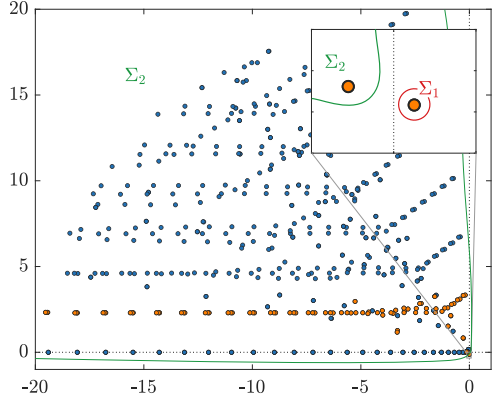}
\centering
\caption[Spectrum of Oldroyd-B fluid]{Spectrum of the differential operator corresponding to \eqref{eq:OB} linearized about the base state $\Psi_b$ shown on the complex plane, partitioned according to \eqref{eq:sugma}. 
As in Figure~\ref{fig:spectrum_lam}, the spectrum obtained via the standard, single Fourier mode eigenvalue problem is highlighted in orange.
}
\label{fig:OB_spect}
\end{figure}

We compute a parameterization $K$ for the unstable manifold of $\Psi_b$ slightly beyond the bifurcation, at $Wi = 13.6$, to sixth order (Figure~\ref{fig:OB_last}). 
The resulting reduced model on $\Theta$ reads
\begin{subequations}
\begin{align}
    \dot{r} &= 10^{-3} \left( 0.249r -0.160r^3 +0.0147r^5 \right),\\
    \dot{\vartheta} &= 0.119 + 0.0005 r^2 +0.000006 r^4.
\end{align}\label{eq:OB_red_dyn}\end{subequations}
The two nontrivial invariant sets of the reduced dynamics \eqref{eq:OB_red_dyn} are  $ \Gamma_1(t) =\{(r,\vartheta(t)) \; \vert \; r = r_1 \}$ and  $ \Gamma_2(t) =\{(r,\vartheta(t)) \; \vert \; r = r_2 \}$, with $r_1 \approx 1.37$ and $r_2 \approx 3.01$.
The corresponding limit cycles of the full system are obtained as $\Psi_1 = K(\Gamma_1)$ (stable) and $\Psi_2 = K(\Gamma_2)$ (unstable), shown via blue and red (respectively) in Figure~\ref{fig:OB_last}.
In place of the usual measure of kinetic energy, we shall use
\begin{equation}
    \mathcal{T} := \frac{\int_\Omega \mathrm{tr} \, T}{\int_\Omega \mathrm{tr} \, T_b}
    \label{eq:OB_T}
\end{equation}
instead to visualize the embedding (Figure~\ref{fig:OB_last}), representing the relative elastic energy of the fluid state.

\begin{figure}
\includegraphics[width=1\textwidth]{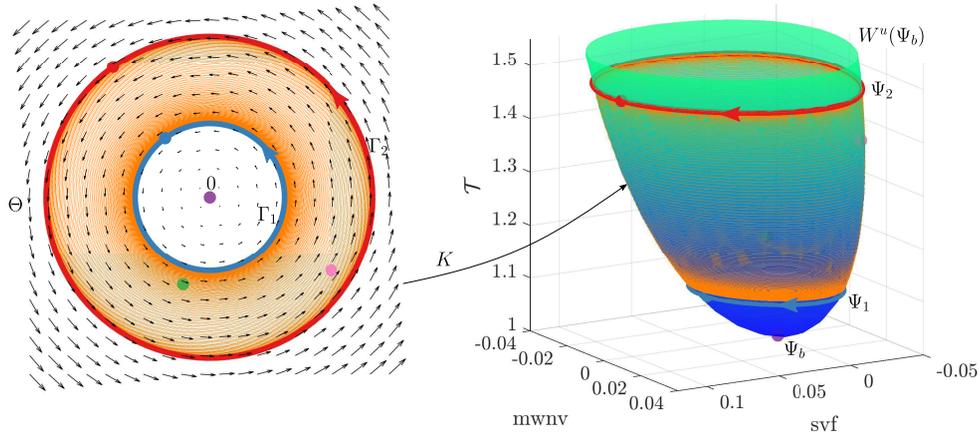}
\centering
\caption[Unstable manifold of Oldroyd-B fluid]{Parameterization of the unstable manifold $W^u(\Psi_b) = W^{\Sigma_1}(\Psi_b)$ about the laminar state. 
(Left) The model space $\Theta$ along with its reduced dynamics described by $R$. (Right) The manifold $W^u(\Psi_b)= \mathrm{im}(K)$ embedded in a three-dimensional phase space spanned by $\mathcal{T}$ \eqref{eq:OB_T}, $\mathrm{mwnv}$ \eqref{eq:mwnv} and $\mathrm{svf}$ \eqref{eq:svf}.
In this example, we choose $\hat{x}_2$ in the definition of $\mathrm{svf}$ to be the $10-$th Gauss-Lobatto point $x_{2,10}$, as it shows far larger fluctuations than the centerline velocity.
A single trajectory traversing between the two limit cycles $\Gamma_1$ and $\Gamma_2$ (resp.\ $\Psi_1$ and $\Psi_2$) is shown in orange on both panels.
}
\label{fig:OB_last}
\end{figure}
 
Next, we trace down the evolution of a single trajectory 
describing the transition between the two limit cycles.
For this, we take an initial condition on $\Theta$ just inside $\Gamma_2$, with initial radius $r_1 \ll r_0 <r_2$ and initial angle $\vartheta_0$, then integrate \eqref{eq:OB_red_dyn} to obtain its orbit $\mathcal{O}^+(\{r_0,\vartheta_0\})$.
The resulting trajectory $\mathcal{O}^+(\{r_0,\vartheta_0\})$ and its counterpart  $K(\mathcal{O}^+(\{r_0,\vartheta_0\}))$ in the full system are highlighted in orange on the left and right panels of Figure~\ref{fig:OB_last} respectively.
Snapshots of its evolution, taken at four points denoted via distinctly colored dots in Figure~\ref{fig:OB_last}, are shown in Figure~\ref{fig:OB_evolution}.

\begin{figure}
\includegraphics[width=1\textwidth]{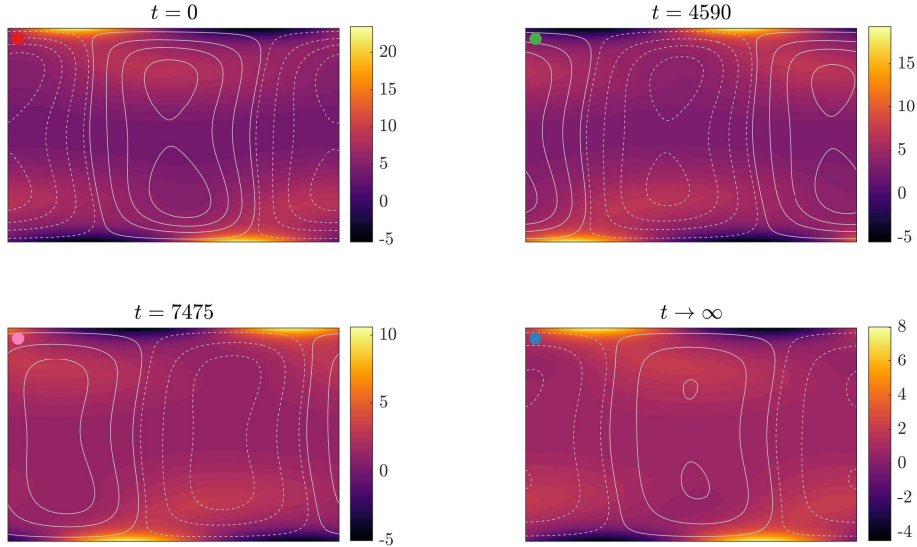}
\centering
\caption[Evolution of Oldroyd-B fluid]{Snapshots along the evolution of the orange trajectory of Figure~\ref{fig:OB_last}. 
Colored contours show $\mathrm{tr} \, (T-T_b)$, the trace of the perturbation polymer stress tensor. 
Lines correspond to level sets of the perturbation stream function.
}
\label{fig:OB_evolution}
\end{figure}

With the nontrivial states in hand, we are now in a position to initiate the branch continuation routine from $Wi = 13.6$.
After fixing a Poincaré section and imposing condition \eqref{eq:poincaresect} onto the states $\Psi_1$ or $\Psi_2$, we reach the travelling wave formulation of Section~\ref{sect:TWformulation}.
Figure~\ref{fig:OB_main} shows the branch launched from $\Psi_1$ in the direction of increasing $Wi$. 
The shape of the branch confirms the initial supercriticality of the Hopf bifurcation of the laminar state at $Wi \approx 13.5$. 
A saddle-node bifurcation follows shortly thereafter, resulting in an unstable branch of edge states reaching to significantly lower $Wi$.
Note that, due to its adequate convergence (see caption of Figure~\ref{fig:OB_main}), we could have initiated the branch continuation routine directly from $\Psi_2$ as well to reach the same conclusions.

\begin{figure}
\includegraphics[width=0.5\textwidth]{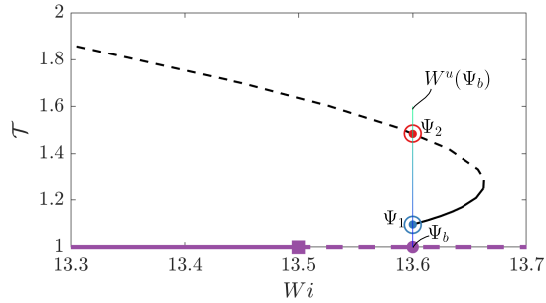}
\centering
\caption[Bifurcation branch of Oldroyd-B fluid]{
Bifurcation diagram in the close vicinity of 
where the wall-mode instability of \cite{Miguel2022} emerges.
The point at which the laminar state undergoes a (supercritical Hopf) bifurcation 
is denoted by a purple square symbol ($Wi \approx 13.5$ for a critical wavenumber of $k=2.3$).
Circles are centered at the states predicted on the unstable manifold $W^u(\Psi_b)$ (also shown via the same colors in Figure~\ref{fig:OB_last}); full dots denote the corresponding converged states via the Newton-Raphson iteration of the continuation routine.
}
\label{fig:OB_main}
\end{figure}

\section{Discussion}
\label{sect:discussion}

In this work, we have 
shown the existence of a generalized class of invariant manifolds (termed spectral submanifolds) about fixed points and periodic orbits 
in the phase space $V$ 
of the Navier-Stokes equations.
Notably, for any $\beta \in \mathbb{R}$ partitioning the spectrum into two disjoint components $\Sigma_1$ (containing elements with real parts larger than $\beta$) 
and $\Sigma_2$ (resp.\ smaller), 
we have a local nonlinear continuation of the spectral subspace $P_{\Sigma_1}V$ 
invariant under the semiflow, and a foliation identifying 
synchronized trajectories
that attract each other at a rate of $O(e^{\beta t})$ (Theorem~\ref{thm:summary}). 
Moreover, we have classified the subscenarios in which 
uniqueness and a higher degree of smoothness of these structures is to be expected. 

Alongside this paper, a computational tool was developed to accommodate spectral submanifold computations for the case of the Navier-Stokes and Oldroyd-B systems.
The algorithm obtains said manifolds via power series expansions along the lines of the parameterization method \cite{Cabre2003a,haro2016},
built on the well-established and by now properly optimized code of \cite{shobhit2021code,shobhit2022} from the engineering mechanics literature. 
Our adapted version is available at \url{https://github.com/gergelybuza/SSM_fluids}, written in Matlab.

Through the course of a few examples within the framework of channel flows,
we have demonstrated some 
avenues 
in which spectral submanifolds could serve as useful tools aiding our understanding 
of the dynamical structure of phase space.
The range of applications 
is by no means exhausted, and we anticipate 
that further directions will be explored in the near future.
Even if we merely utilize spectral submanifolds for the computational benefits they provide as reduced order models,
the above examples already show promising qualities.
Indeed, all computations shown herein were done in a matter of minutes, with peak memory consumption kept below $32$GB.
The latter only occurs during the assembly of the bilinear term (which could be stored as it is independent of the problem at hand), otherwise the memory usage remained below $16$GB, suitable for use on an average laptop. 
We acknowledge, nonetheless, that there is still a lot to be improved.

The bottleneck, 
for the most part, is not computational speed, but rather the immense memory 
requirements posed by storing all necessary polynomial coefficients. 
There are a few recently-emerged techniques
that attempt to mitigate these costs 
via reducing the number of coefficients needed to be stored.
For instance, an alternative form of discretization is encouraged in \cite{Gonzalez2022paramfem} for parabolic PDEs, the finite element method.
While this seems advantageous for most applications due to the sparsity of coefficient matrices, finite elements generally require significantly higher resolutions when compared to Chebyshev-Fourier methods in the specific context of fluid flows (especially wall-bounded flows), which would likely diminish the advantages of sparsity.
Alternatively, differing routes to solving the invariance equation could serve advantageous, as demonstrated by the works of Roberts et al.~\cite{roberts2014dynamical,roberts2014book} and Vizzaccaro and Opreni et al.~\cite{vizzaccaro2022high,opreni2023high}.
Another, perhaps more enticing route is to employ data-driven techniques. 
These have proven to be incredibly efficient  in engineering mechanics \cite{Cenedese2022mech,Cenedese2022}, and now in plane-Couette flows as well \cite{kaszas2022}.
In particular, the authors in \cite{kaszas2022} were able to tackle systems even larger than those considered here, 
facilitating a possible transition to the realm of industrial-scale applications.

\appendix
\section{Proof of Proposition~\ref{prop:semiflow_summary}} 
\label{appendix}
In this appendix, we show the existence and smoothness of a local semiflow generated by the (perturbation) Navier-Stokes system \eqref{eq:main} on $V$.
We reiterate that these results are not novel and the majority of them are already given in the works of Weissler~\cite{WEISSLER1979,Weissler1980TheNI}.

We first recall a local existence result
which in essence is given in \cite{WEISSLER1979} (Theorem 2.1; see also \cite{Weissler1980TheNI}). 
In anticipation of the smoothness considerations in the latter half of this appendix, we shall state this result in a more general setting, allowing for non-autonomous generators.

Non-autonomous semiflows are defined similarly to \eqref{eq:semiflow}, with the dependence on the initial time $t_0 \geq 0$ included, i.e., \eqref{eq:semiflow1} and \eqref{eq:semiflow2} are replaced with
\begin{align*}
    \varphi_{t_0}^{t_0} &= I, \\ 
    \varphi_s^t \circ \varphi^s_{t_0} &= \varphi_{t_0}^t, \qquad \forall t \geq s \geq t_0.
\end{align*}

\begin{proposition} \label{prop:exist}
Let $\mathcal{A}_U$ be as in \eqref{eq:AU}
and let $J_t: V \to L_\sigma^p$ be a locally Lipschitz map with Lipschitz constants locally uniform with respect to $t$, for $t \in [0,t^*)$ ($t^*$ possibly infinite) and $\frac{2d}{d+2} < p < \infty$.
Assume further that $(t,v) \mapsto J_t(v)$ is continuous 
and that $J_t (0) = 0$.
Then there exists a unique local semiflow $\varphi_{t_0}^t$ on $V$ over the interval $[0,t^*)$ defined by
\begin{equation}
    \varphi_{t_0}^t (u_0) = e^{(t-t_0) \mathcal{A}_{U} } u_0  + \int_{t_0}^t e^{(t-s)\mathcal{A}_{U} } J_s \circ \varphi_{t_0}^s (u_0) \, ds, \qquad 0 \leq t_0 < t < t^*,
    \label{eq:prop11main}
\end{equation}
such that the following hold:
\begin{enumerate}[label=\upshape{(\roman*)}]
    \item For each $T>0$ and $r>0$, there exists a $\delta >0$ such that $\mathrm{Lip}_V \big(\varphi_{t_0}^t \, | \, \overline{B_r(0)} \big) \leq C_{r,T}$ holds uniformly over $t_0 \in [0,T]$ and $t \in [t_0,t_0+\delta]$. \label{item:prop11_1}
    \item Fix $u_0 \in V$. If $\sup \{ t \, | \, u_0 \in \mathrm{dom}(\varphi_{t_0}^t)  \} = T < \infty$, then $\lim_{t \to T} \Vert \varphi_{t_0}^t (u_0) \Vert_V = \infty$. \label{item:prop11_2}
    \item $\mathrm{dom} (\varphi_{t_0}^t)$ is open in $V$. Moreover, if $u_0 \in \mathrm{dom} (\varphi_{t_0}^T)$, there is a neighborhood $O$ of $u_0$ such that $O \subset \mathrm{dom} (\varphi_{t_0}^T)$ and the $\varphi_{t_0}^t$, for $t \in [t_0,T]$, are uniformly Lipschitz on $O$. \label{item:prop11_3}
\end{enumerate}
\end{proposition}

\begin{proof}
Under these assumptions the map
\begin{displaymath}
    t \mapsto \int_{t_0}^t e^{(t-s)\mathcal{A}_{U} } J_s (u(s)) \, ds
\end{displaymath}
is well defined and continuous for functions $u$ that are continuous.
Existence/uniqueness follows from a standard contraction mapping argument. Fix $u_0 \in V$, and pick $r_0 > 0$ such that $u_0 \in \overline{B_{r_0}(0)} \subset V$; $r > C_0 r_0$ with $C_0$ as in \eqref{eq:eAU0}. Define
\begin{displaymath}
  X := \left\{ \, u \in C^0 \big([t_0,T],\overline{B_r(0)} \big) \; \big\vert \; u(t_0) = u_0 \, \right\},
\end{displaymath}
a closed subset of $C^0([t_0,T],V)$ and hence a complete metric space with the inherited (standard) metric.
Define $\mathcal{F}:X \to C^0([t_0,T],V)$ by
\begin{displaymath}
  (\mathcal{F}u) (t) := e^{(t-t_0) \mathcal{A}_{U} } u_0  + \int_{t_0}^t e^{(t-s)\mathcal{A}_{U} } J_s (u(s)) \, ds, \qquad t \in [t_0,T].
\end{displaymath}
With $p$ as in the hypothesis of the theorem, we may choose $\alpha < 1$ in accordance with \eqref{eq:alpharange}.
Then, using \eqref{eq:commonplace_bound} and the assumption that $J_t(0) = 0$, we have that
\begin{align}
\sup_{t \in [t_0,T]} \Vert \mathcal{F}u (t) \Vert_V &\leq C_0 \max \{ e^{\omega (T-t_0)},1 \} r_0 \nonumber \\
&+   C_\alpha \frac{(T-t_0)^{1-\alpha}}{1-\alpha} \max \{ e^{\omega (T-t_0)},1 \} \sup_{t \in [t_0,T]} \mathrm{Lip}_{V \to L^p_\sigma} \big( J_t \, | \, \overline{B_r(0)} \big) r. \label{eq:FU_Vbound}
\end{align}
The assumptions ensure that $0<T<t^*$ may be chosen such that the right hand side of \eqref{eq:FU_Vbound} exists and is bounded by $r$, i.e.\ $\mathcal{F} X \subset X$.
Similarly, 
\begin{multline*}
  \sup_{t \in [t_0,T]} \Vert \mathcal{F}u (t) -\mathcal{F}v (t) \Vert_V \leq \\ C_\alpha \frac{(T-t_0)^{1-\alpha}}{1-\alpha} \max \{ e^{\omega (T-t_0)},1 \} \sup_{t \in [t_0,T]} \mathrm{Lip}_{V \to L^p_\sigma} \big( J_t \, | \, \overline{B_r(0)} \big) \sup_{t \in [t_0,T]} \Vert u (t) -v (t) \Vert_V,
\end{multline*}
so the same choice of $T$ guarantees that $\mathcal{F}$ is a contraction.
Thus, there exists a unique fixed point of $\mathcal{F}$ in $X$, which,
if denoted by $\varphi^t_{t_0} (u_0)$, constitutes our local semiflow on $V$.

Statement \ref{item:prop11_1} follows from an argument entirely analogous to Lemma~\ref{lemma:lipsic2}; its proof is therefore omitted.

Statement \ref{item:prop11_2} follows from standard ODE arguments.
Suppose, conversely, that  $$\sup_{t \in [t_0,T)} \Vert \varphi_{t_0}^t (u_0) \Vert_V < \infty.$$ Then, using the integral equation, $\lim_{t \to T} \varphi^t_{t_0} (u_0)$ exists and is bounded (by continuity). 
Repeating the argument above, we may now find a solution on $[T,T+ \varepsilon]$ with initial condition $\varphi^T_{t_0}(u_0)$, for some $\varepsilon > 0$ depending solely on $\Vert \varphi^T_{t_0}(u_0) \Vert_V$.
Since $\varphi_{T}^{T+\varepsilon} \circ \varphi^T_{t_0}  = \varphi_{t_0}^{T+\varepsilon}$, this contradicts the maximality of $T$. 

To see \ref{item:prop11_3}, let $u_0 \in \mathrm{dom}(\varphi_{t_0}^T)$ and choose $r> 0$ such that $\Vert \varphi^t_{t_0}(u_0) \Vert_V \leq r/2$ for $t \in [t_0,T]$.
Statement \ref{item:prop11_1} gives $\delta > 0$ such that $\mathrm{Lip}_V \big(\varphi_{t_0}^t \, | \, \overline{B_r(0)} \big) \leq C_{r,T}$ on $t \in [t_0,t_0+\delta]$, with $t_0 \in [0,T]$.
Choose $k \in \mathbb{N}$ such that $k\delta > T$. Then the set
\begin{displaymath}
    O : = \left\{ \, v \in V \; \vert \; C_{r,T}^k \Vert u_0-v \Vert_V < r/2 \, \right\}
\end{displaymath}
has the required properties with $\mathrm{Lip}_V \big(\varphi_{t_0}^t \, | \, O \big) \leq C_{r,T}^k$ for $t \in [t_0,T]$.
\end{proof}

\begin{remark} \label{remark:globalex}
If $J_t$ in Proposition~\ref{prop:exist} is globally Lipschitz for all $t \geq 0$, then combining statement \ref{item:prop11_2} with standard ODE arguments
yields a globally defined semiflow.
\end{remark}

Next we show that under the additional hypothesis that $J_t \in C^\infty(V,L^p_\sigma)$, we obtain smoothness for the semiflow itself.
If we formally differentiate the equation defining $\varphi_{t_0}^t$ \eqref{eq:prop11main}, we get
\begin{equation}
    D \varphi_{t_0}^t (u_0) [v_0] = e^{(t-t_0) \mathcal{A}_{U} } v_0  + \int_{t_0}^t e^{(t-s)\mathcal{A}_{U} } D J_s (\varphi_{t_0}^s (u_0) ) \circ D \varphi_{t_0}^s (u_0) [v_0] \, ds,
    \label{eq:Dvarphi}
\end{equation}
for some $v_0 \in V$. 
The next proposition confirms that a semiflow satisfying \eqref{eq:Dvarphi} exists, and moreover, that it coincides with the Fréchet derivative of $\varphi_{t_0}^t$ at $u_0$.

\begin{proposition}[Theorem 2.2, \cite{WEISSLER1979}]
\label{prop:smooth}
Let $\mathcal{A}_U$ and $J_t$ be as in Proposition~\ref{prop:exist} and let $\varphi^t_{t_0}$ denote the resulting semiflow.
Suppose further that $J_t$ is continuously Fréchet differentiable infinitely many times, i.e.\ $J_t \in C^{\infty} (V,L^p_\sigma)$ for all $t$, with derivatives depending continuously on $t$.
Then each $\varphi^t_{t_0}$ is $C^{\infty}$ on its domain, with derivatives depending continuously on $t$.
\end{proposition}

\begin{proof}
Fix $u_0 \in V$ and take $K_t = D J_t ( \varphi^t_{t_0}(u_0)) \in \mathcal{L}(V,L^p_\sigma)$, defined for $t \in [t_0, T_{u_0})$, where $T_{u_0}$ is the maximal time of existence associated to $u_0$.
Then,
Proposition~\ref{prop:exist} applied with $K_t$ gives a family of linear maps $\theta_{t_0}^t$ that are defined by
\begin{displaymath}
  \theta_{u_0;t_0}^t (v_0) = e^{(t-t_0) \mathcal{A}_{U} } v_0  + \int_{t_0}^t e^{(t-s)\mathcal{A}_{U} } D J_s (\varphi_{t_0}^s (u_0) ) [ \theta_{u_0;t_0}^s (v_0) ]\, ds
\end{displaymath}
for all times in the original interval $t \in [t_0,T_{u_0})$, by linearity of $K_t$.

Let $T \in (t_0,T_{u_0})$ and choose $r>0$ such that $ \sup_{t \in [t_0,T]} \Vert \varphi^t_{t_0} (u_0) \Vert_V \leq r$. 
Pick $\alpha<1$ according to \eqref{eq:alpharange} again and denote by 
\begin{equation}
    k_r(t_0,t) : = C_\alpha \frac{(t-t_0)^{1-\alpha}}{1-\alpha} \max \{ e^{\omega (t-t_0)},1 \} \sup_{s \in [t_0,t]} \mathrm{Lip}_{V \to L^p_\sigma} \big( J_s \, | \, \overline{B_r(0)} \big),
    \label{eq:kr}
\end{equation}
so that
\begin{equation}
    \left\Vert \int_{t_0}^t e^{(t-s)\mathcal{A}_{U} } D J_s (\varphi_{t_0}^s (u_0) ) ds \right\Vert_{\mathcal{L}(V)}  \leq k_r (t_0,t) \qquad \text{for } t \in [t_0,T]
    \label{eq:krbound}
\end{equation}
with the above choice of $r$. 
Now choose $\tau > t_0$ such that $k_{r}(t_0,\tau) \leq k_{2r}(t_0,\tau) < 1$. 

We next show Gateaux differentiability.
Using Proposition~\ref{prop:exist}\ref{item:prop11_3}, we have a neighborhood $O$ of $u_0$ in $V$ such that $O \subset \mathrm{dom}(\varphi_{t_0}^\tau)$ and $\sup_{t\in [t_0,\tau]}\mathrm{Lip}_V \big( \varphi^t_{t_0} \, | \, O \big) \leq C$.
We claim that $\theta_{u_0;t_0}^t = D^G \varphi_{t_0}^t (u_0)$ for $t \in [t_0,\tau]$, i.e.\ that
\begin{displaymath}
    \Delta_h (t) := \frac{\varphi^t_{t_0}(u_0+h v_0)-\varphi^t_{t_0}(u_0)}{h} - \theta_{u_0;t_0}^t(v_0)
\end{displaymath}
(with $h>0$ small enough such that $u_0 + h v_0 \in O$) converges to $0$ as $h \to 0$ for any $v_0 \in V$.

Writing out $\theta_{u_0;t_0}^t$ and $\varphi^t_{t_0}$ using the integral equations that define them, we get
\begin{equation}
    \Delta_h (t) = \int_{t_0}^t e^{(t-s) \mathcal{A}_U} \mu_h(s) ds +\int_{t_0}^t e^{(t-s) \mathcal{A}_U} DJ_s ( \varphi^s_{t_0} (u_0) ) [\Delta_h (s)] ds,
    \label{eq:deltah}
\end{equation}
where
\begin{multline}
    h \mu_h (s) = J_s \circ \varphi^s_{t_0}(u_0 + h v_0) - J_s \circ \varphi^s_{t_0}(u_0) - DJ_s (\varphi^s_{t_0} (u_0)) [\varphi^s_{t_0}(u_0 + h v_0) - \varphi^s_{t_0}(u_0)]  \\
    = \int_0^1  \Big( DJ_s (\varphi^s_{t_0} (u_0) + \lambda (\varphi^s_{t_0}(u_0 + h v_0) - \varphi^s_{t_0}(u_0)))  - DJ_s (\varphi^s_{t_0} (u_0)) \Big) [\varphi^s_{t_0}(u_0 + h v_0) - \varphi^s_{t_0}(u_0)] d \lambda.
    \label{eq:mu_h}
\end{multline}
Taking the supremum of \eqref{eq:deltah} over $t \in [t_0,\tau]$ and using \eqref{eq:krbound}, we obtain
\begin{displaymath}
    \sup_{t \in [t_0,\tau]} \Vert \Delta_h (t) \Vert_V \leq (1-k_r(t_0,\tau))^{-1} \sup_{t \in [t_0,\tau]} \left\Vert \int_{t_0}^t e^{(t-s) \mathcal{A}_U} \mu_h (s) ds \right\Vert_V.
\end{displaymath}
That the right hand side converges to $0$ as $h \to 0$ follows from the latter representation of $\mu_h$ in \eqref{eq:mu_h} and the dominated convergence theorem.

To conclude that $\varphi_{t_0}^t$ is continuously Fréchet differentiable at $u_0$ (for $t \in [t_0,\tau]$), it suffices to check that $u \mapsto \theta_{u; t_0}^t$ is norm continuous (as a map $V \to \mathcal{L}(V)$) on some open neighbourhood $O_2$ of $u_0$.
Using Proposition~\ref{prop:exist}\ref{item:prop11_3}, take $O_2 \subset O$ such that $\sup_{t \in [t_0,\tau]} \Vert \varphi_{t_0}^t (u) \Vert_V \leq 2r$ for all $u \in O_2$.
Now let $u_n \to u_0$ in $O_2$, and consider 
\begin{align*}
    \theta_{u_n;t_0}^t (v_0) - \theta_{u_0;t_0}^t (v_0) &=
    \int_{t_0}^t e^{(t-s) \mathcal{A}_U} DJ_s (\varphi^s_{t_0}(u_n)) [\theta_{u_n;t_0}^s (v_0) - \theta_{u_0;t_0}^s (v_0)] ds \\
    &+ \int_{t_0}^t e^{(t-s) \mathcal{A}_U} \left( DJ_s (\varphi^s_{t_0}(u_n)) [\theta_{u_0;t_0}^s (v_0)] - DJ_s (\varphi^s_{t_0}(u_0)) [\theta_{u_0;t_0}^s (v_0)] \right) ds
\end{align*}
for any $v_0 \in V$.
Repeating the procedure above, we obtain
\begin{multline*}
    (1-k_{2r}(t_0,\tau))  \sup_{t \in [t_0,\tau]} \Vert \theta_{u_n;t_0}^t (v_0) - \theta_{u_0;t_0}^t (v_0) \Vert_V \leq \\ \sup_{t \in [t_0,\tau]} \left\Vert \int_{t_0}^t e^{(t-s) \mathcal{A}_U} \left( DJ_s (\varphi^s_{t_0}(u_n)) [\theta_{u_0;t_0}^s (v_0)] - DJ_s (\varphi^s_{t_0}(u_0)) [\theta_{u_0;t_0}^s (v_0)] \right) ds \right\Vert_V,
\end{multline*}
where again, the right hand side converges to $0$ by the dominated convergence theorem, as $n \to \infty$.

So far, we have proven that $\varphi_{t_0}^t$ is continuously Fréchet differentiable in a neighbourhood of $u_0$ for $t \in [t_0,\tau]$.
To extend the result to the full time of existence, choose a sequence $\tau_n$ inductively such that $k_{2r}(\tau_{n-1},\tau_n) < 1$.
By the local (in time) uniformity assumption on the Lipschitz constant of $J_t$, the sequence $\tau_n$ may be chosen such that 
$\tau_N \geq T$ for some finite $N$.
Up to possibly shrinking $\tau_N$, we may write 
$\varphi_{t_0}^T = \varphi_{\tau_{N-1}}^{\tau_N} \circ \cdots \circ \varphi_{t_0}^{\tau_1}$ and therefore
\begin{displaymath}
    D \varphi_{t_0}^T (u_0) = D \varphi_{\tau_{N-1}}^{\tau_N} ( \varphi_{t_0}^{\tau_{N-1}}(u_0)) \circ \cdots \circ D \varphi_{t_0}^{\tau_1} (u_0),
\end{displaymath}
where the right hand side is continuous in $u_0$.
Since $u_0$ was arbitrary, the result follows for all of $V$.

The same argument applies to the semiflow given by
\begin{displaymath}
    (u_0,v_0) \mapsto (\varphi_{t_0}^t (u_0), D \varphi_{t_0}^t(u_0) [v_0] )
\end{displaymath}
on $V \times V$.
Therefore, the above procedure can be iterated to yield $C^\infty$ smoothness of the semiflow $\varphi_{t_0}^t$.
\end{proof}

\begin{corollary}
\label{cor:jointsmoothness}
In the setting of Proposition~\ref{prop:smooth},
suppose that $J_t$ depends on a set of $m$ parameters $\mu \in O \subset \mathbb{R}^m$
such that $J_t : V \times O \to L^p_\sigma,$  $(u,\mu) \mapsto J_t(u,\mu)$ is jointly smooth.
Then the resulting semiflow $\varphi_{t_0}^t$ determined by 
\begin{displaymath}
       \varphi_{t_0}^t (u_0,\mu) = e^{(t-t_0) \mathcal{A}_{U} } u_0  + \int_{t_0}^t e^{(t-s)\mathcal{A}_{U} } J_s ( \varphi_{t_0}^s (u_0,\mu),\mu) \, ds
\end{displaymath}
is also jointly smooth on $V \times O$, for each $t$ where defined.
If $J:(t,u,\mu) \mapsto J_t(u,\mu)$ is jointly $C^\infty$, then $\varphi_{t_0}$ is jointly $C^\infty$ on $(\mathcal{D} \cap ( (t_0,\infty) \times V) ) \times O$.
\end{corollary}

\begin{proof}
The first assertion follows from the  same argument as  in Proposition~\ref{prop:smooth}.

For the second part, we employ the same trick as \cite{henry1981}, Corollary~3.4.6.
For any $k> 0$, define 
\begin{displaymath}
   \tilde{\varphi}_0^\tau (u_0, \tilde{\mu} ) : = e^{ \tau \frac{1}{k}\mathcal{A}_{U} } u_0  + \int_0^\tau e^{(\tau-s) \frac{1}{k} \mathcal{A}_{U} } \tilde{J}_s (\tilde{\varphi}_{0}^s (u_0, \tilde{\mu}),\tilde{\mu}) \, ds
\end{displaymath}
with $\tilde{\mu} = (\mu,k,t_0)$ and $\tilde{J}_s (u,\tilde{\mu}) : = \frac{1}{k} J_{t_0 + \frac{s}{k}}(u,\mu)$.
We have that
\begin{displaymath}
   \varphi_{t_0}^t( \, \cdot, \mu ) = \tilde{\varphi}_0^{k(t-t_0)} ( \, \cdot,\tilde{\mu})
\end{displaymath}
by the uniqueness property of Proposition~\ref{prop:exist}.
If $t> t_0$, we may take $k:= (t-t_0)^{-1}$ and obtain
\begin{displaymath}
   \varphi_{t_0}^t( \, \cdot, \mu ) = \tilde{\varphi}_0^{1} ( \, \cdot,(\mu,(t-t_0)^{-1},t_0)),
\end{displaymath}
the latter of which is jointly $C^\infty$ over $(\mu,t,t_0)$ for $t > t_0$ by the first assertion.
\end{proof}

The content of Lemma~\ref{lemma:F} was to verify that there exists $p$ for which $F:V \to L^p_\sigma$ satisfies the assumptions of Propositions~\ref{prop:exist} and \ref{prop:smooth} ($p \in (1,2)$ for $d = 2$, $p \in (6/5,3/2]$ for $d= 3$), so that there exists a local (autonomous) semiflow $\varphi$ defined on an open subset of $\mathbb{R}^{\geq 0} \times V$, such that $\varphi_t \in C^{\infty}(V)$ for each $t$, where defined.
Moreover, by Corollary~\ref{cor:jointsmoothness}, $\varphi:(t,u_0) \mapsto \varphi_t(u_0)$ is jointly $C^\infty$ over $\mathcal{D} \cap (\mathbb{R}^{>0} \times V)$.

Trajectories of the above semiflow are solutions of the integral equation
\begin{equation}
    \varphi_t (u_0) = e^{t \mathcal{A}_{U} } u_0  + \int_{0}^t e^{(t-s)\mathcal{A}_{U} } F \circ \varphi_s (u_0) \, ds,
    \label{eq:integralform}
\end{equation}
and their relation to the original equation \eqref{eq:main} is apriori unknown.
Results of Kato and Fujita~\cite{Kato1962} lifted to a somewhat more general setting \cite{WEISSLER1980} (permitting the use of operator $\mathcal{A}_U$ in place of $\mathcal{A}_0$)
yield the following.
If $u_0 \in V$, then the solution constructed via the integral equation satisfies the original equation \eqref{eq:main} as an equation on $H$ for $t \in (0,T_{u_0})$. 
In particular, $\varphi_t (u_0) \in \mathrm{dom}_2 (\mathcal{A}_U)$ for $t \in (0,T_{u_0})$.
For $u_0 \in \mathrm{dom}_2 (\mathcal{A}_U)$, $\varphi_t(u_0)$ is differentiable in time at $t = 0$ and solves \eqref{eq:main} over $t \in [0,T_{u_0})$ in $H$.

In the two-dimensional case ($d = 2$), global existence of solutions is well known (see, for instance, \cite{robinson2001}) and therefore $\varphi:\mathcal{D} \to V$ extends to define a global semiflow on $\mathbb{R}^{\geq 0} \times V $.  

If equation \eqref{eq:main_cutoff} is considered instead, we obtain a global semiflow $\varphi^{\rho}:\mathbb{R}^{\geq 0} \times V \to V$ in both the two- and three-dimensional cases, with $\varphi^\rho \in C^{\infty}(V)$ for $t \geq 0$ (jointly smooth for $t>0$),
upon applying the preceding propositions
with $J_t \equiv F_\rho$, considering that $F_\rho:V \to L^p_\sigma$ is globally Lipschitz (see Lemma~\ref{lemma:F} and Remark~\ref{remark:globalex}).

\section*{Acknowledgments}
The author thanks Rich Kerswell and Edriss Titi for the useful discussions.
The author gratefully acknowledges the support of the Harding Foundation through a PhD scholarship (https://www.hardingscholars.fund.cam.ac.uk).

\bibliographystyle{alpha}
\bibliography{ex_article}

\end{document}